\documentclass[11pt]{amsart}
\usepackage{amsfonts}
\usepackage{amsmath,amssymb,latexsym,soul,cite,mathrsfs}
\usepackage{mathtools}
\usepackage{color,enumitem,graphicx}
\usepackage[colorlinks=true,urlcolor=blue,
citecolor=red,linkcolor=blue,linktocpage,pdfpagelabels,
bookmarksnumbered,bookmarksopen]{hyperref}
\usepackage[english]{babel}

\usepackage[left=2.9cm,right=2.9cm,top=2.8cm,bottom=2.8cm]{geometry}
\usepackage[hyperpageref]{backref}

\usepackage[colorinlistoftodos]{todonotes}
\makeatletter
\providecommand\@dotsep{5}
\def\listtodoname{List of Todos}
\def\listoftodos{\@starttoc{tdo}\listtodoname}
\makeatother

\numberwithin{equation}{section}
\makeatletter
\@namedef{subjclassname@2020}{%
	\textup{2020} Mathematics Subject Classification}
\makeatother

\newtheorem{theorem}{Theorem}[section]
\newtheorem{proposition}[theorem]{Proposition}
\newtheorem{lemma}[theorem]{Lemma}
\newtheorem{corollary}[theorem]{Corollary}

\begin{document}

	%	\title {nonsmooth version of fountain theorem}
	
	\title [Existence of solution for two classes of quasilinear systems...]{Existence of solution for two classes of quasilinear systems defined on a non-reflexive Orlicz-Sobolev Spaces}

	\author{Lucas da Silva}
	\author{Marco Souto}

	\address[Lucas da Silva]{\newline\indent Unidade Acad\^emica de Matem\'atica
		\newline\indent{Universidade Federal de Campina Grande,}
		\newline\indent
		58429-970, Campina Grande - PB - Brazil}
	\email{\href{mailto:ls3@academico.ufpb.br}{ls3@academico.ufpb.br}}
	
	\address[Marco. A. S. Souto]
	{\newline\indent Unidade Acad\^emica de Matem\'atica
		\newline\indent 
		Universidade Federal de Campina Grande,
		\newline\indent
		58429-970, Campina Grande - PB - Brazil}
	\email{\href{marco@dme.ufcg.edu.br}{marco@dme.ufcg.edu.br}}

	\pretolerance10000

	\begin{abstract}
\noindent This paper proves the existence of nontrivial solution for two classes of quasilinear systems of the type
\begin{equation*}
\left\{\;
\begin{aligned}
-\Delta_{\Phi_{1}} u&=F_u(x,u,v)+\lambda R_u(x,u,v)\;\text{ in } \Omega& \\
-\Delta_{\Phi_{2}} v&=-F_v(x,u,v)-\lambda R_v(x,u,v)\;\text{ in } \Omega& \\
u=v&=0\;\text{ on } \partial\Omega&
\end{aligned}
\right.
\end{equation*}
where $\lambda > 0$ is a parameter, $\Omega$ is a bounded domain in $\mathbb{R}^N$($N \geq 2$) with smooth boundary $\partial \Omega$. The first class we drop the $\Delta_2$-condition of the functions $\tilde{\Phi}_i$($i=1,2$) and assume that $F$ has a double criticality. For this class, we use a linking theorem without the Palais-Smale condition for locally Lipschitz functionals combined with a concentration–compactness lemma for nonreflexive Orlicz-Sobolev space. The second class, we relax the $\Delta_2$-condition of the functions ${\Phi}_i$($i=1,2$). For this class, we consider $F=0$ and $\lambda=1$ and obtain the proof based on a saddle-point theorem of Rabinowitz without the Palais-Smale condition for functionals Fréchet differentiable combined with some properties of the weak$^*$ topology.

	\end{abstract}

	\thanks{{Lucas da Silva was supported by FAPESQ/Brazil.}}
	\thanks{{Marco Souto was partially supported by CNPq/Brazil 309.692/2020-2}}
	\subjclass[2020]{{{Primary: 35A15, 35B33, 35J62; Secondary: 46E30}}} 
	\keywords{{Orlicz-Sobolev spaces; Variational methods; Linking theorem; Saddle-point theorem; Critical growth; $\Delta_{2}-$condition}}

	\maketitle
	
	%------------------------------------------------------------------------------
	\section{Introduction}

In the present paper, we consider the existence of nontrivial solution for a large class of quasilinear systems of the type
\begin{equation}\label{S}
\left\{\;
\begin{aligned}
-\Delta_{\Phi_{1}}  u&=F_u(x,u,v)+\lambda R_u(x,u,v)\;\text{ in } \Omega& \\
-\Delta_{\Phi_{2}} v&=-F_v(x,u,v)-\lambda R_v(x,u,v)\;\text{ in }  \Omega& \\
u=v&=0\;\text{ on } \partial\Omega&
\end{aligned}
\right.
\end{equation}
where $\lambda > 0$ is a parameter, $\Omega$ is a bounded domain in $\mathbb{R}^N$($N \geq 2$) with smooth boundary $\partial \Omega$, $F:\overline{\Omega}\times\mathbb{R}^{2}\rightarrow \mathbb{R}$ and $R:\overline{\Omega}\times\mathbb{R}^{2}\rightarrow \mathbb{R}$ are continuous function verifying some conditions which will be mentioned later. It is important to recall that 
\begin{align*}
\Delta_{\Phi_i}u=div(\phi_i(|\nabla u|)\nabla u)
\end{align*}
where $\Phi_i(i=1,2):\mathbb{R}\rightarrow\mathbb{R}$ is a $N$-function of the form
\begin{align}\label{00.8}
\Phi_i(t)=\int_{0}^{|t|}s\phi_i(s)ds,
\end{align} 
and $\phi_i:(0,\infty)\rightarrow(0,\infty)$ is a $C^1$ function verifying some technical assumptions. 

Set $\Phi_{1} =:\Phi$, $v= 0$, $F_v(x, t, 0) =0$ and $R_v(x, t, 0) = 0$, for all $t\in \mathbb{R}$. Then the system \eqref{S} reduces to the following quasilinear elliptic equation:
\begin{equation}\label{P}
\left\{\;
\begin{aligned}
-\Delta_{\Phi}  u&=f(x,u)+ \lambda r(x,u)\;\text{ in } \Omega& \\
u&=0\;\text{ on } \partial\Omega&
\end{aligned}
\right.
\end{equation}
where $\Omega$ is a domain in $\mathbb{R}^N$, $f(x,u)=F_u(x,u,0)$ and $r(x,u)=R_v(x,u,0)$. 
%
%
%Set $\Phi_{2} = \Phi_{1} =:\Phi$, $u= -v$, $F(x, u, v) = F(x, -v, -u)$ and $R(x, u, v) = R(x, -v, -u)$. Then the system \eqref{S} reduces to the following quasilinear elliptic equation:
%\begin{equation}\label{P}
%\left\{\;
%\begin{aligned}
%-\Delta_{\Phi}  u&=f(x,u)+ \lambda r(x,u)\;\text{ in } \Omega& \\
%u&=0\;\text{ on } \partial\Omega&
%\end{aligned}
%\right.
%\end{equation}
%where $\Omega$ is a domain in $\mathbb{R}^N$, $r(x, u) = R_u(x, u, -u)$ and $f(x, u) = F_u(x, u, -u)$.

The equations like \eqref{P} have been arousing great interest among scholars. We refer readers to [\citenum{alves1},\citenum{FN},\citenum{Bonanno1},\citen{Bonanno2},\citenum{Santos},\citenum{Marcelo},\citenum{Mihailescu2},\citenum{Mihailescu3},\citenum{Donald}] and reference therein for more information. In all of these works the so called $\Delta_{2}$-condition has been assumed on $\Phi$ and $\tilde{\Phi }$, which ensures that the Orlicz-Sobolev space $W^{1,\Phi}_0(\Omega)$ is a reflexive Banach space. This assertion is used several times in order to get a nontrivial solution for elliptic problems taking into account the weak topology and the classical variational methods to $C^1$ functionals. 

In recent years many researchers have been relaxing the $\Delta_{2}$-condition of the functions $\Phi$ and $\tilde{\Phi }$ to study the equation \eqref{P}. From a mathematical point of view the problem becomes more subtle, because in general the functional energy associated with these problems are in general only continuous and the classical variational methods to $C^1$ cannot be used. Also, the weak topology cannot be taken into account, since $W^{1,\Phi}_0(\Omega)$ may not be reflective. For example, in \cite{Garcia}, García-Huidobro et al. have considered existence of solution for the nonlinear eigenvalue problem like \eqref{P} where $r$ is a continuous function verifying some other technical conditions. In the first part of that paper the authors consider the function $\Phi=(e^{t^2}-1)/2$. 

More recently, Fukagai et al, in \cite{FN1}, studied the equation \eqref{P} assuming that the $N$-function $\tilde{\Phi }$ may not verify the $\Delta_2$-condition. In that paper, the authors assumed that $f$ has a critical Sobolev growth, $r$ is a subcritical term and $\lambda$ is a real parameter. Using variational arguments along with the second concentration–compactness lemma of P. L. Lions for nonreflexive Orlicz-Sobolev space, they showed that there exists a constant $\lambda_0 > 0$ such that the boundary-value problem \eqref{P} has a nonnegative nontrivial solution in $W^{1,\Phi}_0(\Omega)$ for any $\lambda>\lambda_0$.

Already in \cite{silvagon}, Silva, Gonçalves and Silva, considered existence of multiple solutions for a class of problem like \eqref{P}. In that paper the $\Delta_2$-condition is not also assumed and the main tool used was the truncation of the nonlinearity together with a minimization procedure for the functional energy associated to the quasilinear elliptic problem \eqref{P}. In another paper, Silva, Carvalho, Silva and Gonçalves, in \cite{Edcarlos}, study a class of problem \eqref{P} where the energy functional satisfies the mountain pass geometry and the $N$-function $\tilde{\Phi } $ does not satisfy the $\Delta_2$-condition and has a polynomial growth.

Set $\phi_{1}(t)=|t |^p-2 $, $\phi_{2}(t)=|t|^q-2$ ($p,q>1$). Then system \eqref{S} reduces to the following $(p, q)$-Laplacian system: 
\begin{equation}\label{S'}
\left\{\;
\begin{aligned}
-\Delta_{p}  u&=F_u(x,u,v)+\lambda R_u(x,u,v)\;\text{ in } \Omega& \\
-\Delta_{q} v&=-F_v(x,u,v)-\lambda R_v(x,u,v)\;\text{ in }  \Omega& \\
u=v&=0\;\text{ on } \partial\Omega&
\end{aligned}
\right.
\end{equation}
For the case where $p = q = 2$, this class of systems is called noncooperative and in recent decades many recent studies have focused on it. For example, in \cite{Figueiredo}, Ding and Figueiredo consider the noncooperative system \eqref{S'} with $\lambda=1$ allowing that the function $F(x, u, v)$ can assume a supercritical and subcritical growth on $v$ and $u$ respectively. They established the existence of infinitely many solutions to \eqref{S'} provided the nonlinear terms $F$ and $R$ are even in $(u, v)$. Already in \cite{Clapp}, Clapp, Ding and Hernández showed that multiple existence of solutions to the noncooperative system \eqref{S'} with some supercritical growth can be established without the symmetry assumption. Motivated by some results found in \cite{Clapp} and \cite{Figueiredo}, Alves and Monari in \cite{alves2} studied the existence of nontrivial solutions for \eqref{S'} when $p$ and $q$ are different from 2, $\lambda=1$ and $F(x, u, v)$ has a supercritical growth on variable $v$ and has a critical growth at infinity on variable $u$ of the type $|u|^{p^*}$ with $p^* =pN/(N - p)$, the critical exponent of the embedding $W^{1,p}_0(\Omega) \hookrightarrow L^{p^*}(\Omega)$. The main difficulty in this case case is the lack of compactness of the functional energy associated to system. To overcome this difficulty, they carefully estimate and prove through the concentration–compactness principle due to Lions \cite{Lions1} the existence of a Palais-Smale sequence that has a strongly convergent subsequence. 

In a brief bibliographical research, we can mention some contributions devoted to the study of system where $\Phi_1$ and $\Phi_2$ are less trivial functions, as can be seen in [\citenum{Huentutripay},\citenum{wang}]. We would like to highlight the paper \cite{wang}, Wang et al. considered the following quasilinear elliptic system in Orlicz-Sobolev spaces:
\begin{equation}\label{S2}
\left\{\;
\begin{aligned}
-\Delta_{\Phi_{1}}  u&=R_u(x,u,v)\;\text{ in } \Omega& \\
-\Delta_{\Phi_{2}} v&= R_v(x,u,v)\;\text{ in }  \Omega& \\
u=v&=0\;\text{ on } \partial\Omega&
\end{aligned}
\right.
\end{equation}
where $\Omega$ is a bounded domain in $R^N$($N \geq 2$) with smooth boundary $\partial\Omega$. In that paper when $F$ satisfies some appropriate conditions including $(\Phi_1,\Phi_{2})$-superlinear and subcritical growth conditions at infinity as well as symmetric condition, by using the mountain pass theorem and the symmetric mountain pass theorem, they obtained that system \eqref{S2} has a nontrivial weak solution and infinitely many weak solutions, respectively. Some of the results obtained extend and improve those corresponding results in Carvalho et al \cite{Marcelo}.

In \cite{Huentutripay}, Huentutripay-Manásevich studied an eigenvalue problem to the following system:
\begin{equation*}\label{S3}
\left\{\;
\begin{aligned}
-\Delta_{\Phi_{1}}  u&=\lambda R_u(x,u,v)\;\text{ in } \Omega& \\
-\Delta_{\Phi_{2}} v&=\lambda R_v(x,u,v)\;\text{ in }  \Omega& \\
u=v&=0\;\text{ on } \partial\Omega&
\end{aligned}
\right.
\end{equation*}
where the function $R$ has the form
$$R(x,u,v)=A_1(x, u) + b(x)\Gamma_1(u)\Gamma_2(v) + A_2(x, v).$$
They extended the results of \cite{Garcia} for this system, that is, for a certain $\lambda$, translated the existence of a solution into an adequate minimization problem and proved the existence of a solution under some reasonable restriction. It is obvious that in \cite{Huentutripay}, the Orlicz-Sobolev spaces need not be reflexive.

Inspired by the mentioned research works cited above and sharpened by the known difficulty of working with nonreflexive Banach spaces, we intend to consider two new classes of problem \eqref{S} where $W ^{ 1,\Phi}_0 (\Omega)$ can be nonreflexive. 

The section 3 of this paper is dedicated to the study of the \eqref{S} system where the functions $\tilde{\Phi }_1$ and $\tilde{\Phi }_2$ may not verify the $\Delta_{2}$-condition. Inspired by \cite{alves2}, we assume that $F:\overline{\Omega}\times\mathbb{R}^{2}\rightarrow \mathbb{R}$ has a supercritical growth on variable $v$
and has a critical growth at infinity on variable $u$ and $R:\overline{\Omega}\times\mathbb{R}^{2}\rightarrow \mathbb{R}$ is a continuous function with subcritical term verifying some conditions which will be mentioned in section 3. The first difficulty in studying this case arises from the lack of differentiability of the energy functional $J_\lambda:W ^{1,\Phi_1}_0 (\Omega)\times W ^{1,\Phi_2}_0 (\Omega)\rightarrow\mathbb{R}$ associated with the system \eqref{S} given by
$$
J_\lambda(u,v)=\int_{\Omega} \Phi_1(|\nabla u|)dx-\int_{\Omega}\Phi_2(|\nabla v|)dx-\int_{\Omega}F(x,u,v)dx-\lambda\int_{\Omega}R(x,u,v)dx.
$$
To get around this difficulty we will use the critical point theory for locally lipschitz fuctionals, here, in particular we apply a version of the linking theorem without Palais-Smale condition for locally lipschitz fuctionals (The version that will be applied in section 3 we took care to enunciate in section 2). A second difficulty of studying this case is the lack of compactness of the energy functional $J_\lambda$. To overcome this difficulty, we adapted some arguments presented in the works of Alves and Soares in \cite{alves2} and from Fukagai et al, in \cite{FN1}. Here, we carefully estimate and prove through the second concentration–compactness lemma of P. L. Lions for nonreflexive Orlicz-Sobolev space that there exists a constant $\lambda_0 > 0$ such that the boundary-value problem \eqref{P} has a nonnegative nontrivial solution in $W^{1,\Phi}_0(\Omega)$ for any $\lambda>\lambda_0$.

The section 4 of this paper is dedicated to the study of the \eqref{S} system where the functions $\Phi_1$ and $\Phi _2$ may not verify the $\Delta_{2}$-condition. Here, we consider $F=0$, $\lambda=1$ and $R$ a continuous function verifying some conditions which will be mentioned in section 4. As stated earlier, the first difficulty in relaxing the $\Delta_{2}$ -condition of the functions $\Phi_1$ and $\Phi _2$ arises from the fact that the energy functional $J:W ^{1,\Phi_1}_0 (\Omega)\times W ^{1,\Phi_2}_0 ( \Omega)\rightarrow\mathbb{R}$ associated with the system \eqref{S} given by
$$
J(u,v)=\int_{\Omega} \Phi_1(|\nabla u|)dx-\int_{\Omega}\Phi_2(|\nabla v|)dx-\int_{\Omega}R(x,u,v)dx.
$$
no belongs to $C^{1}(W ^{1,\Phi_1}_0 (\Omega)\times W ^{1,\Phi_2}_0 (\Omega),\mathbb{R})$. Have this in mind, we have decide to work in the space $W ^{1}_0 E^{\Phi_i}(\Omega)$, because it is topologically more rich than $W ^{1,\Phi_i}_0(\Omega)$, for example, it is possible to prove that the energy functional $J$ is $C^1(W ^{1}_0 E^{\Phi_1}(\Omega)\times W ^{1}_0 E^{\Phi_2}(\Omega), \mathbb{R})$. Even knowing that the result contained in Proposition 3.7 in \cite{alves} presents inconsistency when dropping the $\Delta_2$-condition, we add a ''Ambrosetti–Rabinowitz'' condition under the function $R$ and we refine part of the technique presented by Alves et al., so that, together with the saddle-point theorem of Rabinowitz without Palais-Smale condition, we can prove the existence of a Palais-Smale bounded sequence. Finally, due to the possible lack of reflexivity of spaces $W ^{1,\Phi_i}_0 (\Omega)$$(i=1,2)$, we will use some properties of the weak$^*$ topology of these spaces to guarantee the existence of nontrivial solutions for the system \eqref{S}.

%  para estudar a existencia de soluções do sistema \eqref{S}, nos permite, junto com the saddle-point theorem of Rabinowitz [...] without Palais-Smale condition, provar that there is a Palais-Smale bounded sequence.  

It is important to stress that, to the best of our knowledge, this is the first paper where the linking theorem for locally lipschitz fuctionals and the saddle-point theorem of Rabinowitz has been used to deal with a quasilinear elliptic system driven by an $N$-functions may not verify the $\Delta_{2}$-condition.

\section{Preliminaries}
We recall a few notations and results on the critical point theory for locally Lipschitz functionals defined on a real Banach space X with norm $\lVert\cdot\lVert_{X}$. This results can be found in \cite{Maria} and in references therein.

Let $X$ be a Banach space. Let $I : X \rightarrow \mathbb{R}$ be a locally Lipschitz functional
($I \in Lip_{loc}(X, \mathbb{R})$, that is, for each $x \in X$, there exist an open neighborhood
$N(x)$ of $x$ and a constant $k(x) > 0$, such that 
\begin{align*}
|I(y_1)-I(y_2)|\leq k(x)\lVert y_1 -y_2\lVert,
\end{align*}
for all $y_1$ and $y_2$ in $N(x)$. 

A generalized directional derivative of a locally Lipschitz functional $I : X \rightarrow \mathbb{R}$ at $x \in X$ in the direction $v \in X$, denoted by $I^0(x;v)$, is defined by
\begin{align*}
	I^0(x;v)=\limsup_{h\rightarrow 0\; \lambda\rightarrow 0^+}\frac{I(x+h+\lambda v)-I(x+h)}{\lambda}
\end{align*}
and the generalized gradient of $I$ at $x$ is the set
\begin{align*}
	\partial I(x)=\{\mu\in X^*: \langle \mu, v\rangle\leq I^0(x;v),\;v\in X\}.
\end{align*}

%Como se pode ver na Propriedade $(P_1)$ in [..., pag 176], a função $\lambda_I : X \rightarrow \mathbb{R}$, dada por
%$$\lambda_I(x) = \min\{\lVert \mu\lVert_{X^*}: \mu \in 	\partial I(x)\}.$$
%está bem definida.

Let $Q$ be a compact metric space and let $Q_*$ be a nonempty closed subset strictly contained in $Q$. We set
\begin{align}
\mathcal{P}=\{p\in C(Q, X)\colon p=p_*\text{ on } Q_*\},
\end{align}
where $p_*$ is a fixed continuous map on $Q_*$ and
\begin{align}
c=\inf_{c\in\mathcal{P}}\max_{x\in Q} I(p(x)).
\end{align}
So
\begin{align}
c\geq\max_{x\in Q_*} I(p_*(x)).
\end{align}

We say that the subset $A\subset X$ {\it links with the pair} $(Q,Q_*)$ if $p_*(Q_*)\cap A=\emptyset$ and for each $p\in \mathcal{P}$, $p(Q)\cap A\neq\emptyset$.

\begin{theorem}\label{1}
	Let $I\in Lip_{loc}(X, \mathbb{R})$ and $A\subset I_c=\{x\in X:I(x)\geq c\}$ be a closed subset which links with the pair $(Q,Q_*)$. Then there exists a sequence $(x_n) \subset X$ satisfying 
	\begin{align*}
	\lim_{n\rightarrow\infty }d(x_n,A)=0,\quad \lim_{n\rightarrow\infty }I(x_n)=c\quad\text{ e }\quad
	\lim_{n\rightarrow\infty }\lambda_I(x_n)=0,
	\end{align*}
	with
	$$\lambda_I(x_n) = \min\{\lVert \mu\lVert_{X^*}: \mu \in 	\partial I(x_n)\}.$$
\end{theorem}

\begin{proposition}\label{11}
	Let $I:X\rightarrow \mathbb{R}$ be a continuous and Gateaux-differentiable functional such that $I' :X \rightarrow X^*$ is continuous from the norm topology of $X$ to the weak$^*$-topology of $X^*$. Then $I\in Lip_{loc}(X, \mathbb{R})$ and $	\partial I(x)=\{I'(x)\},$ $\forall x\in X$.
\end{proposition}

The Theorem \ref{1} together with Proposition \ref{11} allows us to propose a linking theorem for Gateaux-differentiable functionals. This result will be fundamental to study the class of system proposed in section 3.
\begin{theorem}(The linking theorem)\label{link}
Let $X$ be a real Banach space with $X=Y\oplus Z$, where $Y$ is finite
dimensional. 
Suppose that $I:X\longrightarrow \mathbb{R}$ is continuous and Gateaux-differentiable with derivative $I':X\longrightarrow X^{*}$ continuous from the norm topology of $E$ to the weak$^* $-topology of $X^*$ satisfying:

		\noindent{$(I_1)$} There is $\sigma> 0$ such that if $\mathcal{N}=\{u\in Z: \lVert u\lVert\leq\sigma\}$, then $\displaystyle b\;\dot{=}\inf_{\partial \mathcal{N}}I>0$.
		
	\noindent{$(I_2)$} There are $z_*\in Z\cap \partial B_1$ and $\rho >\sigma$ such that $$0=\sup_{\partial \mathcal{M}}I< d\;\dot{=}\sup_{ \mathcal{M}}I, $$

	where $$\mathcal{M}=\{u=\lambda z_*+y:\lVert u\lVert\leq \rho,\; \lambda\geq 0,\;y\in Y\}$$
%	
%	 e $$\mathcal{N}_0:=\partial \mathcal{M}=\{u=\lambda z_*+y:\lVert u\lVert= \rho\text{ e } \lambda\geq 0\text{ ou } \lVert u\lVert\leq\rho\text{ e } \lambda= 0\}.$$
	If
	\begin{align*}
	c=\inf_{\gamma\in \Gamma}\max_{x\in \mathcal{N}}I(\gamma(t)), 
	\end{align*}
	with \begin{align*}
	\Gamma=\{\gamma\in C(\mathcal{N},X): \gamma|_{\partial \mathcal{N}}=Id_{\partial\mathcal{N}} \}.
	\end{align*}
Then, $b\leq c$ and there is a sequence $(u_n) \subset X$ such that
	\begin{align*}
	\quad I(u_n)\rightarrow c\quad\text{ and }\quad
 I'(u_n)\rightarrow 0.
	\end{align*}
\end{theorem}

\begin{proof}The result follows from Proposition \ref{11} and Theorem \ref{1} with $\mathcal{P}=\Gamma$, $Q = \mathcal{N}$, $Q_* = \partial \mathcal{N}$, $p_* = Id_{Q_*}$ e $A =\{x\in Z+ Y:I(x)\geq c\}$. 

\end{proof}

For the last section of this paper we will use the already known saddle-point theorem of Rabinowitz without Palais-Smale condition. The proof of this result also follows from Theorem \ref{1} along with Proposition \ref{11}.

\begin{theorem}(Saddle-point theorem)\label{saddle}
Let $X$ be a real Banach space with $X=Y\oplus Z$, where $Y$ is finite
dimensional. Suppose that $I:X\rightarrow \mathbb{R}$ is continuous and Gateaux-differentiable with derivative $I':X\rightarrow X^{*}$ continuous from the norm topology of $E$ to the weak$^* $-topology of $X^*$ satisfying:

\noindent{$(I_1)$} there are constants $\rho > 0$ and $\alpha_1 \in \mathbb{R}$ such that if $\mathcal{M}=\{u\in Y: \lVert u \lVert\leq \rho\}$, then $I|_{\partial \mathcal{M}}\leq \alpha_1$.

\noindent{$(I_2)$} there is a constant $\alpha_2>\alpha_1$ such that $I|_{Z}\geq \alpha_2$.
If
	\begin{align*}
	c=\inf_{\gamma\in \Gamma}\max_{x\in \mathcal{M}}I(\gamma(t)), 
	\end{align*}
with  \begin{align*}
	\Gamma=\{\gamma\in C(\mathcal{M},X): \gamma|_{\partial \mathcal{M}}=Id|_{\partial\mathcal{M}} \}.
	\end{align*}
Then, there is $(u_n) \subset X$ such that
\begin{align*}
\quad I(u_n)\rightarrow c\quad\text{ and }\quad
I'(u_n)\rightarrow 0.
\end{align*}
\end{theorem}

\section{The functions $\tilde{\Phi }_1$ and $\tilde{\Phi }_2$ may not verify the $\Delta_{2}$-condition}

In this section, we study the existence of solutions for the following class of quasilinear systems in Orlicz-Sobolev spaces:
	\begin{equation*}
\left\{\;
\begin{aligned}
-div(\phi_1(|\nabla u|)\nabla u)&=F_u(x,u,v)+\lambda R_u(x,u,v)\;\text{ in } \Omega& \\
-div(\phi_2(|\nabla v|)\nabla v)&=-F_u(x,u,v)-\lambda H_v(x,u,v)\;\text{ in }  \Omega& \\
u=v&=0\;\text{ on } \partial\Omega&
\end{aligned}
\right.
\leqno{(S_1)}
\end{equation*}
where $\lambda > 0$ is a parameter, $\Omega$ is a bounded domain in $\mathbb{R}^N$($N \geq 2$) with smooth boundary $\partial \Omega$, and 
%where the $N$-functions $\Phi$ and $\Psi$ are of the form 
%\begin{align}
%\Phi(t)=\int_{0}^{|t|}s\phi(s)ds\;\text{ and }\;\Psi(t)=\int_{0}^{|t|}s\psi(s)ds,
%\end{align}
 $\phi_i (i=1,2):(0,\infty)\rightarrow(0,\infty)$  are two functions which satisfy:
 
 \noindent{$(\phi_1)$} $\phi_i\in C^1(0,+\infty)$ and  $ t\mapsto t\phi_i(t)$ are stricly increasing.

\noindent{$(\phi_2)$} $t\phi_i(t)\rightarrow 0$ as $t\rightarrow 0$ and $t\phi_i(t)\rightarrow +\infty$ as $t\rightarrow +\infty$

\noindent{$(\phi_3)$} 
$\displaystyle
1\leq\ell_i= \inf_{t>0}{\dfrac{t^{2}\phi_i(t)}{\Phi_i(t)}}\leq \sup_{t>0}{\dfrac{t^{2}\phi_i(t)}{\Phi_i(t)}}=m_i<N,$ where $\Phi_i(t)=\int_{0}^{|t|}s\phi_i(s)ds$ and $\ell_i < m_i < \ell_i ^{*}$.
%\begin{equation*}
%t\longmapsto t\psi(t)\;\text{ is increasing for }\;t>0.
%\leqno{(\psi_1)}
%\end{equation*}
%\begin{equation*}
%\displaystyle\lim_{t\rightarrow0^{+}}t\psi(t)=0 \;\;\text{ and }\;\; \displaystyle\lim_{t\rightarrow+\infty}t\psi(t)=+\infty.
%\leqno{(\psi_2)}
%\end{equation*}
%\begin{equation*}
%1\leq m_1= \inf_{t>0}{\dfrac{\psi(t)t^{2}}{\Psi(t)}}\leq \sup_{t>0}{\dfrac{\psi(t)t^{2}}{\Psi(t)}}=m_2<N,\;\;\;\;\text{ and }\;\;\;\;m_1 < m_2 <m_1 ^{*}.
%\leqno{(\phi_3)}
%\end{equation*}

Before continuing this section, consider $\alpha_1,\alpha_2\in(0,\frac{N}{N-1}-1)$ such that $\alpha_1\leq \alpha_2$. We would like to point out that $\Phi_{1}(t)=|t |\ln(|t|^{\alpha_1} +1)$ and $\Phi_{2}(t) =|t |\ln(|t|^{\alpha_2} +1)$ satisfying $(\phi_1) - (\phi_3)$ with $\ell_1=\ell_2=1$ and $m_1=1+\alpha_1$, $m_2=1+\alpha_2$ respectively. These functions are examples of $N$-functions whose the complementary functions $\tilde{\Phi}_{1}$ and $\tilde{\Phi}_{2}$ do not satisfy the
$\Delta_{2}$-condition, consequently $W^{1 ,\Phi_{\alpha_1}}_0(\Omega)\times W ^{1,\Phi_{\alpha_2}}_0(\Omega)$ is nonreflexive.

 In this section, we would like to recall that $(u,v)\in W^{1,\Phi_1}_0(\Omega)\times W^{1,\Phi_2}_0(\Omega)$ is a weak solution of $(S_1)$ whenever
\begin{align*}
\int_{\Omega}\phi_1(|\nabla u|)\nabla u\nabla w_1dx-\int_{\Omega}\phi_2(|\nabla v|)\nabla v\nabla w_2dx=\int_{\Omega}H_u(x,u,v)w_1dx+\int_{\Omega}H_v(x,u,v)w_2dx,
\end{align*}
for all $(w_1,w_2)\in W^{1,\Phi_1}_0(\Omega)\times W^{1,\Phi_2}_0(\Omega)$. Here, let us consider the $H$ function as follows:
$$H(x,u,v)=F(x,u,v)+ \lambda R(x,u,v)$$
where $F(x,u,v)={\Phi_{1*}(u)}+G(v)$, $\lambda > 0$ is a real parameter and $\Phi_{1*}$ denotes the Sobolev conjugate function of $\Phi_{1}$ defined by
\begin{align*}
\Phi_{1*}^{-1}(t)=\int_{0}^{t}\dfrac{\Phi_1 ^{-1}(s)}{s^{(N+1)/N}}ds\;\;\text{for}\;\;t>0\;\;\text{ when }\;\;\int_{1}^{+\infty}\dfrac{\Phi_1 ^{-1}(s)}{s^{(N+1)/N}}ds=+\infty.
\end{align*}
Furthermore, the functions $G$ and $R$ satisfy the following conditions:
 
\noindent{$(G_1)$} There are $C>0$, $G\in C^1(\mathbb{R},\mathbb{R})$, $a_1,a_2\in(1,\infty)$ and a $N$-function $ A(t)=\int_{0}^{|t|}sa(s)ds$ satisfying
\begin{equation*}
m_2<a_1\leq \dfrac{a(t)t^2}{A(t)}\leq a_2 , \;\;\forall t>0
\leqno{(i)}
\end{equation*}
and
\begin{equation*}
|g(s)|\leq  a_1Ca(|s|)|s|, \;\;\;\text{ for all }s\in\mathbb{R}
\leqno{(ii)}
\end{equation*}
where $g(s)=G'(s)$. If $a_2\geq \ell_2 ^*$, we add that
\begin{equation*}
(g(t)-g(s))(t-s)\geq C a(|t-s|)|t-s|^2, \;\;\;\text{ for all }t,s\in\mathbb{R}.
\leqno{(iii)}
\end{equation*}

\noindent{$(G_2)$} There exists $\nu \in(0, {\ell_1})$ such that 
\begin{equation*}
0\leq \nu G(s)\leq sg(s), \;\;\;\text{ for all }s\in\mathbb{R}.
\end{equation*}

\noindent{$(R_1)$} $R\in C^{1}(\overline{\Omega}\times\mathbb{R}^{2})$, $R_u(x,0,0)=0$, $R_v(x,0,0)=0$, $R(x,u,v)\geq0$ and $\linebreak R_u(x,u,v)u\geq0$, for all $(x,u,v)\in \overline{\Omega}\times\mathbb{R}^{2}$.

\noindent{$(R_2)$} There are $N$-functions $B(t)=\int_{0}^{|t|}sb(s)ds$, $P(t)=\int_{0}^{|t|}sp(s )ds$, $Q(t)=\int_{0}^{|t|}sq(s)ds$ and $Z(t)=\int_{0}^{|t|}sz(s)ds $ satisfying
\begin{equation*}
m_1< p_1\leq\dfrac{p(t)t^{2}}{P(t)}\leq p_2<\ell_1^{*}
\leqno{(i)}
\end{equation*}
\begin{equation*}
m_1< b_1\leq\dfrac{b(t)t^{2}}{B(t)}\leq b_2<\ell_1^{*}
\leqno{(ii)}
\end{equation*}
\begin{equation*}
m_2< q_1\leq\dfrac{q(t)t^{2}}{Q(t)}\leq q_2<\ell_2^{*}
\leqno{(iii)}
\end{equation*}
\begin{equation*}
m_2< z_1\leq\dfrac{z(t)t^{2}}{Z(t)}\leq z_2<\ell_2^{*},
\leqno{(iv)}
\end{equation*}

\noindent with $\max\{b_2,q_2\}<\min\{\ell_1^*,\ell_2^*\}$ so that
\begin{equation}\label{R.2}
|R_u(x,u,v)|\leq C(p(|u|)u+q(|v|)v) \;\;\;\;\text{ and }\;\;\;\; |R_v(x,u,v)|\leq C(b(|u|)u+z(|v|)v),
\end{equation}
for all $(x,u,v)\in \Omega\times\mathbb{R}^{2}$ and for some constant $C>0$.

\noindent{$(R_3)$} There exists $\mu \in(m_1, {\ell_1^{*}})$ such that
\begin{equation*}
\dfrac{1}{\mu}R_u(x,u,v)+\dfrac{1}{\nu}R_v(x,u,v)-R(x,u,v)\geq 0,\;\text{ for all } x\in \Omega \text{ and } (u,v)\in \mathbb{R}^2,
\end{equation*}
 where $\nu$ is given by condition $(G_2)$.

\noindent{$(R_4)$} There exists $s\in (m_1,\max\{p_2,b_2\}]$, a nonempty open subset $\Omega_0\subset\Omega$ and a constant $\omega>0$ such that
\begin{equation*}
R(x,u,v)\geq \omega|u|^s\;\;\text{ for all } x\in \Omega_0 \text{ and } (u,v)\in \mathbb{R}^2.
\end{equation*}

The main result of this section is the following.

\begin{theorem}\label{teo1}
	If $(\phi_1)-(\phi_3)$, $(G_1)-(G_2)$, $(R_1)-(R_4)$ hold, then there exists $ \lambda_0 >0$ such that $(S_1)$ possesses a nontrivial solution for all $\lambda>\lambda_0$.
\end{theorem}

Fix $p\in (m_1,\ell_1 ^*)$ and $q\in (m_2,\ell_2 ^*)$. The function $R(u,v)=|u|^p+C|v|^q +\varepsilon\sin |u|^p\sin|v|^q$ satisfies $(R_1)-(R_4)$ with $P(t)=B(t)=|t|^p /p$, $Q(t)=Z(t)=|t|^q /q$, $C>0$ and $\varepsilon>0$ small enough.

Before proving the above theorem, we have to fix some notations. In the sequel $V_A$ stands for the space $W^{1,\Phi_2}_{0}(\Omega)\cap L^{A}(\Omega)$ endowed with the norm
 \begin{align*}
 \lVert v\lVert_A=\lVert v\lVert_{W^{1,\Phi_2}_{0}(\Omega)}+|v|_{A},
 \end{align*}
 where $\lVert v\lVert_{W^{1,\Phi_2}_{0}(\Omega)}$ and $|v|_{A}$ denote the usual norms in in $W^{1,\Phi_2}_{0}(\Omega)$ and $L^{A}(\Omega)$, respectively.
 
We write $X$ for the space $W^{1,\Phi_1}_{0}(\Omega)\times V_A$ endowed with the norm
  \begin{align*}
 \lVert (u,v)\lVert^2=\lVert u\lVert_{W^{1,\Phi_1}_{0}(\Omega)}^2+\lVert v\lVert_{A}^2,
 \end{align*}
 where $\lVert u\lVert_{W^{1,\Phi_1}_{0}(\Omega)}$ denotes the usual norm in $W^{1,\Phi_1}_{0}(\Omega)$.
% 
% and $J_\lambda:X\longrightarrow\mathbb{R}$ denota o funcional dado por 
% \begin{align}
% 	J_\lambda(u,v)=\int_{\Omega} \Phi(|\nabla u|)dx-\int_{\Omega} \Psi(|\nabla v|)dx-\int_{\Omega}H(x,u,v)dx.
% \end{align}
Under the assumptions $(G_1)$ and $(R_2)$, the functional $\mathcal{H}_{\lambda}$ given by
 \begin{align}
\mathcal{H}_{\lambda}(u,v)=\int_{\Omega}H(x,u,v)dx.
\end{align}
 is well defined, belongs to $C^{1}(X,\mathbb{R})$ and 
\begin{align}
\mathcal{H}_{\lambda}'(u,v)(w_1,w_2)=\int_{\Omega}H_u(x,u,v)w_1dx+\int_{\Omega}H_v(x,u,v)w_2dx,
\end{align} 
for all $(u,v),(w_1,w_2)\in X$. Now, we consider the functional $Q:X\rightarrow\mathbb{R}$ which is given by
	\begin{align}
	Q(u,v)=\int_{\Omega} \Phi_1(|\nabla u|)dx-\int_{\Omega} \Phi_2(|\nabla v|)dx,
	\end{align}
	It is well known in the literature that $Q\in C^1(E,\mathbb R)$ when $\Phi_1$, $\Phi_2$, $\tilde{\Phi}_1$ and $\tilde{\Phi}_2$ satisfy the $\Delta_2$-condition  and this occurs when we have the condition satisfied to $\ell_1>1$ and $\ell_2>1$. When $\ell_1=1$ (or $\ell_2=1$), we know that $\tilde{\Phi}_1\notin (\Delta_{2})$ (or $\tilde{\Phi}_2\notin (\Delta_{2})$) and therefore cannot guarantee the differentiability of functional $Q$. However, following the ideas presented in \cite{Lucas}, it is clear that the functional $Q$ is continuous and Gateaux-differentiable with derivative $Q': X\rightarrow X^*$ given by
	\begin{align*}
	Q'(u,v)(w_1,w_2)=\int_{\Omega} \phi_1(|\nabla u|)\nabla u\nabla w_1 dx-\int_{\Omega} \phi_2(|\nabla v|)\nabla v\nabla w_2 dx,
	\end{align*}
 continuous from the norm topology of $X$ to the weak$^*$-topology of $X^*$. Therefore, we can conclude that the energy functional $J_\lambda:X\rightarrow\mathbb{R}$ associated with the system $(S_1)$ given by
	$$
	J_\lambda(u,v)=\int_{\Omega} \Phi_1(|\nabla u|)dx-\int_{\Omega}\Phi_2(|\nabla v|)dx-\int_{\Omega}H(x,u,v)dx.
	$$
	is continuous and Gateaux-differentiable with derivative $J_\lambda':X\rightarrow X^*$ defined by
	{\footnotesize\begin{align*}
	J_\lambda'(u,v)(w_1,w_2)=\int_{\Omega} \phi_1(|\nabla u|)\nabla u\nabla w_1 dx-\int_{\Omega} \phi_2(|\nabla v|)\nabla v\nabla w_2 dx-\int_{\Omega}H_u(x,u,v)w_1dx-\int_{\Omega}H_v(x,u,v)w_2dx
	\end{align*}}
	 continuous from the norm topology of	$X$ to the weak$^*$-topology of $X^*$. Since $J_{\lambda}'(0, 0)=0$, we say that $(u, v)$ is a nontrivial solution of $(S_1)$ when  $J_{\lambda}'(u,v)(w_1,w_2)=0$, for all $(w_1,w_2)\in X$ and satisfies $J_{\lambda}(u, v) \neq0$.
	
%	
%	sabemos que é estritamente convexa e semicontínua inferior em relação para a topologia fraca$^{*}$, mas não é necessariamente de classe $C^1(X,\mathbb{R})$. Devido a este fato, vamos procurar um ponto crítico no sub-differential. A partir dos comentários apresentados, usaremos uma versão do teorema do passo da Montanha generalizado sem a condição $(PS)$ (Ver Apêndice 3) que é obtido através de uma releitura do  teorema do passo da Montanha generalizado desenvolvido por Szulkin. Nesse sentido, diremos que  $u \in X$ é um ponto crítico para $J$ se $0\in \partial J(u,v)=\partial Q(u,v)-\mathcal{H}'(u,v)$. Então $(u,v)\in X$ é um ponto crítico de $J$ se, e somente se, $\mathcal{H}'(u,v)\in\partial Q(u,v)$, what, since $Q$ is convex, is equivalent to
%	\begin{align}
%	Q(w_1,w_2)-Q(u,v)\geq \int_{\Omega}H_u(x,u,v)(w_1-u)dx+\int_{\Omega}H_v(x,u,v)(w_2-v)dx,
%	\end{align} 
%	para todo $(w_1,w_2)\in X$. 
 
In order to apply the linking theorem \ref{link}, we introduce one more piece of notation. Since $(V_A, \lVert\cdot\lVert_{A})$ is separable, then there exists a sequence $(e_n)\subset V_A$ such
\begin{align}\label{0.38}
	V_A=\overline{span\{e_n:n\in\mathbb{N}\}}.
\end{align}
Hereafter, for each $n\in\mathbb{N}$ we denote by $V^{n}_A $ and $X_n$ the following spaces
\begin{align*}
V^{n}_A ={span\{e_j:j=1,\cdots,n\}} \quad\text{ and }\quad X_n=W^{1,\Phi_1}_{0}(\Omega)\times V_A^{n}.
\end{align*}
The restriction of $J_\lambda$ to $X_n$ will be denoted by $J_{\lambda,n}$. Then $J_{\lambda,n} : X_n\longrightarrow\mathbb{R} $ is the functional given by
	$$
J_{\lambda,n}(u,v)=\int_{\Omega} \Phi_1(|\nabla u|)dx-\int_{\Omega}\Phi_2(|\nabla v|)dx-\int_{\Omega}H(x,u,v)dx.
$$
is continuous and Gateaux-differentiable with derivative $J_{\lambda,n}':X_n\longrightarrow X_n^*$ given by
	{\footnotesize\begin{align*}
	J_{\lambda,n}'(u,v)(w_1,w_2)=\int_{\Omega} \phi_1(|\nabla u|)\nabla u\nabla w_1 dx-\int_{\Omega} \phi_2(|\nabla v|)\nabla v\nabla w_2 dx-\int_{\Omega}H_u(x,u,v)w_1dx-\int_{\Omega}H_v(x,u,v)w_2dx
	\end{align*}}
continuous from the norm topology of $X_n$ to the weak$^*$-topology of $X_n^*$.

In the following, we prove that $J_{\lambda,n}$ satisfies the hypotheses of Theorem \ref{link}.

\begin{lemma}\label{lemma1}
Assume that $(G_1)-(G_2)$ and $(R_1)-(R_4)$ hold. For every $\lambda>0$, there exist $\sigma >0$ and $\rho>\sigma $  such that if $u_*\in W^{1,\Phi_1}_{0}(\Omega)$ satisfies $\lVert u_*\lVert_{W^{1,\Phi_1}_{0}(\Omega)}=1$, then
\begin{equation*}
d_n:=\sup_{\mathcal{M}^n _{u_*}}J_{\lambda,n}\geq b_n:=\inf_{\mathcal{N}_n}J_{\lambda,n}>0=\max_{\partial\mathcal{M}^n _{u_*}}J_{\lambda,n},
\end{equation*}
where
\begin{align*}
	\mathcal{M}^n _{u_*}=\{(\theta u_*,v)\in X_n: \lVert(\theta u_*, v)\lVert^2\leq\rho^2,\; \theta\geq 0\}\;\text{ and }\;
\mathcal{N}_n =\{(u,0)\in X_n: \lVert u\lVert_{W^{1,\Phi_1}_{0}(\Omega)}=\sigma\}.
\end{align*}
\end{lemma}

\begin{proof} By $(R_1)$ and $(R_2)$,
%Por definição do funcional $J_{\lambda,n}$,
%	\begin{align*}
%		J_{\lambda,n}(u,0)=\int_{\Omega} \Phi(|\nabla u|)dx-\int_{\Omega} \Phi_{*}(|u|)dx-\lambda\int_{\Omega}R(x,u,0)dx.
%	\end{align*}
%Note que, integrando a primeira desigualdade em \eqref{R.2}, de $0$ até $t$, obtemos
%	\begin{align*}
%	|R(x,t,s)-R(x,0,s)|\leq C(P(|t|)+uq(|s|)s)\;\text{ for all }(x,t,s)\in \Omega\times\mathbb{R}^{2}
%	\end{align*}
%	Daí, fazendo $s=0$, pela premissa $(R_1)$ a desigualdade acima se reduz a seguinte desigualdade
%		\begin{align*}
%	|R(x,t,0)|\leq CP(|t|), \;\text{ for all }(x,t,0)\in \Omega\times\mathbb{R}^{2}
%	\end{align*}
%para alguma constante $C>0$. Assim, 	
	\begin{align*}
J_{\lambda,n}(u,0)\geq\xi^{0}_{\Phi}(\lVert u\lVert_{{1,\Phi}})-\xi^{1}_{\Phi_*}(\lVert u\lVert_{{\Phi_*}})-C\lambda\xi^{1}_{P}(\lVert u\lVert_{{P}}),
\end{align*}	
where 	
\begin{align*}
	\xi^{0}_{\Phi_1}(t)=\min\{t^{\ell_1},t^{ m_1}\},\quad \xi^{1}_{\Phi_{1*}}(t)=\max\{t^{\ell_1^*},t^{ m_1^*}\}	\quad\text{ and }\quad\xi^{1}_{P}(t)=\max\{t^{p_1},t^{ p_2}\}.
\end{align*}
	
Now, remember that by the assumption $(R_2)(i)$ it is possible to show the following limits:
\begin{align*}
	\lim_{t\rightarrow 0}\dfrac{P(|t|)}{\Phi_1(|t|)}=0\quad\text{ and }\quad\lim_{|t|\rightarrow+\infty}\dfrac{P(|t|)}{\Phi_{1*}(|t|)}=0.
\end{align*} 	
Through these two limits we can guarantee the existence of a constant $C_1>0$ that does not depend on $u$ so that
\begin{align*}
	\lVert u\lVert_{{P}}\leq C_1\lVert u\lVert_{{1,\Phi_1}},\quad\forall u \in W^{1,\Phi_1}_0 (\Omega).
\end{align*}
Another important inequality was proved by Donaldson and Trudinger \cite{dona}, which establishes the existence of a constant $S_0 > 0$ that depends on $N$ such that
\begin{align}\label{l}
\lVert u\lVert_{{\Phi_{1*}}}\leq S_0\lVert u\lVert_{{1,\Phi_1}},\quad\forall u \in W^{1,\Phi_1}_0 (\Omega).
\end{align}
Thus,
\begin{align*}
J_{\lambda,n}(u,0)\geq\lVert u\lVert_{{1,\Phi_1}}^{m_1}-\lVert u\lVert_{1,\Phi_1}^{\ell_1^*}-C\lambda\lVert u\lVert_{1,\Phi_1}^{p_1}, \;\;\text{ for }\lVert u\lVert_{{1,\Phi}}<1
\end{align*}
Since $m_1<\ell_1^*$ and $m_1<p_1$, choose $\rho >0$ sufficiently small such that 
\begin{align}\label{0.23}
J_{\lambda,n}(u,0)\geq C\sigma^{\ell_2},\;\;\text{ for } \lVert u\lVert_{{1,\Phi}}=\rho,
\end{align} 	
therefore
\begin{align}\label{0.66}
	b_n:=\inf_{\mathcal{N}_n}J_{\lambda,n}\geq C\sigma^{\ell_2},\;\;\forall n\in\mathbb{N}.
\end{align}
	
Now, from $(G_2)$ and $(R_1)$,	
\begin{align}\label{0.0}
J_{\lambda,n}(0,v)\leq 0,\;\;\forall v \in V^{n}_A.
\end{align}	
Consider $u_*\in W^{1,\Phi_1}_0 (\Omega)$ with $\lVert u_*\lVert_{{1,\Phi_1}}=1$, by assumptions $(G_2)$ and $( R_1)$,
\begin{align*}
J_{\lambda,n}(\theta u_*,v)&\leq\int_{\Omega} \Phi_1(|\nabla (\theta u_*)|)dx-\int_{\Omega} \Phi_2(|\nabla v|)dx-\int_{\Omega} \Phi_{1*}(|\theta u_*|)dx\\
&\leq\xi^{1}_{\Phi_1}(\theta) \xi^{1}_{\Phi_1}(\lVert  u_*\lVert_{{1,\Phi_1}})-\xi^{0}_{\Phi_2}(\lVert  v\lVert_{{1,\Phi_2}})-\xi^{0}_{\Phi_{1*}}(\theta) \xi^{0}_{\Phi_{1*}}(\lVert  u_*\lVert_{{\Phi_{1*}}})	
\end{align*}	
for each $\theta>0$ and $v\in V^{n}_A$, where 	
\begin{align*}
\xi^{1}_{\Phi_1}(t)=\max\{t^{\ell_1},t^{ m_1}\},\quad \xi^{0}_{\Phi_{1*}}(t)=\min\{t^{\ell_1^*},t^{ m_1^*}\}	\quad\text{ and }\quad\xi^{0}_{\Phi_2}(t)=\max\{t^{\ell_2},t^{ m_2}\}.
\end{align*}	
	
If $a_2< m_1^*$, then $A$ increases essentially more slowly than $\Phi_{2*}$ near infinity. From Theorem $8.35$ in \cite{Adms} it follows that $L^{\Phi_{2*}}(\Omega)$ is continuously embedded in $L^{A}(\Omega)$, consequently $W^{ 1,\Phi_1}_0 (\Omega)=V_A$ and as norms $\lVert\cdot\lVert_{A}$ and $\lVert\cdot\lVert_{{1,\Phi_2}}$ are equivalent. Given this, there is a constant $C>0$ such that
\begin{align*}
J_{\lambda,n}(\theta u_*,v)\leq\xi^{1}_{\Phi_1}(\theta)-C\xi^{0}_{\Phi_2}(\lVert  v\lVert_{{A}})-\xi^{0}_{\Phi_*}(\theta) \xi^{0}_{\Phi_{1*}}(\lVert  u_*\lVert_{{\Phi_{1*}}}),	
\end{align*}
for each $\theta>0$ and $v\in V^{n}_A$.

Note that $\lVert (\theta u_* ,v)\lVert^2=\theta^2+\lVert v\lVert_{A}^2=\rho^2$ implies that 
\begin{align*}
	\theta^2\geq \dfrac{\rho^2}{2}\;\;\;or\;\;\; \lVert v\lVert_{A}^2\geq \dfrac{\rho^2}{2}.
\end{align*}
Assume that $\theta^2\geq {\rho^2}/{2}$ occurs, then for $\rho>0$ large enough, we have
\begin{align*}
\xi^{1}_{\Phi_1}(\theta)-C\xi^{0}_{\Phi_2}(\lVert  v\lVert_{{A}})-\xi^{0}_{\Phi_{1*}}(\theta) \xi^{0}_{\Phi_{1*}}(\lVert  u_*\lVert_{{\Phi_{1*}}})=\theta^{m_1}-C\xi^{0}_{\Phi_2}(\lVert  v\lVert_{{A}})-\theta^{\ell_1 ^*}\xi^{0}_{\Phi_*}(\lVert  u_*\lVert_{{\Phi_*}})<0,
\end{align*}
because $m_1<\ell_1 ^*$. Similar property happens when $ \lVert v\lVert_{A}^2\geq {\rho^2}/{2}$. Therefore, we conclude that there exists $\rho>\sigma$ such that
\begin{align}\label{0.1}
J_{\lambda,n}(\theta u_*,v)\leq 0,
\end{align}
for all $(\theta u_* ,v)\in X_n$ so that $\lVert \theta u_*\lVert_{{1,\Phi_1}}+\lVert v\lVert_{A}^2=\rho^ 2$ and $\theta>0$. By \eqref{0.0} and \eqref{0.1}, we have $\displaystyle\max_{\partial\mathcal{M}^n _{u_*}}J_n=0$, since $(0,0)\in \partial\mathcal{M}^n _{u_*}$, and the proof is complete in this case.
	
Now, if $a_2\geq m_1^*$ , from $(G_1)(ii)$ there is a positive constant $C$ such that
\begin{align*}
J_{\lambda,n}(\theta u_*,v)&\leq\int_{\Omega} \Phi_1(|\nabla (\theta u_*)|)dx-\int_{\Omega} \Phi_2(|\nabla v|)dx-\int_{\Omega} \Phi_{1*}(|\theta u_*|)dx-C\int_{\Omega} A(| v|)dx\\
&\leq\xi^{1}_{\Phi_1}(\theta) \xi^{1}_{\Phi_1}(\lVert  u_*\lVert_{{1,\Phi}})-\xi^{0}_{\Phi_2}(\lVert  v\lVert_{{1,\Phi_2}})-\xi^{0}_{\Phi_{1*}}( \theta) \xi^{0}_{\Phi_{1*}}(\lVert  u_*\lVert_{{\Phi_{1*}}})-\xi^{0}_{A}(| v|_{{A}}),
\end{align*}
where 
\begin{align*}
\xi^{0}_{A}(t)=\min\{t^{a_1}, t^{a_2}\}.
\end{align*}
Observing that  $\lVert (\theta u_* ,v)\lVert^2=\theta^2+\lVert v\lVert_{A}^2=\rho^2$ implies that 
\begin{align*}
\theta^2\geq \dfrac{\rho^2}{2},\;\;\lVert v\lVert_{{1,\Phi_2}}^2\geq \dfrac{\rho^2}{4} \;\;\;or\;\;\; |v|_{A}^2\geq \dfrac{\rho^2}{4},
\end{align*}
the same argument used in the former case implies that for $\rho > 0$ large enough
\begin{align}
J_n(\theta u_*,v)\leq 0,
\end{align}
for all $(\theta u_* ,v)\in X_n$ so that $\lVert \theta u_*\lVert_{{1,\Phi_1}}+\lVert v\lVert_{A}^2=\rho^ 2$ and $\theta>0$. Therefore, the lemma is proved.

\end{proof}

	In order to prove the Theorem \ref{teo1}, we need to consider that $\Omega_0\subset \Omega$ be an open set satisfying $(R_4)$ and $u_0\in W^{1,\Phi_1}_0(\Omega)$$  $ such that
	\begin{align}\label{2.55}
u_0\geq0,\;\; u_0\neq0,\;\; supp(u_0)\subset \Omega_0 \;\text{ and }\; \lVert u_0\lVert_{W^{1,\Phi_1}_{0}(\Omega)}=1.
\end{align}
%\begin{lemma}\label{0.34}
%	Then there is $t_0=t_0(u_*)>0$ such that
%	\begin{align*}
%	J_{\lambda,n}(t_0 u_*,0)<0.
%	\end{align*}
%\end{lemma}
%\begin{proof}
%	For every $t > 0$, note that for $(R_1)$,
%	\begin{align*}
%	J_{\lambda,n}(t u_*,0)&=\int_{\Omega} \Phi(t|\nabla u_*|)dx-\int_{\Omega} \Phi_{*}(t|u_0*|) dx-\lambda\int_{\Omega}R(x,tu_*,0)dx\\
%	&\leq\xi^{1}_{\Phi}(t)\int_{\Omega} \Phi(|\nabla u_*|)dx-\xi^{0}_{\Phi_*}(t)\int_{\Omega} \Phi_{*}(|u_*|)dx-\lambda\int_{\Omega}R(x,tu_*,0)dx\\
%	&\leq\xi^{1}_{\Phi}(t)\int_{\Omega} \Phi(|\nabla u_*|)dx-\xi^{0}_{\Phi_*}(t)\int_{\Omega} \Phi_{*}(|u_*|)dx,
%	\end{align*}
%	where $\xi^1 _{\Phi}(t)=\max\{t^{\ell_1},t^{\ell_2}\}$ and $\xi^0 _{\Phi_*}(t) =\min\{t^{\ell_1^*},t^{\ell_2^*}\}$ from which it follows that
%	\begin{align}\label{0.24}
%	J_{\lambda,n}(t u_*,0)\leq t^{\ell_2}\int_{\Omega} \Phi(|\nabla u_*|)dx -t^{\ell_1^*}\int_{\Omega} \Phi_{*}(|u_*|)dx,\;\;\;t\geq 1.
%	\end{align}
%	Since $\ell_2<\ell_1^{*}$, we conclude that
%	\begin{align*}
%	J_{\lambda,n}(t u_*,0)<0,
%	\end{align*}
%	for $t$ large enough.
%	
%\end{proof}
Then, by Lemma \ref{lemma1}, we can apply the linking theorem \ref{link} to functional $J_{\lambda,n}$ using a point $z_n = (u_0, 0)$ and the sets
	$$ Y_n=\{0\}\times V^{n}_r,\quad Z=W^{1,\Phi_1}_0(\Omega)\times\{0\}\quad\text{ and }\quad\mathcal{N}_n=\{(u,0)\in X_1:\lVert u\lVert_{{1,\Phi_1}}=\sigma\}.$$
	Then, there are sequences $(u_k, v_k)\subset X_n$ such that
	\begin{align}\label{2.50}
	J_{\lambda,n}(u_k,v_k)\rightarrow c_{\lambda,n}\;\;\text{ and }\;\; J_{\lambda,n}'(u_k,v_k)\rightarrow 0\;\text{ as }k\rightarrow\infty
	\end{align} 
where 
\begin{align}\label{2.51}
b_n\leq c_{\lambda,n}:=\inf_{\gamma\in \Gamma}\max_{u\in\mathcal{M}^n _{u_0}} J_{\lambda,n} (\gamma(u)),
\end{align}
\begin{align*}
\Gamma=\{\gamma\in C(\mathcal{M}^n _{u_0}, X_n) :\;\gamma|_{\partial\mathcal{M}^n _{u_0}}=Id_{\partial\mathcal{M}^n _{u_0}}\}.
\end{align*}

\begin{lemma}
The sequence $(u_k,v_k)$ is bounded in $X_n$.
\end{lemma}
\begin{proof}From \eqref{2.50}
\begin{align*}
J_{\lambda,n}(u_k,v_k)-J_n'(u_k,v_k) (\dfrac{1}{\mu}u_k,\dfrac{1}{\nu}v_k)=	c_{\lambda,n}+o_k(1)\lVert(u_k,v_k)\lVert
\end{align*}
By $(G_2)$, $(R_3)$, $(\phi_3)$,
 \begin{align*}
J_{\lambda,n}(u_k,v_k)-J_{\lambda,n}'(u_k,v_k) (\dfrac{1}{\mu}u_k,\dfrac{1}{\nu}v_k)
\geq \left(1-\dfrac{m_1}{\mu}\right) \xi^{0}_{\Phi_1}(\lVert u_k\lVert_{1,\Phi_1})+\left(\dfrac{\ell_2}{\nu}-1\right)\xi^{0}_{\Phi_2}(\lVert v_k\lVert_{1,\Phi_2}),
 \end{align*}
where $\xi^0 _{\Phi_1}(t)=\min\{t^{\ell_1},t^{m_1}\}$ and $\xi^0 _{\Phi_2}(t)=\min\{t^{\ell_2},t^{m_2}\}$.  Since $V_A ^n$ is a finite dimensional space, the norms $\lVert \cdot\lVert_{1,\Phi_2}$ and $\lVert\cdot\lVert_{A}$ are equivalent, hence, from the above inequalities
  \begin{align}\label{21}
c_{\lambda,n}+o_k(1)\lVert(u_k,v_k)\lVert
 \geq \left(1-\dfrac{m_1}{\mu}\right) \xi^{0}_{\Phi_1}(\lVert u_k\lVert_{1,\Phi_1})+\left(\dfrac{\ell_2}{\nu}-1\right)\xi^{0}_{\Phi_2}(C\lVert v_k\lVert_{A}),
 \end{align}
for some $C=C(n)>0$. Suppose for contradiction that, up to a subsequence, $\lVert(u_k,v_k)\lVert \rightarrow +\infty$ as $k\rightarrow +\infty$. This way, we need to study the following situations:

\noindent{$(i)$}  $\lVert u_k\lVert_{1,\Phi_1}\rightarrow +\infty$ and $\lVert v_k\lVert_{A}\rightarrow \infty$

\noindent{$(ii)$} $\lVert u_k\lVert_{1,\Phi_1}\rightarrow +\infty$ and $\lVert v_k\lVert_{A}$ is bounded

\noindent{$(iii)$}  $\lVert v_k\lVert_{A}\rightarrow \infty$ and $\lVert u_k\lVert_{1,\Phi_1}$ is bounded

In the first case, the inequality \eqref{21} implies that
\begin{align*}
2c_{\lambda,n} ^2+o_k(1)\lVert(u_k,v_k)\lVert ^2\geq\left(1-\dfrac{m_1}{\mu}\right)^2 \lVert u_k\lVert_{1,\Phi_1}^{2\ell_1} +\left(\dfrac{\ell_2}{\nu}-1\right)^2\lVert v_k\lVert_{A}^{2\ell_2}.
\end{align*}
for $k$ large enough. Which is absurd, because $\ell_1\geq 1$, $\ell_2 \geq 1$ and $o_k(1) \rightarrow 0$.

In case $(ii)$, we have for $k$ large enough
\begin{align*}
2c_{\lambda,n} ^2+C_1+o_k(1)\lVert u_k\lVert_{1,\Phi}\geq\left(1-\dfrac{m_1}{\mu}\right)^2 \lVert u_k\lVert_{1,\Phi_1}^{2\ell_1}
\end{align*}
an absurd. The last case is similar to the case $(iii)$.

\end{proof}
\begin{corollary}\label{0.33}
	The following sequences
	\begin{align*}
	\{\lVert u_k\lVert_{\Phi_{1*}}\}_{k\in\mathbb{N}},\quad\left\{\int_{\Omega}\Phi_1(|\nabla u_k|)dx\right\}\quad\text{ and }\quad\left\{\int_{\Omega}\Phi_{1*}(|\nabla u_k|)dx\right\}
	\end{align*}
	are bounded.
	
\end{corollary}

From Lemma \ref{10.6}, we may assume that there exists a subsequence of $(u_k, v_k)$,
still denoted by itself, and  $(w_n, y_n)\in X_n$ such that
\begin{align}\label{2.01}
u_k \xrightharpoonup[\quad]{\ast}w_n\;\; \text{weakly\;in }& W^{1,\Phi_1}_0(\Omega) \;\;\text{ and }\;\;v_k \xrightharpoonup[\quad]{\ast}y_n\;\; \text{weakly\;in } V_n, \;\;\text{ as }\;\;k\rightarrow\infty\\& u_k\xrightharpoonup[\quad]{}w_n\quad\label{0.2}\text{in } L^{\Phi_{1*}}(\Omega)\\&
\dfrac{\partial u_k}{\partial x_i} \xrightharpoonup[\quad]{\ast}\dfrac{\partial w_n}{\partial x_i} \quad\text{in } L^{\Phi_1}(\Omega),\;i\in \{1,\cdots,N\}.
\end{align}
\begin{align}\label{0.4}
u_k\longrightarrow w_n\quad\text{in } L^{\Phi_1}(\Omega),
\end{align}
and
\begin{align}\label{0.3}
u_k(x)\longrightarrow w_n(x)\quad\text{a.e. } \Omega.
\end{align}

In view of \eqref{2.01} and Corollary \ref{0.33}, for each $n\in \mathbb{N}$ we may assume that there exist nonnegative functions $\mu_n, \nu_n\in\mathcal{M}(\mathbb{R}^N)$, the space of Radon measures, such that
\begin{align}\label{2.59}
\begin{split}
\Phi(|\nabla u_k|)\xrightharpoonup[\quad]{\ast}\mu_n\quad\text{in } \mathcal{M}(\mathbb{R}^N)\;\text{ and }\;
\Phi_{*}(| u_k|)\xrightharpoonup[\quad]{\ast}\nu_n\quad\text{in } \mathcal{M}(\mathbb{R}^N)\;\text{ as }k\rightarrow\infty.
\end{split}
\end{align}

The result below is known as second concentration-compactness lemma of P. L. Lions. We would like to point out that also this lemma holds for nonreflexive Orlicz-Sobolev space. The proof of this fact can be seen in Proposition 4.3. in \cite{FN}.

\begin{lemma}\label{CCL} $(i)$ For every $n\in \mathbb{N}$ and $\lambda>0$, there exist an at most countable set $I_\lambda$,  a family $\{x_i\}_{i\in I_\lambda}$ of
	distinct points in  $\mathbb{R}^N$ and a family $\{\nu_{i}\}_{i\in I_\lambda}$ of constant  $\nu_{i} > 0$ such that
	\begin{align}
	\nu_n=\Phi_{1*}(|w_n|)+\sum_{i\in I_\lambda}\nu_i\delta_{x_{i}}.
	\end{align}
	\noindent{$(ii)$} In addition we have
	\begin{align}
	\mu_n\geq\Phi_1(|\nabla w_n|)+\sum_{i\in I_\lambda}\mu_i\delta_{x_i},
	\end{align}
	for some $\mu_j>0$ satisfying
	\begin{align}
	0<\nu_j\leq \max\{S_0^{\ell_1^*}\mu_i^{\ell_1^*/\ell_1},S_0^{m_1^*}\mu_i^{m_1^*/\ell_1},S_0^{\ell_1^*}\mu_i^{\ell_1^*/m_1},S_0^{m_1^*}\mu_i^{m_1^*/m_1}\}
	\end{align}
	for all $i\in I_\lambda$, where $\delta_{x_i}$ is the Dirac measure of mass 1 concentrated at $x_i$.
\end{lemma}

\begin{lemma}
	The set $\{x_i\}_{i\in I_\lambda}$ in Lemma \ref{CCL} is a finite set.
\end{lemma}

\begin{proof}
Let an $x_i$ be fixed. Take $\varphi\in C^{\infty}_0(\mathbb{R}^N)$ such that 
	\begin{align*}
	0\leq \varphi\leq 1,\quad \varphi(x)=1\text{ in } B_1(0) \quad\text{ and }\quad \varphi(x)=0\text{ in } \mathbb{R}^N\backslash B_2(0) 
	\end{align*}
	and put $\varphi_{\varepsilon}(x)=\varphi((x-x_i)/\varepsilon)$ for $\varepsilon>0$. It is clear that $\{\varphi_{\varepsilon}u_k\}_{k\in\mathbb{N}}$ is bounded in $W^{1,\Phi_1}_0(\Omega)$, thus $J_ {\lambda,n}'(u_k,v_k)(\varphi_{\varepsilon}u_k,0)=o_k(1)$, that is,
	
	\begin{align*}
	\begin{split}
	\int_{\Omega}&\phi_1(|\nabla u_k|)\nabla u_k\cdot \nabla (\varphi_{\varepsilon}u_k)dx=
	\int_{\Omega}\phi_{1*}(| u_k|)u_k (\varphi_{\varepsilon}u_k)dx+\lambda\int_{\Omega}R_u(x,u_k,v_k)(\varphi_{\varepsilon}u_k)dx+o_k(1)
	\end{split}
	\end{align*}
Knowing that
\begin{align*}
\int_{\Omega}R_u(x,u_k,v_k)(\varphi_{\varepsilon}u_k)dx\longrightarrow \int_{\Omega}R_u(x,w_n,y_n)(\varphi_{\varepsilon}w_n)dx \;\text{ as } k\rightarrow \infty
\end{align*}
We can conclude that
{\small\begin{align}\label{0.9}
	\begin{split}
	\int_{\Omega}\phi_1(|\nabla u_k|)\nabla u_k\cdot \nabla (\varphi_{\varepsilon}u_k)dx\leq m_1^*
	\int_{\Omega}\Phi_{1*}(| u_k|) \varphi_{\varepsilon}dx+\lambda\int_{\Omega}R_u(x,w_n,y_n)(\varphi_{\varepsilon}w_n)dx+o_k( 1)
	\end{split}
	\end{align} }
By $(\phi_{3})$,	
	\begin{align*}
	\begin{split}
	\int_{\Omega}\phi_1(|\nabla u_k|)\nabla u_k\cdot \nabla (\varphi_{\varepsilon}u_k)dx\geq \ell_1\int_{\Omega}\Phi_1(|\nabla u_k|)\varphi_{\varepsilon}dx+\int_{\Omega}\phi_1(|\nabla u_k|)u_k\nabla u_k\nabla \varphi_{\varepsilon}dx.
	\end{split}
	\end{align*}	
	Therefore,
		{\footnotesize\begin{align}\label{0.7}
		\begin{split}
	\ell_1\int_{\Omega}\Phi_1(|\nabla u_k|)\varphi_{\varepsilon}dx+\int_{\Omega}\phi_1(|\nabla u_k|)u_k\nabla u_k\nabla \varphi_{\varepsilon}dx\leq m_1^*
		\int_{\Omega}\Phi_{1*}(| u_k|) \varphi_{\varepsilon}dx+\lambda\int_{\Omega}R_u(x,w_n,y_n)(\varphi_{\varepsilon}w_n)dx+o_k(1)
		\end{split}
		\end{align}	}
	
Since the sequence $(\phi_1(|\nabla u_k|)\dfrac{\partial u_k}{\partial x_j} )_{k\in\mathbb{N}}$ is limited to $L^{\tilde{\Phi }_1}(\Omega)$, there is $\omega_j\in L^{\tilde{\Phi }_1}(\Omega)$ such that
\begin{align}\label{0.14}
\phi_1(|\nabla u_k|)\dfrac{\partial u_k}{\partial x_j} \xrightharpoonup[\quad]{\ast} \omega_j\text{ in }L^{\tilde{\Phi }_1}(\Omega), \;\;j\in\{1,\cdots, N\}
\end{align}
for some subsequence. Knowing that $$u_k\dfrac{\partial \varphi_{\varepsilon}}{\partial x_j}\longrightarrow w_n\dfrac{\partial \varphi_{\varepsilon}}{\partial x_j} \;\text{ in } L ^{{\Phi _1}}(\Omega), \;\;j\in\{1,\cdots, N\}$$
we conclude that
\begin{align}\label{0.6}
\int_{\Omega}u_k\phi_1(|\nabla u_k|)\dfrac{\partial u_k}{\partial x_j}\dfrac{\partial \varphi_{\varepsilon}}{\partial x_j}dx\longrightarrow \int_ {\Omega}w_n\phi_1(|\nabla w_n|)\dfrac{\partial w_n}{\partial x_j}\dfrac{\partial \varphi_{\varepsilon}}{\partial x_j}dx,\;\;j \in\{1,\cdots, N\}.
\end{align}
Therefore, considering $\omega=(\omega_1,\cdots,\omega_N)$ we get
%	\begin{align*}
%	\Big|\int_{\Omega}\big(u_k^n\phi(|\nabla u_k^n|)\nabla u_k^n\nabla \varphi_{\varepsilon}-u_nw\nabla \varphi_{\varepsilon}\big)dx\Big|&=\Big|\int_{\Omega}\big(u_k^n\phi(|\nabla u_k^n|)\nabla u_k^n-u_nw\big)\nabla \varphi_{\varepsilon}dx\Big|\\
%	&\leq \sum_{j=1}^{N}\Big|\int_{\Omega}\big(u_k^n\phi(|\nabla u_k^n|)\dfrac{\partial u_k^n}{\partial x_j}-u_nw\big)\dfrac{\partial \varphi_{\varepsilon}}{\partial x_j}dx\Big|
%	\end{align*}	
%	donde segue de \eqref{0.6} que	
%	\begin{align*}
%	\Big|\int_{\Omega}\big(u_k^n\phi(|\nabla u_k^n|)\nabla u_k^n\nabla \varphi_{\varepsilon}-u_nw\nabla \varphi_{\varepsilon}\big)dx\Big|\longrightarrow 0,
%	\end{align*}
%	as $k\rightarrow \infty$. 
%	Assim,
	\begin{align}\label{0.8}
	\int_{\Omega}\phi_1(|\nabla u_k|)u_k\nabla u_k\nabla \varphi_{\varepsilon}dx-	\int_{\Omega}u_n\omega\nabla \varphi_{\varepsilon}dx=o_k(1)
	\end{align}
From \eqref{0.7} and \eqref{0.8},
	{\footnotesize\begin{align}\label{0.10}
	\begin{split}
	m_1^*
	\int_{\Omega}\Phi_{1*}(| u_k|) \varphi_{\varepsilon}dx+\lambda\int_{\Omega}R_u(x,u_n,v_n)(\varphi_{\varepsilon}u_n)dx\geq \ell_1\int_{\Omega}\Phi_1(|\nabla u_k|)\varphi_{\varepsilon}dx+\int_{\Omega}u_n\omega\nabla \varphi_{\varepsilon}dx+o_k(1)
	\end{split}
	\end{align}}
By \eqref{2.59}, taking the limit of $k\rightarrow +\infty$, we get
	\begin{align}\label{0.15}
	\begin{split}
m_1^*
	\int_{\Omega} \varphi_{\varepsilon}d\nu_n
	+\lambda\int_{\Omega}R_u(x,u_n,v_n)\varphi_{\varepsilon}u_ndx\geq\ell_1\int_{\Omega}\varphi_{\varepsilon}d\mu_n+\int_{\Omega}u_nw\nabla \varphi_{\varepsilon}dx+o_k(1).
	\end{split}
	\end{align}
	
On the other hand, given $v\in W^{1,\Phi_1}_0(\Omega)$, it follows from \eqref{2.50} that
	\begin{align*}
	o_k(1)=\int_{\Omega}\phi_1(|\nabla u_k|)\nabla u_k \nabla vdx-
	\int_{\Omega}\phi_{1*}(| u_k|)u_k vdx-\lambda\int_{\Omega}R_u(x,u_k,v_k)vdx,
	\end{align*}
	since the sequence $(\phi_{1*}(|u_k|)u_{k})$ is bounded in $L^{\tilde{\Phi }_{1*}}(\Omega)$, there is $\eta_n\in L^{ \tilde{\Phi }_{1*}}(\Omega)$ such that
	\begin{align}\label{0.13}
	\phi_{1*}(|u_{k}|)u_{k} \xrightharpoonup[\quad]{}\eta_n \text{ in } L^{\tilde{\Phi }_{1*}}(\Omega)\;\text{ as } k\rightarrow\infty
	\end{align}
so, from \eqref{0.14} and \eqref{0.13},
	\begin{align*}
	\int_{\Omega}\omega \nabla vdx-
	\int_{\Omega}\eta_n vdx-\lambda\int_{\Omega}R_u(x,u_n,v_n)vdx=0,\;\;\forall v\in W^{1,\Phi_1} _0(\Omega).
	\end{align*}
 In particular, for $v=u_n\varphi_{\varepsilon}$, we have
	\begin{align}\label{0.16}
	\int_{\Omega}u_n \omega\nabla\varphi_{\varepsilon}dx=
	\int_{\Omega}\eta_n u_n\varphi_{\varepsilon}dx+\lambda\int_{\Omega}R_u(x,u_n,v_n)u_n\varphi_{\varepsilon}dx-\int_{\Omega}\varphi_{\varepsilon} \omega\nabla u_n dx,
	\end{align}
Therefore, 
	\begin{align}\label{00.1}
	\lim_{\varepsilon\rightarrow 0}\int_{\Omega}u_n (\omega\nabla\varphi_{\varepsilon})dx=0
	\end{align}
%	Considerando
%	\begin{equation*}
%	\varphi(x)=\left\{\;\begin {aligned}
%	1, & \;\;x=x_i & \\
%	0,& \;\;x \in \Omega-\{x_i\} &
%	\end{aligned}
%	\right.,
%	\end{equation*}
%	temos 
%	\begin{align*}
%	\varphi_{\varepsilon}(x)\longrightarrow\varphi(x)\;\text{ pontualmente }\text{ in }\Omega
%	\end{align*}
%	quando $\varepsilon\rightarrow0$. Além disso,
%	\begin{align*}
%	\varphi_{\varepsilon}(x)\leq \chi_{B_1(x_i)}(x),
%	\end{align*}
%	para $\varepsilon>0$ suficientemente pequeno. Sabendo disto, pelo Teorema de Lebesgue
%	\begin{align}
%	\int_{\Omega}\varphi_{\varepsilon}d\mu_n \longrightarrow \int_{\Omega}\varphi d\mu_n=\mu_n(\{x_i\})=\mu_i
%	\end{align}
%	and
%	\begin{align}\label{0.17}
%	\int_{\Omega}\varphi_{\varepsilon}d\nu_n \longrightarrow \int_{\Omega}\varphi d\nu_n=\nu_n(\{x_i\})=\nu_i,
%	\end{align}
%	quando $\varepsilon\rightarrow0$.
It follows from \eqref{0.15} and \eqref{00.1} that
\begin{align}\label{0.18}
\ell_1\mu_i\leq m_1^{*}\nu_i,\;\;i\in I_\lambda,
\end{align}
and by Lemma \ref{CCL},
	\begin{align}\label{0.35}
	\ell_1\mu_i \leq m_1^*S_0^\beta\mu_i ^{\alpha},
	\end{align}
for some $\alpha$ and $\beta$ verifying
	\begin{align}\label{0.36}
	1<\alpha\in\left\{\dfrac{\ell_1^*}{\ell_1},\dfrac{m_1^*}{\ell_1},\dfrac{\ell_1^*}{m_1},\dfrac{m_1^*}{m_1}\right\}\;\;\text{ and }\;\beta\in \{\ell_1^*,m_1^*\}.
	\end{align}
Thereby,
	\begin{align*}
	0<\dfrac{\ell_1}{m_1^*S_0 ^\beta}\leq \mu_i^{\alpha-1},\;\;i\in I_\lambda,
	\end{align*}
showing that
	\begin{align}\label{0.19}
	\mu_i \geq \big(\dfrac{\ell_1}{m_1^*S_0 ^\beta}\big)^{\frac{1}{\alpha-1}},\;\;i\in I_\lambda.
	\end{align}
By \eqref{0.18} and \eqref{0.19}
	\begin{align*}
	\nu_i \geq \big(\dfrac{\ell_1}{m_1^*}\big)^{\frac{\alpha}{\alpha-1}}S_0 ^{-\frac{\beta}{\alpha-1}},\;\;i\in I_\lambda.
	\end{align*}
	Since $\nu_n$  is a finite measure, the last inequality yields $I_\lambda$ is finite.
	
\end{proof}

In fact we will see that the set $\{x_i\}_{i\in I_\lambda}$ is an empty set for $\lambda>0$ large enough. For this, we will establishes an important estimate from above of the mountain level of
functional $J_{\lambda,n}$.

%
%\begin{lemma}\label{0.34}
%	Let $\Omega_0\subset \Omega$ be an open set satisfying $(R_4)$ and $u_0\in C^{\infty}_0(\Omega)$, such that
%	\begin{enumerate}
%		\item[(i)] $u_0\geq 0$, $u_0\not\equiv0;$
%		\item[(ii)] $supp(u_0)\subset \Omega_0$.
%	\end{enumerate}
%	Then there is $t_0=t_0(u_0)>0$ such that
%	\begin{align*}
%	J_{\lambda,n}(t_0 u_0,0)<0.
%	\end{align*}
%\end{lemma}
%\begin{proof}
%For every $t > 0$, note that for $(R_1)$,
%\begin{align*}
%J_{\lambda,n}(t u_0,0)&=\int_{\Omega} \Phi(t|\nabla u_0|)dx-\int_{\Omega} \Phi_{*}(t|u_0|) dx-\lambda\int_{\Omega}R(x,tu_0,0)dx\\
%&\leq\xi^{1}_{\Phi}(t)\int_{\Omega} \Phi(|\nabla u_0|)dx-\xi^{0}_{\Phi_*}(t)\int_{\Omega} \Phi_{*}(|u_0|)dx-\lambda\int_{\Omega}R(x,tu_0,0)dx\\
%&\leq\xi^{1}_{\Phi}(t)\int_{\Omega} \Phi(|\nabla u_0|)dx-\xi^{0}_{\Phi_*}(t)\int_{\Omega} \Phi_{*}(|u_0|)dx,
%\end{align*}
%where $\xi^1 _{\Phi}(t)=\max\{t^{\ell_1},t^{\ell_2}\}$ and $\xi^0 _{\Phi_*}(t) =\min\{t^{\ell_1^*},t^{\ell_2^*}\}$ from which it follows that
%\begin{align}\label{0.24}
%J_{\lambda,n}(t u_0,0)\leq t^{\ell_2}\int_{\Omega} \Phi(|\nabla u_0|)dx -t^{\ell_1^*}\int_{\Omega} \Phi_{*}(|u_0|)dx,\;\;\;t\geq 1.
%\end{align}
%Since $\ell_2<\ell_1^{*}$, we conclude that
%\begin{align*}
%J_{\lambda,n}(t u_0,0)<0,
%\end{align*}
%for $t$ large enough.
%	
%\end{proof}

\begin{lemma}\label{0.26}
	Let $n\in \mathbb{N}$ be arbitrary and consider $u_0$ given in \eqref{2.55}. Then, there is $\lambda_0>0$ such that if $\lambda>\lambda_0$, we have that
%	\begin{align*}
%	\max_{\mathcal{M}^n _{u_*}}J_{\lambda,n}<\omega,
%	\end{align*}
%	where
%	\begin{align*}
%		\omega:=\left(1-\dfrac{\ell_2}{\mu}\right)\min\left\{\left(\dfrac{\ell_1}{\ell_2^*S_0^{\beta}}\right)^{\frac{1}{\alpha-1}}  : 	\alpha\in\left\{\dfrac{\ell_1^*}{\ell_1},\dfrac{\ell_1^*}{\ell_1},\dfrac{\ell_1^*}{\ell_2},\dfrac{\ell_2^*}{\ell_2}\right\}\;\text{ and }\;\beta\in \{\ell_1^*,\ell_2^*\} \right\}.
%	\end{align*}
%Consequently,
	\begin{align}\label{0.37}
	c_{\lambda,n}<\left(1-\dfrac{m_1}{\mu}\right)\min\left\{\left(\dfrac{\ell_1}{m_1^*S_0^{\beta}}\right)^{\frac{1}{\alpha-1}}  : 	\alpha\in\left\{\dfrac{\ell_1^*}{\ell_1},\dfrac{m_1^*}{\ell_1},\dfrac{\ell_1^*}{m_1},\dfrac{m_1^*}{m_1}\right\}\;\text{ and }\;\beta\in \{\ell_1^*,m_1^*\} \right\},
	\end{align}
	where $	c_{\lambda,n}$ is given in \eqref{2.51}.
%	\begin{align*}
%		c_{\lambda,n}=\inf_{\gamma\in \Gamma}\max_{u\in\mathcal{M}^n _{u_*}} J_{\lambda,n} (\gamma(u)),
%	\end{align*}
%	\begin{align*}
%		\Gamma=\{\gamma\in C(\mathcal{M}^n _{u_*}, X_n) :\;\gamma|_{\partial\mathcal{M}^n _{u_*}}=Id\}.
%	\end{align*}
\end{lemma}
\begin{proof}By $(R_4)$, given $\theta\geq0$ and $v\in V_A ^n$, we have
\begin{align*}
J_{\lambda,n}(\theta u_0,v)
&\leq\xi^{1}_{\Phi_1}(\theta)\xi^{1}_{\Phi_1}(\lVert u_0\lVert_{{1,\Phi_1}})-\xi^{0}_{\Phi_{1*}}(\theta)\xi^{1}_{\Phi_1}(\lVert u_0\lVert_{{\Phi}_1})-{\lambda}\int_{\Omega_0}R(x,\theta u_0,v)dx\\
&\leq\xi^{1}_{\Phi_1}(\theta)\xi^{1}_{\Phi_1}(\lVert u_0\lVert_{{1,\Phi_1}})-\xi^{0}_{\Phi_{1*}}(\theta)\xi^{1}_{\Phi_1}(\lVert u_0\lVert_{{\Phi_1}})-{\lambda}\omega\int_{\Omega_0}|\theta u_0|^sdx
\end{align*}
This inequality implies that
\begin{align}\label{2.53}
0<b_n\leq \inf_{\gamma\in \Gamma}\max_{u\in\mathcal{M}^n _{u_0}} J_{\lambda,n} (\gamma(u))\leq\max_{\mathcal{M}^n _{u_0}}J_{\lambda,n}\leq\max_{\theta\geq0}\mathcal{V}_{\lambda}(\theta)
\end{align}	
	where 
	\begin{align}\label{2.52}
		\mathcal{V}_{\lambda}(\theta)=\xi^{1}_{\Phi_1}(\theta)\xi^{1}_{\Phi_1}(\lVert u_0\lVert_{{1,\Phi_1}})-\xi^{0}_{\Phi_{1*}}(\theta)\xi^{1}_{\Phi_1}(\lVert u_0\lVert_{{\Phi_1}})-{\lambda}\omega\int_{\Omega_0}|\theta u_0|^sdx.
	\end{align}
In what follows, we denote by $\theta_\lambda > 0$ the real number verifying	
	\begin{align}\label{2.54}
		\max_{\theta\geq0}\mathcal{V}_{\lambda}(\theta)=\mathcal{V}_{\lambda}(\theta_\lambda)
	\end{align}
Let us see that $\mathcal{V}_{\lambda}(\theta_\lambda)\rightarrow0$ as $\lambda\rightarrow\infty$. For that, consider
%	By $(R_4)$ and \eqref{2.50},
%	\begin{align*}
%0<b_n&\leq \inf_{\gamma\in \Gamma}\max_{u\in\mathcal{M}^n _{u_*}} J_{\lambda,n} (\gamma(u))\\
%&\leq\max_{t\geq 0} J_{\lambda_m,n}(t u_0,0)\\
%&=J_{\lambda_m,n}(t_{\lambda_m} u_0,0)\\
%&=\int_{\Omega} \Phi(t_{\lambda_m}|\nabla u_0|)dx-\int_{\Omega} \Phi_{*}(t_{\lambda_m}|u_0|)dx-{\lambda_m}\int_{\Omega}R(x,t_{\lambda_m}u_0,0)dx\\
%&\leq\xi^{1}_{\Phi}(t_{\lambda_m})\int_{\Omega} \Phi(|\nabla u_0|)dx-\xi^{0}_{\Phi_*}(t_{\lambda_m})\int_{\Omega} \Phi_{*}(|u_0|)dx-{\lambda_m}\int_{\Omega}R(x,t_{\lambda_m}u_0,0)dx\\
%&\leq K_1\int_{\Omega} \Phi(|\nabla u_0|)dx-K_2\int_{\Omega} \Phi_{*}(|u_0|)dx-{\lambda_m}\int_{\Omega_0}R(x,t_{\lambda_m}u_0,0)dx
%\end{align*}	
%	
%	Considere a função $$\mathcal{V}_\lambda(\theta)=\xi_{\Phi}^1(\theta)\xi_{\Phi}^1(\lVert u_*\lVert_{{1,\Phi}})-\xi_{\Phi_*}^0(\theta)\xi_{\Phi_*}^0(\lVert u_*\lVert_{{\Phi_*}})$$ 
%	
%	Para mostrar esse Lemma, veremos que $\max_{\theta\geq 0} $
%	Por \eqref{0.24},
%	\begin{align*}
%	J_{\lambda,n}(t u_0,0)\longrightarrow-\infty,\;\;\text{ when }\;t\rightarrow \infty.
%	\end{align*}
%	Considerando 
%	\begin{align}\label{0.25}
%	\max_{t\geq 0} J_{\lambda,n}(t u_0,0)=J_{\lambda,n}(t_\lambda u_0,0),
%	\end{align}
%	obtemos de \eqref{0.23} que $J_{\lambda,n}(t_\lambda u_0,0)>0$. Seja 	
$(\lambda_m)$ a sequence verifying
	\begin{align*}
	\lambda_m\longrightarrow\infty\;\;\text{ as }\;m\rightarrow \infty.
	\end{align*}
	We claim that $(\theta_{\lambda_m})$ is a bounded sequence. Indeed, assuming by contradiction that $(\theta_{\lambda_m})$ is not bounded, we have that for a subsequence, still denote by itself,
	\begin{align*}
	\theta_{\lambda_m}\rightarrow\infty\;\;\text{ as}\;\;m\rightarrow \infty.
	\end{align*}
	Por \eqref{2.52}, \eqref{2.53} and \eqref{2.54},
	\begin{align*}
	0<	\max_{\theta\geq0}\mathcal{V}_{\lambda_m}(\theta)=\mathcal{V}_{\lambda_m}(\theta_{\lambda_m})\leq \theta_{\lambda_m}^{m_1} \xi^{1}_{\Phi_1}(\lVert u_0\lVert_{{1,\Phi_1}})-\theta_{\lambda_m}^{\ell_1^*}\xi^{1}_{\Phi_1}(\lVert u_0\lVert_{{\Phi_1}})\rightarrow-\infty\;\;\text{ as }\;m\rightarrow \infty,
	\end{align*}
	which is an absurd, and so, $(\theta_{\lambda_m})$ is bounded. We claim that $\theta_{\lambda_m}\rightarrow0$ as $m\rightarrow \infty$. If the above limit does not hold, we can assume by contradiction, that for some subsequence, still denote by $(\theta_{\lambda_m})$, there is $k_0 > 0$ such that
	\begin{align*}
		\theta_{\lambda_m}>k_0>0,\;\;\forall n\in\mathbb{N}
	\end{align*}
Then, by \eqref{2.52} and \eqref{2.54},
	\begin{align*}
	0<	\max_{\theta\geq0}\mathcal{V}_{\lambda_m}(\theta)=\mathcal{V}_{\lambda_m}(\theta_{\lambda_m})\leq\xi^{1}_{\Phi_1}(\theta_{\lambda_m})\xi^{1}_{\Phi_1}(\lVert u_0\lVert_{{1,\Phi_1}})-\xi^{0}_{\Phi_{1*}}(\theta_{\lambda_m})\xi^{1}_{\Phi_1}(\lVert u_0\lVert_{{\Phi_1}})-{\lambda_m}\omega\int_{\Omega_0}|\theta_{\lambda_m} u_0|^sdx,
	\end{align*}
thus
	\begin{align*}
	0<\max_{\theta\geq0}\mathcal{V}_{\lambda_m}(\theta)\rightarrow-\infty\;\;\text{ as }\;\;m\rightarrow \infty,
	\end{align*}
which is a contradiction. Hence, 
	\begin{align*}
	t_{\lambda_m}\rightarrow0\;\;\text{ as}\;\;m\rightarrow \infty,
	\end{align*}
which leads to
	\begin{align*}
	\max_{\theta\geq0}\mathcal{V}_{\lambda_m}(\theta)=\mathcal{V}_{\lambda_m}(\theta_{\lambda_m})\rightarrow0\;\;\text{ as }\;\;m\rightarrow \infty,
	\end{align*}
and by \eqref{2.53} it follows that
	\begin{align*}
	c_{\lambda_m,n}\rightarrow0\;\;\text{ as }  \;\;m\rightarrow\infty.
	\end{align*}
	
\end{proof}

%Pelo Lemma \ref{0.26}, podemos considerar $\lambda_0>0$ de modo que 
%\textcolor{red}{\begin{align}\label{0.37}
%	c_{\lambda,n}<\left(1-\dfrac{\ell_2}{\mu}\right)\left(\dfrac{\ell_1}{\ell_2^*\max\{S_0^{\ell_1^*},S_0^{\ell_2^*}\}}\right)^{1-\alpha},\;\;\forall \lambda>\lambda_0\;\text{ and }\;\forall n\in\mathbb{N}
%	\end{align}}
%\noindent onde $\alpha$ é dado em \eqref{0.35} and \eqref{0.36}.

\begin{lemma}\label{lema1}
	For every $n\in\mathbb{N}$ and $\lambda>\lambda_0$ the set $I_\lambda$ is empty, where $\lambda_0$ was given in Lemma \ref{0.26}.
\end{lemma}

\begin{proof}
	Let $(u_k,u_k)\subset X_n$ the $(PS)_{c_{\lambda,n}}$ sequence obtained in \eqref{2.50}. By the assumptions $(G_2)$, $(R_3)$, $(\phi_3)$, we have that
	\begin{align*}
	c_{\lambda,n}+o_k(1)\geq J_n(u_k,v_k)- J_{\lambda,n}'(u_k,v_k) (\dfrac{1}{\mu}u_k,\dfrac{1}{\nu}v_k)
	\geq\left(1-\dfrac{m_1}{\mu}\right)\int_{\Omega}\Phi_1(|\nabla u_k|)dx.
	\end{align*}

Fixing a function $\varphi\in C^{\infty}_0(\mathbb{R}^N)$ with $\varphi(x)=1$ on $\overline{\Omega}$, we derive that
	\begin{align*}
	c_{\lambda,n}+o_k(1)\geq\left(1-\dfrac{m_1}{\mu}\right)\int_{\mathbb{R}^N}\Phi_1(|\nabla u_k|)\varphi dx.
	\end{align*}
	Taking the limit of $k\rightarrow +\infty$, we get
	\begin{align*}
	c_{\lambda,n}\geq\left(1-\dfrac{m_1}{\mu}\right)\int_{\mathbb{R}^N}\varphi d\mu_n\geq\left(1-\dfrac{m_1}{\mu}\right)\mu_n(\overline{\Omega}).
	\end{align*}
	Supposing that $I_\lambda$ is not empty, there is $i\in I_\lambda$, and so,
		\begin{align*}
	c_{\lambda,n}\geq\left(1-\dfrac{m_1}{\mu}\right)\mu_i.
	\end{align*}	
In $\eqref{0.19}$ we show that
	\begin{align*}
	\mu_i\geq\big(\dfrac{\ell_1}{m_1^*S_0 ^\beta}\big)^{\frac{1}{\alpha-1}},\;i\in I_\lambda
	\end{align*}
where $\alpha$ is given in \eqref{0.35} and \eqref{0.36}. Therefore, we can conclude that
	\begin{align*}
	c_{\lambda,n}\geq \left(1-\dfrac{m_1}{\mu}\right)\min\left\{\left(\dfrac{\ell_1}{m_1^*S_0^{\beta}}\right)^{\frac{1}{\alpha-1}}  : 	\alpha\in\left\{\dfrac{\ell_1^*}{\ell_1},\dfrac{m_1^*}{\ell_1},\dfrac{\ell_1^*}{m_1},\dfrac{m_1^*}{m_1}\right\}\;\text{ and }\;\beta\in \{\ell_1^*,m_1^*\} \right\}.
	\end{align*}
	Then, if $\lambda \geq \lambda_0$, the last inequality together with Lemma \ref{0.26} yields $I_\lambda=\emptyset$ is empty.
	
\end{proof}

\begin{lemma}
 For $\lambda \geq  \lambda_0$, the sequence $(u_k)$ is strongly convergent for its weak limit $w_n$ in $ L^{\Phi_{1*}}(\Omega)$ as $k\rightarrow \infty$.
\end{lemma}

\begin{proof}Let $\varphi\in C^{\infty}(\mathbb{R}^N)$ be a function verifying  $\varphi(x)=1,$ for $x\in \Omega$. In this case,
	\begin{align*}
	\begin{split}
	\lim_{k\rightarrow \infty}\int_{\Omega}\Phi_{1*}(u_k)dx= 	\lim_{k\rightarrow \infty}\int_{\mathbb{R}^N}\Phi_{1*}(u_k)\varphi dx= \int_{\mathbb{R}^N}\varphi d\nu_n
	\end{split}
	\end{align*}
	The Lemma \ref{CCL}(i) combined with Lemma \ref{lema1} gives
		\begin{align*}
	\begin{split}
	\lim_{k\rightarrow \infty}\int_{\Omega}\Phi_{1*}(u_k)dx =\int_{\mathbb{R}^N}\Phi_{1*}( u_n)\varphi dx
	&=\int_{\Omega}\Phi_{1*}( u_n) dx.
	\end{split}
	\end{align*}
	Since $\Phi_{1*}$ is a convex function, it follows from a result found in \cite{brezis} that
	\begin{align*}
	\lim_{k\rightarrow \infty}\int_{\Omega}\{\Phi_{1*}(| u_k|)-\Phi_{1*}(| u_k-u_n|)-\Phi_{1*}(|u_n|)\}dx=0.
	\end{align*}
	Then,
	\begin{align*}
	\lim_{k\rightarrow \infty}\int_{\Omega}\Phi_{1*}(| u_k-u_n|)dx=0,
	\end{align*}
we can conclude that $(u_k)$  converges strongly for $u_n$ in $L^{\Phi_{1*}}(\Omega)$.
	
\end{proof}

%

%De um certo modo, acabamos de mostrar através desse lemma que existe uma certa "compacidade" entre os espaços $W^{1,\Phi}_{0}(\Omega)$ and $L^{\Phi_{*}}(\Omega)$.
%

\begin{lemma}\label{0.30}
Consider $\lambda>\lambda_0$ and $(u_k)\subset W^{1,\Phi_1}_0(\Omega)$ the sequence obtained in \eqref{2.50}. Then, for some subsequence, still denoted by itself, 
\begin{align*}
u_k\rightarrow w_n\;\text{ in }W^{1,\Phi_1}_0(\Omega)\;\text{ as }\; k\rightarrow \infty.
\end{align*}
\end{lemma}

\begin{proof}
	Since $(u_k)$ is a bounded sequence in $W^{1,\Phi_1}_0(\Omega)$, then
	{\small$$o_k(1)=\int_{\Omega}\phi_1(|\nabla u_k|)\nabla u_k \nabla (v-u_k)dx-
		\int_{\Omega}\phi_{1*}(| u_k|)u_k (v-u_k)dx-\lambda\int_{\Omega}R_u(x,u_k,v_k)(v-u_k)dx.$$}
	Given $v\in W^{1,\Phi_1}_0(\Omega)$, by the convexity of $\Phi_1$ it follows that
	\begin{align*}
	\Phi_1(|\nabla v|)-\Phi_1(|\nabla u_k|)\geq \phi_1(|\nabla u_k|)\nabla u_k\nabla (v-u_k),
	\end{align*}
	thus,
	{\footnotesize\begin{align}\label{0.27}
		\begin{split}
		\int_{\Omega}\Phi_1(|\nabla v|)dx-\int_{\Omega}\Phi_1(|\nabla u_k|)dx\geq
		\int_{\Omega}\phi_{1*}(| u_k|)u_k (v-u_k)dx-\lambda\int_{\Omega}R_u(x,u_k,v_k)(v-u_k)dx+ o_k(1).
		\end{split}
		\end{align}}
	Through the sequence $(u_k)$ in $W^{1,\Phi_1}_0(\Omega)$ together with the limits
	$$ u_k(x)\longrightarrow w_n\;\;\;\text{ a.e.\;in }\Omega \quad\text{ and }\quad\dfrac{\partial u_k}{\partial x_i}\xrightharpoonup[\quad]{} \dfrac{\partial w_n}{\partial x_i}\;\;\;\text{ in }L^1 (\Omega),$$
	we can apply [\citenum{ET}, Theorem $2.1$, Chapter 8] to get
	\begin{align*}
	\liminf_{k\rightarrow\infty}\int_{\Omega}\Phi_1(|\nabla u_k|)dx\geq \int_{\Omega}\Phi_1(|\nabla w_n|)dx,
	\end{align*}
	Furthermore, since $(u_k)$ strongly converges to $u_n$ in $L^{\Phi_{1*}}(\Omega)$,
	\begin{align*}
	\int_{\Omega}\phi_{1*}(| u_k|)u_k (v-u_k)dx\rightarrow \int_{\Omega}\phi_{1*}(| w_n|)w_n (v-w_n)dx,\;\text{ as }\; k\rightarrow \infty.
	\end{align*}

	Therefore, it follows from \eqref{0.27} that
	{\begin{align*}
		\int_{\Omega}\Phi_1(|\nabla  v|)dx-\int_{\Omega}\Phi_1(|\nabla w_n|)dx
		\geq
		\int_{\Omega}\phi_{1*}(| w_n|)w_n (v-w_n)dx+\lambda\int_{\Omega}R_u(x,w_n,y_n)(v-w_n)dx.
		\end{align*}}
	By arbitrariness $v$ we can conclude that
	\begin{align*}
	\int_{\Omega}\phi_1(|\nabla w_n|)\nabla w_n \nabla(w_n-u_k)dx=
	\int_{\Omega}\phi_{1*}(| w_n|)w_n (w_n-u_k)dx+\lambda\int_{\Omega}R_u(x,w_n,y_n)(w_n-u_k)dx,
	\end{align*}
	implying that
	\begin{align}\label{0.28}
	\int_{\Omega}\phi_1(|\nabla w_n|)\nabla w_n \nabla(w_n-u_k)dx=o_k(1).
	\end{align}
	
On the other hand, since $(u_k,v_k)$ is a sequence $(PS)_{c_{\lambda,n}}$,
	\begin{align*}
		o_k(1)=J_{\lambda,n}'(u_k,v_k)(w_n-u_k,0)
		=&\int_{\Omega}\phi_1(|\nabla u_k|)\nabla u_k \nabla(w_n-u_k)dx-
		\int_{\Omega}\phi_{1*}(| u_k|)u_k (w_n-u_k)dx\\&-\lambda\int_{\Omega}R_u(x,u_k,v_k)(w_n-u_k)dx,
		\end{align*}  
	Therefore,
	\begin{align}\label{0.29}
	\int_{\Omega}\phi_1(|\nabla u_k|)\nabla u_k \nabla(w_n-u_k)dx=o_k(1).
	\end{align}
From \eqref{0.28} and \eqref{0.29},
	\begin{align*}
	\int_{\Omega}\big(\phi_1(|\nabla u_k|)\nabla u_k-\phi_1(|\nabla w_n|)\nabla w_n\big)\big(\nabla u_k-\nabla w_n\big)dx\longrightarrow 0\;\text{ as }k\rightarrow \infty.
	\end{align*}
	Now,  applying a result due to Dal Maso and Murat \cite{Dal}, it follows that
	\begin{align}\label{111}
	\nabla u_k(x)\longrightarrow \nabla w_n(x)\;\text{ a.e.\;in }\Omega\;\text{ as }k\rightarrow \infty.
	\end{align}
%\end{proof}
%
%
%
%\begin{lemma}\label{0.30}
%Let $(u_k)\subset W^{1,\Phi_1}_0(\Omega)$ the sequence obtained in \eqref{2.50}. Then, for
%some subsequence, still denoted by itself, 
%	\begin{align*}
%	u_k\longrightarrow w_n\;\text{ in }W^{1,\Phi_1}_0(\Omega)\;\text{ as }\; k\rightarrow \infty.
%	\end{align*}
%\end{lemma}
%
%\begin{proof}
Since $(u_k)$ is bounded in $W^{1,\Phi_1}_0(\Omega)$ and $\Phi_1\in (\Delta_{2})$, then the sequence $(\phi_1(|\nabla u_k|)\dfrac{\partial u_k}{\partial x_j} )_{k\in\mathbb{N}}$ is bounded by $L^{\tilde{\Phi }_1}(\Omega)$, for each $j\in\{1,\cdots, N\}$. Furthermore, by \eqref{111}, it follows that
 \begin{align*}
 \phi_1(|\nabla u_k(x)|)\dfrac{\partial u_k(x)}{\partial x_j} \rightarrow \phi_1(|\nabla w_n(x)|)\dfrac{\partial w_n(x)}{\partial x_j}\;\text{ a.e.\;in }\Omega\;\text{ as }\; k\rightarrow \infty.
 \end{align*}
Thus, by Lemma 2.5 in \cite{AlvesandLeandro},
\begin{align}\label{0.31}
	\int_{\Omega}\phi_1(|\nabla u_k|)\nabla u_k \nabla vdx\rightarrow\int_{\Omega}\phi_1(|\nabla u_n|)\nabla u_n \nabla vdx,\;\;v\in C^{\infty}_0(\Omega) \;\text{ as }\; k\rightarrow \infty.
\end{align} 
Still due to the boundedness of the sequence $(\phi_1(|\nabla u_k|)\dfrac{\partial u_k}{\partial x_j} )_{k\in\mathbb{N}}$ in $L^{\tilde{\Phi }_1}(\Omega)$, there will be $v_j\in L^{\tilde{\Phi }_1}(\Omega)$ such that
\begin{align*}
	\phi_1(|\nabla u_k|)\dfrac{\partial u_k}{\partial x_j}\xrightharpoonup[\quad]{\ast} v_j \;\text{ in } L^{\tilde{\Phi }_1}(\Omega) \;\text{ as }\; k\rightarrow \infty.
\end{align*}
i.e,
\begin{align}\label{0.32}
\int_{\Omega}\phi_1(|\nabla u_k|)\dfrac{\partial u_k}{\partial x_j}wdx\rightarrow \int_{\Omega}v_jwdx, \;\forall w \in E^{{\Phi _1}}(\Omega)=L^{{\Phi _1}}(\Omega) \;\text{ as }\; k\rightarrow \infty.
\end{align} 
By \eqref{0.31} and \eqref{0.32}, it follows that $v_j=\phi_1(|\nabla u_n|)\dfrac{\partial u_n}{\partial x_j}$, for each $j\in\{1 ,\cdots, N\}$. Still from \eqref{0.32},
\begin{align*}
\int_{\Omega}\phi_1(|\nabla u_k|){\nabla u_k}\nabla wdx\rightarrow \int_{\Omega}\phi_1(|\nabla u_n|){\nabla u_n}\nabla wdx, \;\;\forall w \in W^{{1,\Phi_1 }}_0(\Omega) \;\text{ as }\; k\rightarrow \infty.
\end{align*}
We know from \eqref{0.29} that
 \begin{align*}
\int_{\Omega}\phi_1(|\nabla u_k|)\nabla u_k \nabla(u_n-u_k)dx=o_k(1),
\end{align*}
then
\begin{align*}
\int_{\Omega}\phi_1(|\nabla u_k|)|\nabla u_k|^2dx\rightarrow \int_{\Omega}\phi_1(|\nabla u_n|)|\nabla u_n|^2dx \;\text{ as }\; k\rightarrow \infty.
\end{align*}
Given this, we can conclude that
%
%Como a sequência $(\phi(|\nabla u_k^n|)|\nabla u_k^n|)_{k\in\mathbb{N}}$ é limitada em $L^{\tilde{\Phi }}(\Omega)$, pela desigualdade de Holder $(\phi(|\nabla u_k^n|)|\nabla u_k^n|^2)_{k\in\mathbb{N}}$ é limitada em $L^{1}(\Omega)$. Sendo 
% \begin{align*}
% 	\phi(|\nabla u_k|)|\nabla u_k|^2\geq 0,\;\;\forall k\in\mathbb{N}.
% \end{align*}
% and
% \begin{align*}
%\phi(|\nabla u_k(x)|)|\nabla u_k(x)|^2\longrightarrow \phi(|\nabla u_n(x)|)|\nabla u_n(x)|^2\;\;\text{ a.e.\;in }\Omega \;\text{ as }\; k\rightarrow \infty.
%\end{align*} 
% when $k\rightarrow \infty$, então por teoria da medida,
 \begin{align*}
 \phi_1(|\nabla u_k|)|\nabla u_k|^2\rightarrow \phi_1(|\nabla u_n|)|\nabla u_n|^2\;\;\text{in }L^{1}(\Omega)\;\text{ as }\; k\rightarrow \infty.
 \end{align*} 
By $(\phi_{3})$ together with the $\Delta_2$-condition, it follows that
%
% Assim, a menos de subsequência, existe $M\in L^{1}(\Omega)$ com 
% \begin{align*}
% \phi(|\nabla u_k^n|)|\nabla u_k^n|^2\leq M(x)\;\;\text{ a.e.\;in }\Omega
% \end{align*}
%for all $k\in\mathbb{N}$. Diante disso, como $\Phi\in (\Delta_2)$, então 
%\begin{align*}
%	\Phi(|\nabla u_k^n- \nabla u_n|)&\leq 	\Phi(|\nabla u_k^n|+ |\nabla u_n|)\\
%	&\leq C(	\Phi(|\nabla u_k^n|)+ \Phi(|\nabla u_n|))\\
%	&\leq \dfrac{C}{\ell_1}	\phi(|\nabla u_k^n|)|\nabla u_k^n|^2+ C\Phi(|\nabla u_n|)\\
%	&\leq \dfrac{C}{\ell_1}	M(x)+ C\Phi(|\nabla u_n|)
%\end{align*}
% para alguma constante $C>0$. Como 
% \begin{align*}
% \Phi(|\nabla u_k^n(x)-\nabla u_n(x)|)\longrightarrow 0,\;\;\text{ a.e.\;in }\Omega
% \end{align*} 
% when $k\rightarrow \infty$, então pelo Teorema da Convergência Dominada,
% \begin{align*}
% \int_{\Omega}\Phi(|\nabla u_k(x)-\nabla u_n(x)|)dx\longrightarrow0\;\text{ as }\; k\rightarrow \infty. 
% \end{align*}
%Sabendo que $\Phi\in (\Delta_2)$, concluímos that
 \begin{align*}
 u_k\rightarrow u_n\;\;\text{ in }W^{1,\Phi_1}_0(\Omega)\;\text{ as }\; k\rightarrow \infty.
 \end{align*}
This finishes the proof.
 
\end{proof}

\begin{lemma}\label{0.62}
For $\lambda>\lambda_0$, the sequence $(w_n,y_n)$ is bounded in $X$. Moreover 
	\begin{align}\label{0.46}
		J_{\lambda,n} (w_n,y_n)=c_{\lambda,n}\;\;\text{ and }\;\; J_{\lambda,n}'(w_n,y_n)=0\text{ in }X_n^*.
	\end{align}
\end{lemma}

\begin{proof}
Since $V^n _{A}$ is a finite dimensional space, $(v_k)$ converges strongly to $(y_n)$ in $V^n _{A}$. Therefore,
$$(u_k,v_k)\longrightarrow (w_n,y_n)\;\;\text{ in } X_n\;\text{ as }\; k\rightarrow \infty.$$
which implies
$$J_{\lambda,n}(w_n,y_n)=c_{\lambda,n} \in [b_{n},d_{n}] \;\;\text{ and }\;\; J'_{\lambda,n}(w_n,y_n)=0\;\;\text{ in } X_n ^*.$$

In a first moment, let us assume that $a_2< \ell_2^*$. By hypothesis $m_2<a_1$, then $W^{1,\Phi_2}_0(\Omega)$ is continuously embedded in $L^{A}(\Omega)$, thus, there will be $C>0$ such
\begin{align}\label{2.57}
	\lVert v\lVert_{{A}}\leq C\lVert v\lVert_{{1,\Phi_2}}, \;\;\forall v\in V^{n}_A
\end{align} 
By $(G_2)(ii)$, $(R_3)$ and \eqref{2.57},
{\small\begin{align}\label{2.56}
\begin{split}
	c_{\lambda,n}= &J_{\lambda,n}(w_n,y_n)-J_{\lambda,n}'(w_n,y_n) (\dfrac{1}{\mu}w_n,\dfrac{1}{\nu}y_n)\\
\geq &\left(1-\dfrac{m_1}{\mu}\right)\int_{\Omega}\Phi_1(|\nabla w_n|)dx+\left(\dfrac{\ell_2}{\nu}-1\right)\int_{\Omega}\Phi_2(|\nabla y_n|)dx+
\left(\dfrac{\ell_1^*}{\mu}-1\right)\int_{\Omega}\Phi_{1*}(|\nabla w_n|)dx\\
\geq&\left(1-\dfrac{m_1}{\mu}\right) \xi^{0}_{\Phi}(\lVert w_n\lVert_{1,\Phi_1})+\left(\dfrac{\ell_2}{\nu}-1\right)\xi^{0}_{\Phi_2}\big(\frac{1}{C}\lVert y_n\lVert_{A}\big).
\end{split}
\end{align}}
It follows from the inequalities \eqref{0.37} and \eqref{2.56} that $(w_n,y_n)$ is bounded in $X$.

Now, let us assume that $a_2\geq \ell_2^*$. By $(R_3)$, $(G_2)$, \eqref{2.57} and from items $(i)-(iii)$ of $(G_1)$, it follows that
{\small\begin{align*}
c_{\lambda,n}= &J_{\lambda,n}(u_n,v_n)-J_{\lambda,n}'(u_n,v_n) (\dfrac{1}{\mu}u_n,\dfrac{1}{\nu}v_n)\\
\geq &\left(1-\dfrac{m_1}{\mu}\right)\int_{\Omega}\Phi_1(|\nabla u_n|)dx+\left(\dfrac{\ell_2}{\nu}-1\right)\int_{\Omega}\Phi_2(|\nabla v_n|)dx+\dfrac{C}{\nu}\int_{\Omega}a(|v_n|)|v_n|^2dx-Ca_1\int_{\Omega}A(v_n)dx\\
\geq &\left(1-\dfrac{m_1}{\mu}\right)\int_{\Omega}\Phi_1(|\nabla u_n|)dx+\left(\dfrac{\ell_2}{\nu}-1\right)\int_{\Omega}\Phi_2(|\nabla v_n|)dx+\left(\dfrac{a_1C}{\nu}-a_1C\right)\int_{\Omega}A(|v_n|)dx\\
\geq&\left(1-\dfrac{m_1}{\mu}\right) \xi^{0}_{\Phi_1}(\lVert u_n\lVert_{1,\Phi_1})+\left(\dfrac{\ell_2}{\nu}-1\right) \xi^{0}_{\Phi_2}(\lVert v_n\lVert_{1,\Phi_2})+\left(\dfrac{a_1C}{\nu}-a_1C\right)\xi^{0}_{A}(| v_n|_{{A}}).
\end{align*}}
From the above inequality together with \eqref{0.37}, we can conclude that $(w_n,y_n)$ is bounded by $X$.

\end{proof}

\subsection{Proof of Theorem \ref{teo1}}
The proof of Theorem \ref{teo1} will be carried out in three lemmas. We start observing
that since $(w_n,y_n)$ is bounded, there is no loss of generality in assuming that
\begin{align}\label{0.55}
(w_n,y_n)\xrightharpoonup[\quad]{\ast} (u,v)\;\;\text{ in }X\;\text{ as }\; n\rightarrow \infty.
\end{align}
The same arguments used in the proof of Lemma \ref{0.30} can be repeated to show that
	\begin{align}\label{0.03}
u_n\rightarrow u\;\;\text{ in }W^{1,\Phi_1}_0(\Omega)\;\text{ as }\; n\rightarrow \infty.
\end{align}

By the limit \eqref{0.55}, it follows that
\begin{align}\label{0.56}
y_n\xrightharpoonup[\quad]{\ast} v\;\;\text{ in }L^{A}(\Omega)
\end{align}
and
\begin{align}\label{0.57}
y_n\xrightharpoonup[\quad]{\ast} v\;\;\text{ in }W^{1,\Phi_2}_0(\Omega)
\end{align}

\begin{lemma}\label{0.53}
For $\lambda>\lambda_0$, the sequence $(y_n)$ verifies the following limit
	$y_n\rightarrow v$ in $L^{A}(\Omega)$.	
\end{lemma}
\begin{proof}
	From \eqref{0.38}, there is $(\xi_k) \subset V_A $ such that
	\begin{align}\label{0.43}
	\xi_k\rightarrow v\;\text{ in }\;V_A
	\end{align}
	and
	\begin{align*}
	\xi_k=\sum_{i=1}^{j(k)}\alpha_ie_i\in V_A^{j(k)}
	\end{align*}
	where $j(k)\in \mathbb{N}$ for all $k\in\mathbb{N}$. For each $k \in \mathbb{N}$, it follows that
	\begin{align*}
	V_A^{j(k)}\subset V_A^{n}\;\text{ for all }\; n\geq n_0
	\end{align*}
	for some $n_0\geq j(k)$.
	
	If  $a_2\geq \ell_2^*$, from $(G_1)$, we have that there is $C>0$ such that
	{\small\begin{align}\label{0.39}
	\begin{split}
	a_1C\int_{\Omega}A(|y_n-\xi_k|)dx\leq C\int_{\Omega}a(|y_n-\xi_k|)|y_n-\xi_k|^2dx
	\leq \int_{\Omega}(g(y_n)-g(\xi_k))(y_n-\xi_k)dx.
	\end{split}
	\end{align}}
	Since $J'_{\lambda,n}(u_n,v_n)=0$ in $X_n ^*$, we derive that
	{\small\begin{align}\label{0.40}
	\begin{split}
	\int_{\Omega}(g(y_n)-g(\xi_k))(y_n-\xi_k)dx
	=&\int_{\Omega}\phi_2(|\nabla y_n|)\nabla y_n(\nabla\xi_k-\nabla y_n)dx-\lambda\int_{\Omega}R_{v}(x,w_n,y_n)y_ndx\\
	&+\lambda\int_{\Omega}R_{v}(x,w_n,y_n)\xi_kdx-\int_{\Omega}g(\xi_k)(y_n-\xi_k)dx
	\end{split}
	\end{align}}
Due to the convexity of $\Phi_2$, we have
\begin{align}\label{0.41}
\Phi_2(|\nabla \xi_k|)-\Phi_2(|\nabla y_n|)\geq \phi_2(|\nabla y_n|)\nabla y_n\nabla (\xi_k-y_n), \;\;n\in \mathbb{N}
\end{align}
It follows from the above inequalities that
	\begin{align}
	\begin{split}
	a_1C\int_{\Omega}A(|y_n-\xi_k|)dx\leq&\int_{\Omega}\Phi_2(|\nabla  \xi_k|)dx-\int_{\Omega}\Phi_2(|\nabla w_n|)dx-\lambda\int_{\Omega}R_{v}(x,w_n,y_n)y_ndx\\
	&+\lambda\int_{\Omega}R_{v}(x,w_n,y_n)\xi_kdx-\int_{\Omega}g(\xi_k)(y_n-\xi_k)dx
	\end{split}
	\end{align}
Knowing that
$$ y_n(x)\longrightarrow v(x)\;\text{ a.e.\;in }\Omega\;\;\text{ and }\;\;\dfrac{\partial y_n}{\partial x_i}\xrightharpoonup[\quad]{} \dfrac{\partial v}{\partial x_i}\;\text{ in }L^1(\Omega),$$
we can apply [\citenum{ET}, Theorem $2.1$, Chapter 8] to get
\begin{align}\label{0.42}
\liminf_{n\rightarrow\infty}\int_{\Omega}\Phi_2(|\nabla v_n|)dx\geq \int_{\Omega}\Phi_2(|\nabla v|)dx.
\end{align}
Taking as limit $n \rightarrow \infty$, it follows that
	{\small\begin{align}
	\begin{split}
	\limsup_{n\rightarrow\infty}\left(a_1C\int_{\Omega}A(|y_n-\xi_k|)dx\right)\leq&\int_{\Omega}\Phi_2(|\nabla  \xi_k|)dx-\int_{\Omega}\Phi_2(|\nabla v|)dx-\int_{\Omega}R_{v}(x,u,v)vdx\\
	&+\int_{\Omega}R_{v}(x,u,v)\xi_kdx-\int_{\Omega}g(\xi_k)(v-\xi_k)dx.
	\end{split}
	\end{align}}
By the limit \eqref{0.43}, given $\delta > 0$ there is $k_0 \in \mathbb{N}$ such that
	{\footnotesize\begin{align*}
	\dfrac{1}{a_1C}\left[\int_{\Omega}\Phi_2(|\nabla  \xi_k|)dx-\int_{\Omega}\Phi_2(|\nabla u|)dx-\int_{\Omega}R_{v}(x,u,v)vdx
	+\int_{\Omega}R_{v}(x,u,v)\xi_kdx-\int_{\Omega}g(\xi_k)(v-\xi_k)dx\right]<\dfrac{\delta}{2},
	\end{align*}}
	for each $k\geq k_0$. Hence,
	\begin{align}
	\begin{split}
	\limsup_{n\rightarrow\infty}\int_{\Omega}A(|y_n-\xi_k|)dx\leq\dfrac{\delta}{2},\;\;\text{ for all } k\geq k_0.
	\end{split}
	\end{align}
	Given $0<\varepsilon<4$, for $\delta$ sufficiently small, it follows that
	\begin{align}\label{0.44}
	\begin{split}
	\limsup_{n\rightarrow\infty}\int_{\Omega}A(|y_n-\xi_k|)dx\leq\dfrac{\varepsilon}{4},\;\;\text{ for all } k\geq k_0.
	\end{split}
	\end{align}
	Fixing $k \geq k_0$ sufficiently large such that
	\begin{align}\label{0.01}
	|\xi_k-v|_{A}<\Big(\dfrac{\varepsilon}{4}\Big)^{1/a_1}
	\end{align}
%	daí,
%	\begin{align*}
%	\int_{\Omega}A\left({|\xi_k-v|}\right)dx\leq\dfrac{\varepsilon}{4}.
%	\end{align*}
follows from $(G_1)(i)$ that
\begin{align}\label{0.45}
\int_{\Omega}A\left({|y_n-v|}\right)dx\leq C\int_{\Omega}A\left({|y_n-\xi_k|}\right)dx+C |\xi_k-v|_{A}^{a_1}\leq C\int_{\Omega}A\left({|y_n-\xi_k|}\right)dx+\dfrac{\varepsilon C}{4},
\end{align}
for some constant $C>0$ that does not depend on $n$ and $k$.
By \eqref{0.44} and \eqref{0.45}, we have
\begin{align*}
\limsup_{n\rightarrow\infty}\int_{\Omega}A\left({|y_n-v|}\right)dx<\dfrac{\varepsilon C}{2}
\end{align*}
and by the arbitrariness of $\varepsilon>0$,
\begin{align*}
\lim_{n\rightarrow\infty}\int_{\Omega}A\left({|y_n-v|}\right)dx=0.
\end{align*}
Therefore,
\begin{align*}
y_n\rightarrow v\;\;\text{ in }L^A(\Omega).
\end{align*}

Now, let us consider $a_2< \ell_2^*$, then $A$ increases essentially more slowly than $\Phi_{2*}$ near infinity. In this case, the space $W^{1,\Phi_2}_0(\Omega)$ is compactly embedded in $L^{A}(\Omega)$, therefore, the desired limit follows directly from that compact embedding.
	
\end{proof}

The following lemma is made using similar arguments to those given in Lemma \ref{0.30}. Therefore, we will omit its proof.
\begin{lemma}\label{0.54}
For $\lambda>\lambda_0$, the sequence $(y_n)$ verifies the following limit
$y_n\rightarrow v$ in $W^{1,\Phi_2}_0(\Omega)
$.	
\end{lemma}
From the above lemmas, we can conclude that
\begin{align}\label{0.02}
y_n\rightarrow v\;\text{ in }\;V_A.
\end{align}
In view of the above facts, it is possible to obtain the following result.
\begin{lemma}
For $\lambda>\lambda_0$, the pair $(u, v)$ satisfies $J'(u,v)=0$ in $X$ and $J(u, v) \neq0$.
\end{lemma}
\begin{proof}
	Fixing $k, n \in \mathbb{N}$ with $n \geq k$, we have $X_k \subset X_n$. Thus, for $(\varphi_1, \varphi_2) \in X_k$, it	follows that
	$$J'_{\lambda,n}(w_n, y_n) (\varphi_1, \varphi_2) = 0,\quad\forall n\geq k,$$
	because, by Lemma \ref{0.62}, $J'_{\lambda,n}(w_n, y_n) = 0$. Combining \eqref{0.02} with \eqref{0.03} we get
	\begin{align}\label{0.64}
		J'_{\lambda}(u, v) (\varphi_1, \varphi_2) = 0,\quad\text{ for all }(\varphi_1, \varphi_2)\in X_k.
	\end{align}
	We claim that
		\begin{align}\label{0.65}
	J'_{\lambda}(u, v) (\varphi_1, \varphi_2) = 0,\quad\text{ for all } (\varphi_1, \varphi_2)\in X.
	\end{align}
	In fact, we start observing that for all $\varphi_1\in W^{1,\Phi_1}_0(\Omega)
	$, the pair $(\varphi_1, 0) \in X_k$ for all $k$. Hence, $J'_{\lambda}(u, v)(\varphi_1, 0) = 0$. On the other hand, for $\varphi_2 \in V_A$, there exists $\chi_{n} \in V_A^{k(n)}$
	such that
	\begin{align*}
		\lim_{n\rightarrow\infty} \chi_{n}=\varphi_2, \quad\text{ in } V_A.
	\end{align*}
	From \eqref{0.64},
		\begin{align}
	J'_{\lambda}(u, v) (0, \chi_n) = 0,\quad\text{ for all }n\in \mathbb{N}.
	\end{align}
	which implies after passage to the limit as $n \rightarrow \infty$ that
		\begin{align}
	J'_{\lambda}(u, v) (0, \varphi_2) = 0,\quad\text{ for all }\varphi_2\in V_A.
	\end{align}
	Thus, \eqref{0.65} is proved.
	Using the fact that $(w_n, y_n) \rightarrow (u, v)$ in $X$ and that
	$J'_{\lambda}(w_n, y_n) \geq b_n \geq C\sigma^{\ell_2}>0, \text{ for all } n\in \mathbb{N}$, for some constant $C >0$ which does not depend on $n$, we have that $J'_{\lambda}(u, v) \geq C\sigma^{\ell_2}>0,$
	from where it follows that $(u, v)$ is a nontrivial solution for $(S_1)$, and the proof is complete.
	
\end{proof}

\section{The $N$-functions ${\Phi_1} $ and ${\Phi_2 }$ may not verify the $\Delta_{2}$-condition.}

In this section, we study the existence of solutions for the following class of quasilinear systems in Orlicz-Sobolev spaces:
\begin{equation*}
\left\{\;
\begin{aligned}
-div(\phi_1(|\nabla u|)\nabla u)&=R_u(x,u,v)\;\text{ in } \Omega& \\
-div(\phi_2(|\nabla v|)\nabla v)&=-R_v(x,u,v)\;\text{ in }  \Omega& \\
u=v&=0\;\text{ on } \partial\Omega&
\end{aligned}
\right.
\leqno{(S_2)}
\end{equation*}
where $\Omega$ is a bounded domain in $\mathbb{R}^N$($N \geq 2$) with smooth boundary $\partial \Omega$, and  $\phi_i (i=1,2):(0,\infty)\rightarrow(0,\infty)$  are two functions which satisfy:

\noindent{$(\phi_1 ')$} $\phi_i\in C^1(0,+\infty)$,  $ t\mapsto t\phi_i(t)$ are stricly increasing and $t\mapsto t^2\phi_i(t)$ is convex in $\mathbb{R}$.

\noindent{$(\phi_2 ')$} $t\phi_i(t)\rightarrow 0$ as $t\rightarrow 0$ and $t\phi_i(t)\rightarrow +\infty$ as $t\rightarrow +\infty$

\noindent{$(\phi_3 ')$} 
$\displaystyle
1<\ell_i\leq{\frac{t^{2}\phi_i(t)}{\Phi_i(t)}},$ where $\Phi_i(t)=\int_{0}^{|t|}s\phi_i(s)ds$, $t\in\mathbb{R}$.

\noindent{$(\phi_4 ')$} $\displaystyle\liminf_{t\rightarrow +\infty}\frac{\Phi_i(t)}{t^{q_i}}>0, \text{ for some } q_i>N.$
%\begin{equation*}
%t\longmapsto t\phi(t)\;\text{ is increasing for }t>0 \text{ and } \;t\longmapsto t^2\phi(t) \;\text{ is convex in } \mathbb{R}.
%\leqno{(\phi_1')}
%\end{equation*}
%\begin{equation*}
%\displaystyle\lim_{t\rightarrow0^{+}}t\phi(t)=0 \;\;\text{ and }\;\; \displaystyle\lim_{t\rightarrow+\infty}t\phi(t)=+\infty.
%\leqno{(\phi_2')}
%\end{equation*}
%\begin{equation*}
%1<\ell\leq {\dfrac{\phi(t)t^{2}}{\Phi(t)}},\;\;\forall t>0.
%\leqno{(\phi_3')}
%\end{equation*}

\noindent{$(\phi_5 ')$} $\left|1-\frac{\Phi_1(t)}{t^2\phi_1(t)}\left(1+\frac{t\phi_1'(t)}{\phi_1(t)}\right)\right|\leq 1$, $t\in\mathbb{R}$.
%\begin{equation*}
%\left|1-\dfrac{\Phi(t)}{t^2\phi(t)}\left(1+\frac{t\phi'(t)}{\phi(t)}\right)\right|\leq 1.
%\leqno{(\phi_4')}
%\end{equation*}
%\noindent{($\phi_5'$)} There is $\delta>0$ such that 
%\begin{align*}
%1+\dfrac{t\phi'(t)}{\phi(t)}\leq \dfrac{t^2\phi(t)}{\Phi(t)},\;\;\;\forall t\in (0,\delta)
%\end{align*}
%and
%\begin{align*}
%1+\dfrac{t\phi'(t)}{\phi(t)}\geq \dfrac{t^2\phi(t)}{\Phi(t)},\;\;\;\forall t\geq\delta.
%\end{align*}
%\begin{equation*}
%\displaystyle\liminf_{t\rightarrow +\infty}\dfrac{\Phi(t)}{t^p}>0\; \text{ for some } p>N.
%\leqno{(\phi_5')}
%\end{equation*}
%
%\begin{equation*}
%t\longmapsto t\psi(t)\;\text{ is increasing for }t>0 \text{ and } t\longmapsto t^2\psi(t) \text{ is convex in } \mathbb{R}.
%\leqno{(\psi_1 ')}
%\end{equation*}
%\begin{equation*}
%\displaystyle\lim_{t\rightarrow0^{+}}t\psi(t)=0 \;\;\text{ and }\;\; \displaystyle\lim_{t\rightarrow+\infty}t\psi(t)=+\infty.
%\leqno{(\psi_2 ')}
%\end{equation*}
%\begin{equation*}
%1<m\leq{\dfrac{\psi(t)t^{2}}{\Psi(t)}}, \;\;\;\forall t>0.
%\leqno{(\psi_3 ')}
%\end{equation*}
%\begin{equation*}
%\displaystyle\liminf_{t\rightarrow +\infty}\dfrac{\Psi(t)}{t^q}>0, \text{ for some } q>N.
%\leqno{(\psi_4 ')}
%\end{equation*}

The assumption $(\phi_4 ')$ implies that the embedding
\begin{align*}
W_{0}^{1,\Phi_i}(\Omega)\hookrightarrow W^{1,q_i}(\Omega)
\end{align*}
for some $q_i>N$ is continuous. Hence,
\begin{align*}
W_{0}^{1,\Phi_i}(\Omega)\hookrightarrow C^{0,\alpha_i}(\overline{\Omega})
\end{align*}
is continuous for some $\alpha_i\in(0,1)$ and
\begin{align}\label{2.0}
W_{0}^{1,\Phi_i}(\Omega)\hookrightarrow C(\overline{\Omega})
\end{align}
is compact. In what follows, we denote by $\Lambda_i>0$ the best constant that satisfies
\begin{equation}\label{0}
\lVert u\lVert_{C(\overline{\Omega})}\leq \Lambda_i \lVert u\lVert_i,\;\;\forall u\,\in\,W_{0}^{1,\Phi_i}(\Omega), 
\end{equation} 
where $\lVert\cdot\lVert_i=\lVert\nabla\cdot\lVert_{L^{\Phi_i}(\Omega)}$. 

If $d$ is twice the diameter of $\Omega$, then there exists $\delta\geq0$ such that
\begin{equation*}
\dfrac{t^2}{d^2}\leq \Phi_1
(t/d),\;\;\;\forall |t| \geq \delta
\leqno{(\phi_6')}
\end{equation*}

Before continuing this section, we would like to point out that $\Phi_1(t)=(e^{t^2}-1)/2$ and $\Phi_2(t) = |t|^p/p$ with $p> N$ satisfying $(\phi_1 ') - (\phi_6 ')$. Moreover, we would like to recall that $(u,v)\in W^{1,\Phi_1}_0(\Omega)\times W^{1,\Phi_2}_0(\Omega)$ is a weak solution of $(S_2)$ whenever
\begin{align*}
\int_{\Omega}\phi_1(|\nabla u|)\nabla u\nabla \varphi_1dx-\int_{\Omega}\phi_2(|\nabla u|)\nabla u\nabla \varphi_2dx=\int_{\Omega}R_u(x,u,v)\varphi_1dx+\int_{\Omega}R_v(x,u,v)\varphi_2dx,
\end{align*}
for all $(\varphi_1,\varphi_2)\in W^{1,\Phi_1}_0(\Omega)\times W^{1,\Phi_2}_0(\Omega)$. Here, let us consider the $R$ function satisfying the following conditions:
\\
\noindent{$(R_1 ')$} $R\in C^{1}(\overline{\Omega}\times\mathbb{R}^{2})$ and $R_v(x,u,0)\neq 0$ for all $(x, u)\in \Omega\times \mathbb{R}$.

 \noindent{$(R_2 ')$} $R(x,u,0)\leq \dfrac{1}{2
 	}\Phi_1(u/d)+\frac{1}{2d^2
 }|u|^2$, for all $(x, u)\in \Omega\times \mathbb{R}$.

 \noindent{$(R_3 ')$} $R(x,0,v)\geq- \dfrac{1}{2}\Phi_2(v/d)-Mv$, for all $(x, v)\in \Omega\times \mathbb{R}$, for some constant $M>0$.

\noindent{$(R_4 ')$} There are $\nu >0$, $\mu>1$ and $0<\beta<1$ such that
\begin{equation*}
\dfrac{1}{\mu}h(u)R_u(x,u,v)u+\dfrac{1}{\nu}R_v(x,u,v)v-R(x,u,v)\geq 0,\;\;\;\forall(x,u,v)\in \Omega\times\mathbb{R}^{2}
\leqno{(i)}
\end{equation*}
and
\begin{equation*}
\beta R(x,u,v)-\dfrac{1}{\mu}h(u)R_u(x,u,v)u\geq 0,\;\;\;\forall(x,u,v)\in \Omega\times\mathbb{R}^{2}
\leqno{(ii)}
\end{equation*}
where $h(u)=\frac{\Phi_1(u)}{u^2\phi_1(u)}$.

The main result of this section is the following.
	\begin{theorem}\label{teo2}
	Assume that $(\phi_1')-(\phi_6')$ and $(R_1)-(R_4)$ hold. Then, the system $(S_2)$ possesses a nontrivial solution.
	\end{theorem}

We observe that $R(u,v)=\Phi_1(u)^\sigma\Phi_2(v)^\theta+v^+$ satisfies $(R_1 ')-(R_4 ')$ for some $\theta,\sigma>1$, where $v^+ := \max\{0, v\}$.

Under the assumptions $(\phi_1')-(\phi_6')$ it is well known in the literature that the $N$-functions $\Phi_1$ and $ \Phi_2 $ might not satisfy the $\Delta_{2}$-condition, and as a consequence, $W^{1,\Phi_1}_0(\Omega)$ and $W^{1,\Phi_2}_0(\Omega)$ might not be reflexive anymore. Another important fact we can highlight is that under these conditions, it is well known that there are $u\in W^{1,\Phi_1}_0(\Omega)$ and $v\in W^{1,\Phi_2 }_0(\Omega)$ such that
\begin{align*}
\int_{\Omega}\Phi_1(|\nabla u|)dx=\infty\;\;\text{ and }\;\;\int_{\Omega}\Phi_2(|\nabla v|)dx=\infty
\end{align*}
In order to avoid this problem, we will work with the space  $W^{1}_0E^{\Phi_1}(\Omega)\times W^{1}_0E^{\Phi_2}(\Omega)$, because in this
space the functional $Q:W^{1}_0E^{\Phi_1}(\Omega)\times W^{1}_0E^{\Phi_2}(\Omega)\longrightarrow \mathbb{R}$ given by
\begin{align*}
	Q(u,v)=\int_{\Omega}\Phi_1(|\nabla u|)dx-\int_{\Omega}\Phi_2(|\nabla v|)dx
\end{align*}
belongs to $C^1(W^{1}_0E^{\Phi_1}(\Omega)\times W^{1}_0E^{\Phi_2}(\Omega),\mathbb{R})$. However, independent of $\Delta_{2}$-condition, the embedding \eqref{2.0} guarantee that the funcional $H:W^{1,\Phi_1}_0(\Omega)\times W^{1,\Phi_2}_0(\Omega)\longrightarrow\mathbb{R}$ given by
\begin{align*}
H(u,v)=\int_{\Omega}R(x,u,v)dx
\end{align*}
belongs to $C^1(W^{1,\Phi_1}_0(\Omega)\times W^{1,\Phi_2}_0(\Omega),\mathbb{R})$. In particular, $H|_{W^{1}_0E^{\Phi_1}(\Omega)\times W^{1}_0E^{\Phi_2}(\Omega)}$ is also of class $C^1 $. That is, the energy functional  $J:W^{1}_0E^{\Phi_1}(\Omega)\times W^{1}_0E^{\Phi_2}(\Omega)\longrightarrow \mathbb{R}$ associated to the system $(S_2)$ given by
\begin{align*}
	J(u,v)=\int_{\Omega}\Phi_1(|\nabla u|)dx-\int_{\Omega}\Phi_2(|\nabla v|)dx-\int_{\Omega}R(x,u,v)dx
\end{align*}
belongs to $C^1(W^{1}_0E^{\Phi_1}(\Omega)\times W^{1}_0E^{\Phi_2}(\Omega),\mathbb{R})$.

In order to apply the Saddle-point theorem, in the next one we fix some notations. Since $W^{1}_0 E^{\Phi_2}(\Omega)$ is separable, there exists a sequence $(e_n)\subset W^{1}_0 E^{\Phi_2}(\Omega)$ such that 
\begin{align}\label{2.1}
W^{1}_0 E^{\Phi_2}(\Omega)=\overline{span\{e_n:n\in\mathbb{N}\}}.
\end{align}
Hereafter, for each $n\in\mathbb{N}$ we denote by $V_{n} $, $X_n$ and $X_n'$ the following spaces
\begin{align*}
V_{n} ={span\{e_j:j=1,\cdots,n\}},\quad X_n=W^{1}_0 E^{\Phi_1}(\Omega)\times V_{n} \quad\text{ and }\quad X_n'=W^{1,\Phi_1}_{0}(\Omega)\times V_{n}.
\end{align*}
The restriction of $J$ to $X_n$ will be denoted by $J_{n}$. Then $J_n:X_n\longrightarrow\mathbb{R}$ is the functional given by 
\begin{align*}
J_n(u,v)=\int_{\Omega}\Phi_1(|\nabla u|)dx-\int_{\Omega}\Phi_2(|\nabla v|)dx-\int_{\Omega}R(x,u,v)dx.
\end{align*}
From the regularity of $J$, it follows that $J_n$ belongs to $ C^{1}(X_n,\mathbb{R})$ with
{\small\begin{align*}
J_n'(u,v)(w_1,w_2)=\int_{\Omega}\phi_1(|\nabla u|)\nabla u\nabla w_1dx-\int_{\Omega}\phi_2(|\nabla v|)\nabla v\nabla w_2dx-\int_{\Omega}R_u(x,u,v)w_1dx-\int_{\Omega}R_v(x,u,v)w_2dx,
\end{align*}}
for all $(w_1,w_2)\in X_n.$

In the following, we prove that $J_n$ satisfies the hypotheses of Theorem \ref{saddle}.

\begin{lemma}\label{lemma2}
	Under the space $Z=W^{1}_0E^{\Phi_1}(\Omega)\times\{0\}$ the functional $J_n$ is bounded below.
\end{lemma}

\begin{proof}
%Recorde que 
%\begin{align*}
%\lVert u\lVert_{\infty}\leq \Lambda_1\lVert u\lVert_{1},\;\;\;\forall u\in W^{1}_{0}E^{\Phi}(\Omega).
%\end{align*}
By the condition $(R_2 ')$,
	\begin{align}\label{2.61}
	J_{n}(u,0)
	\geq\int_{\Omega} \Phi_1(|\nabla u|)dx-\dfrac{1}{2}\int_{\Omega} \Phi_1(|u|/d)dx-\dfrac{1}{2d^2}\int_{\Omega} |u|^2
	dx.
	\end{align}
It is worth remembering here the Poincaré inequality
		\begin{align*}
\int_{\Omega}\Phi_1(|u|/d)dx\leq \int_{\Omega}\Phi_1(|\nabla u|)dx, \; \forall u\in W^{1}_{0}E^{\Phi_1}(\Omega).
	\end{align*}
	For more details on this inequality, we infer the reader to \cite{JP}.
Hence, using the Poincaré inequality together with the hypothesis $(\phi_{6}')$ on the inequality \eqref{2.61}, we obtain
	\begin{align*}
	J_{n}(u,0)
	\geq& -\dfrac{1}{2d^2}\int_{[|u|\leq\delta ]} |u|^2\geq -\dfrac{\delta^2}{2d^2}|\Omega|,\;\;\forall u\in W^{1}_{0}E^{\Phi_1}(\Omega).
	\end{align*}
This finishes the proof.

\end{proof}

\begin{lemma}\label{lemma21} If $\lVert v\lVert_2\rightarrow\infty$, then  $J(0,v)\rightarrow -\infty$.
\end{lemma}

\begin{proof}	
	
Let $v \in W^{1}_{0}E^{\Phi_1}(\Omega)$ with $\lVert v\lVert_{1}\geq 1$. The assumption $(R_3 ')$ together with the Poincaré inequality implies that
	\begin{align}\label{2.5}
	J(0,v)
\leq-\frac{1}{2}\int_{\Omega} \Phi_2(|\nabla v|)dx+M\int_{\Omega} {|v|}dx
 	\end{align}	

From $(\phi_3 ')$,
$$\frac{d}{ds}ln(\Phi_2(rs))=\frac{\psi_2(rs)r^2s}{\Phi_2(rs)}\geq \frac{\ell_2}{s},\;\;\forall s,r>0$$
thus,
$$\int_{1}^{t} \frac{d}{ds}ln(\Phi_2(rs))ds\geq {\ell_2}\int_{1}^{t}\frac{1}{s} ds, \;\;\;\forall t\geq1.$$
Therefore,
$$ln\frac{\Phi_2(rt)}{\Phi_2(r)}\geq ln(t^{\ell_2}),\;\;\;\forall t\geq1.$$
Because of the monotonicity of the logarithmic function,
$$\frac{\Phi_2(rt)}{\Phi_2(r)}\geq t^{\ell_2},\;\;\;\forall t\geq1.$$
And as a consequence of this inequality, we have
\begin{align}\label{7}
\int_{\Omega} \Phi_2(|\nabla v|)dx\geq \lVert v\lVert_{2}^{\ell_2} \;\;\text{ for } \;\lVert v\lVert_{2}\geq1.
\end{align}
 	
By means of the inequalities \eqref{2.5} and \eqref{7}, we conclude that
\begin{align*}
J(0,v)\leq-\lVert v\lVert_{2}^{\ell_2}+M|\Omega|\Lambda_2\lVert v\lVert_{2}.
\end{align*}
Since $1<\ell_2$, the result follows.
\end{proof}

\begin{corollary}\label{lemma22} If $\lVert v\lVert_2\rightarrow\infty$, then  $J_n(0,v)\rightarrow -\infty$.

\end{corollary}

\begin{corollary}\label{cr1}
There is $M>0$ such that $\displaystyle\inf_{Z}J_n>\max_{\partial\mathcal{M}_n}J_n:=b_n$ where $\mathcal{M}_n=B_M(0)\cap Y_n,$.
\end{corollary}

\begin{proof}
By the above Corollary $J_n(0,v)\rightarrow-\infty$ as $ \lVert v\lVert_2\rightarrow +\infty $ in  $Y$, then, fix $M>1$ such that
$\displaystyle J_n(0,v)<\inf_{Z}J_n$ for $\lVert v\lVert_2=M$ and $v\in Y_n$. Since $dim Y_n<\infty$, we can conclude
$\displaystyle\inf_{Z}J_n>\max_{\mathcal{N}_n}J_n$.
	
\end{proof}

%\begin{corollary}\label{2.24}
%	Existe $M>0$ tal que $\displaystyle d_n=\sup_{\mathcal{M}}J_n \leq M$, para todo $n\in\mathbb{N}$.
%\end{corollary}

Then, by Lemma \ref{lemma1}, we can apply the Saddle-point theorem \ref{saddle} to functional $J_{n}$ using the sets
	$$ Y_n=\{0\}\times V_{n},\quad Z=W^{1}_0E^{\Phi_1}(\Omega)\times\{0\},\quad\text{ and }\quad \mathcal{M}_n=B_M(0)\cap Y_n,$$	
where $M>0$ is obtained from Corollary \ref{cr1}.
Then, there exists a sequence $(u_k, v_k)\subset X_n$ with 
	\begin{align}\label{4.03}
	J_{n}(u_k,v_k)\longrightarrow c_{n}\;\;\text{ and }\;\;	J_{n}'(u_k,v_k)\longrightarrow 0\;\text{as } k\rightarrow+\infty.
	\end{align} 
where 
\begin{align}\label{2.511}
c_{n}=\inf_{\gamma\in \Gamma}\max_{u\in\mathcal{M}_n } J_{n} (\gamma(u)),
\end{align}
with
\begin{align*}
\Gamma=\{\gamma\in C(\mathcal{M}_n, X_n) :\;\gamma|_{\mathcal{N}_n }=Id\}.
\end{align*}

%\textcolor{red}{Uma das grandes dificuldades em se trabalhar com modelos de $N$-funções do tipo $\Phi(t)=e^{t^2} -1$ é provar que sequência acima é limitada. 
%Aqui, preciso mencionar que nem sempre vale o argumento usado por Alves, porém adaptamos alguns argumentos apresentados por eles para fazer com que o próximo lemma se torne verdadeiro}

\begin{lemma}\label{4.02}
	The sequence $(u_k,v_k)$ is bounded in $X_n$.
\end{lemma}

\begin{proof}Define the function
		\begin{equation*}
	\eta(t)=	\left\{\;
	\begin{aligned}
	\frac{\Phi_1(t)}{t\phi_1(t)}\;&\text{ if } t>0& \\
0\;\;\;\;\;&\text{ if }  t=0&
	\end{aligned}
	\right.
	\end{equation*}
and consider the sequence
	\begin{align*}
		g_k(x)=\eta(u_k(x)),\;\;x\in\Omega.
	\end{align*}
	A direct computation leads to
	\begin{align*}
		\nabla g_k=\left[1-\dfrac{\Phi_1(u_k)}{u_k^2\phi_1(u_k)}\Big(1+\dfrac{u_k\phi_1 '(u_k)}{\phi_1(u_k)}\Big)\right]\nabla u_k.
	\end{align*}
Furthermore, considering the hypothesis $(\phi_{4}')$, it is shown without difficulty that $g_k\in W^{1}_{0}E^{\Phi_1}(\Omega)$ and $\lVert g_k\lVert_{1} \leq \lVert u_k\lVert_ {1}$ for each $k\in \mathbb{N}.$
Being $(u_k,v_k)$ a sequence $(PS)_{c_n}$, then by $(R_4 ')(i)$ and $(\phi_3 ')$
	\begin{align}\label{2.6}
		\begin{split}
		c_n+1+o_k(1)\lVert(u_k,v_k)\lVert
%		=&	c_n+1+o_k(1)\sqrt{\lVert u_k\lVert_1 ^2+\lVert v_k\lVert_1 ^2}\\
%		&\geq 	c_n+1+o_k(1)\sqrt{\lVert g_k\lVert_1 ^2+\lVert v_k\lVert_1 ^2}\\
%		&=c_n+1+o_k(1)\lVert(g_k,v_k)\lVert\\
		\geq& J_{n}(u_k,v_k)-J_{n}'(u_k,v_k)\big(\frac{1}{\mu}g_k,\frac{1}{\nu} v_k\big)\\
		=\int_{\Omega}\Phi(|\nabla u_k|)dx-&\dfrac{1}{\mu}\int_{\Omega}\phi(|\nabla u_k|)|\nabla u_k|^2S(u_k)dx-\int_{\Omega}\Psi(|\nabla v_k|)dx+\dfrac{1}{\nu}\int_{\Omega}\psi(|\nabla v_k|)|\nabla v_k|^2dx\\
		+\dfrac{1}{\mu}\int_{\Omega}R_u(x,u_k,&v_k)u_kh(u_k)dx+\dfrac{1}{\nu}\int_{\Omega}R_v(x,u_k,v_k)v_kdx-\int_{\Omega}R(x,u_k,v_k)dx	\\
		\geq \int_{\Omega}\Phi_1(|\nabla u_k|)dx&-\dfrac{1}{\mu}\int_{\Omega}\phi_1(|\nabla u_k|)|\nabla u_k|^2S(u_k)dx+\left(\dfrac{\ell_2}{\nu}-1\right)\int_{\Omega}\Phi_2(|\nabla v_k|)dx.
		\end{split}
	\end{align}
where $h(t)=\frac{\Phi_1(t)}{t^2\phi_1(t)}$ and $S(t)=1-\frac{\Phi_1(t)}{t^2\phi_1(t)}\Big(1+\frac{t\phi_1 '(t)}{\phi_1(t)}\Big)$.(The functions $S$ and $h$ were introduced by Alves et al in \cite{alves}) On the other hand, it follows from $(R_4 ')(ii)$ that
		\begin{align*}
	\begin{split}
	c_n+1+o_k(1)\lVert g_k\lVert_1\geq& -\beta J_{n}(u_k,v_k)+J_{n}'(u_k,v_k)\big(\frac{1}{\mu}g_k,0\big)\\
	=&-\beta\int_{\Omega}\Phi_1(|\nabla u_k|)dx+\dfrac{1}{\mu}\int_{\Omega}\phi_1(|\nabla u_k|)|\nabla u_k|^2S(u_k)dx+\beta\int_{\Omega}\Phi_2(|\nabla v_k|)dx\\
	&-\dfrac{1}{\mu}\int_{\Omega}R_u(x,u_k,v_k)u_kh(u_k)dx+\beta\int_{\Omega}R(x,u_k,v_k)dx,\\
	\geq&-\beta\int_{\Omega}\Phi_1(|\nabla u_k|)dx+\dfrac{1}{\mu}\int_{\Omega}\phi_1(|\nabla u_k|)|\nabla u_k|^2S(u_k)dx,
	\end{split}
	\end{align*}
i.e,
\begin{align}\label{2.7}
-\dfrac{1}{\mu}\int_{\Omega}\phi_1(|\nabla u_k|)|\nabla u_k|^2S(u_k)dx\geq -c_n-1-o_k(1)\lVert u_k \lVert_1-\beta\int_{\Omega}\Phi_1(|\nabla u_k|)dx.
\end{align}
From \eqref{2.6} and \eqref{2.7},
\begin{align*}
2(c_n+1)+o_k(1)\lVert(u_k,v_k)\lVert\geq (1-\beta)\int_{\Omega}\Phi_1(|\nabla u_k|)dx+\left(\dfrac{ \ell_2}{\nu}-1\right)\int_{\Omega}\Phi_2(|\nabla v_k|)dx.
\end{align*}
%	Como $\lVert(u_k,v_k)\lVert\geq \lVert u_k\lVert_1,$ para todo $k\in\mathbb{N}$, segue que
%\begin{align*}
%	2(c_n+1)+o_k(1)\lVert(u_k,v_k)\lVert\geq (1-\lambda)\int_{\Omega}\Phi(|\nabla u_k|)dx+\left(\dfrac{m}{\nu}-1\right)\int_{\Omega}\Psi(|\nabla v_k|)dx.
%	\end{align*}
 Suppose for contradiction that, up to a subsequence, $\lVert(u_k,v_k)\lVert \rightarrow +\infty$ as $k\rightarrow +\infty$. This way, we need to study the following situations:

\noindent{$(i)$}  $\lVert u_k\lVert_{1}\rightarrow +\infty$ and $\lVert v_k\lVert_{2}\rightarrow \infty$

\noindent{$(ii)$} $\lVert u_k\lVert_{1}\rightarrow +\infty$ and $\lVert v_k\lVert_{2}$ is bounded

\noindent{$(iii)$}  $\lVert v_k\lVert_{2}\rightarrow \infty$ and $\lVert u_k\lVert_{1}$ is bounded

In the first case, there is $k_0\in\mathbb{N}$ such that
\begin{align*}
\int_{\Omega}\Phi_1(|\nabla u_k|)dx\geq \lVert u_k\lVert_{1}\quad\text{ and }\quad\int_{\Omega}\Phi_2(|\nabla v_k|)dx\geq \lVert v_k\lVert_{2},\;\forall k\geq k_0.
\end{align*}
Hence, the inequality \eqref{21} implies that
\begin{align*}
2c_{n} ^2+o_k(1)\lVert(u_k,v_k)\lVert ^2\geq\left(1-\beta\right)^2 \lVert u_k\lVert_{1}^{2} +\left(\dfrac{\ell_2}{\nu}-1\right)^2\lVert v_k\lVert_{2}^{2},\;\forall k\geq k_0.
\end{align*}
Which is absurd.

In case $(ii)$, there is $k_0\in\mathbb{N}$ such that
\begin{align*}
\int_{\Omega}\Phi_1(|\nabla u_k|)dx\geq \lVert u_k\lVert_{1},\;\forall k\geq k_0.
\end{align*}
Thus, the inequality \eqref{21} is reduced to
\begin{align*}
2c_{n} ^2+C_1+o_k(1)\lVert u_k\lVert_{1}\geq\left(1-\beta\right)^2 \lVert u_k\lVert_{1}^{2},\;\forall k\geq k_0.
\end{align*}
which is absurd. The last case is similar to the case $(ii)$. The above analysis shows that $(u_k,v_k)$ is now a bounded sequence in $X_n$.

\end{proof}

From Lemmas \ref{4.02} and \ref{10.6}, we may assume that there exists a subsequence of $(u_k, v_k)$,
still denoted by itself, and  $(w_n, y_n)\in X_n '$ such that
\begin{align}\label{4.01}
u_k \xrightharpoonup[\quad]{\ast}w_n\;\; \text{weakly\;in }& W^{1,\Phi_1}_0(\Omega) \;\;\text{ and }\;\;v_k \xrightharpoonup[\quad]{\ast}y_n\;\; \text{weakly\;in } V_n, \;\;\text{ as }\;\;k\rightarrow\infty.
\end{align}

Here, we highlight that the pair $(w_n, y_n)$ may not belong to the space $X_n$, because whenever $\Phi_{1}$ does not satisfy the $\Delta_{2}$-condition the space $X_n$ is not a weak$^*$ closed subspace of $X_n '$.

The results below will be used to ensure that the sequence $(w_n,y_n)$ is bounded in $ W^{1,\Phi_1}_0(\Omega)\times W^{1,\Phi_2}_0(\Omega) $, moreover, we will do some results that will be fundamental.
\begin{lemma}\label{2.9}
	The sequence $(u_k,v_k)$ obtained in \eqref{4.03} satisfies
	\begin{align*}
	\int_{\Omega}\phi_1(|\nabla u_k|)\nabla u_k\nabla \varphi dx=\int_{\Omega}R_u(x,u_k,v_k)\varphi dx+o_k(1),\;\;\forall k \in\mathbb{N}\;\text{ and }\; \varphi\in W^{1,\Phi_1}_0(\Omega).
	\end{align*}
\end{lemma}
\begin{proof}From \ref{4.03},
	\begin{align}\label{2.8}
			J_{n}'(u_k,v_k)(\varphi,0)=o_k(1)\lVert \varphi\lVert_{1},\;\;\forall\varphi\in W^{1}_0E^{\Phi_1}(\Omega).
	\end{align}
By definition, the space $W^{1,\Phi_1}_0(\Omega)$ is the weak$^*$ closure of $C^{\infty}_0(\Omega)$ in $W^{1,\Phi_1}(\Omega)$, thus, given $\varphi\in W^{1,\Phi_1}_0(\Omega)$ there will be a sequence $(\varphi_m)$ in $C^{\infty}_0( \Omega)$ such that
\begin{align}\label{4.55}
\varphi_m \xrightharpoonup[\quad]{\ast}\varphi\; \text{\;in } W^{1,\Phi_1}_0(\Omega).
\end{align}
It is clear that $(\lVert \varphi_m\lVert_{1})$ is bounded in $\mathbb{R},$ so by \eqref{2.8},
		\begin{align*}
o_k(1)=	\int_{\Omega}\phi_1(|\nabla u_k|)\nabla u_k\nabla \varphi_mdx-\int_{\Omega}R_u(x,u_k,v_k)\varphi_mdx,\;\;\forall k\in\mathbb{N}
	\end{align*}
Using the fact that $ \phi_1(|\nabla u_k|)\frac{\partial u_k}{\partial x_i}\in E^{\tilde{\Phi}_1}(\Omega)$ along with the limit \eqref {4.55}, we will get
	\begin{align*}
\lim_{m\rightarrow \infty}\int_{\Omega}\phi_1(|\nabla u_k|)\nabla u_k\nabla \varphi_mdx=\int_{\Omega}\phi_1(|\nabla u_k|)\nabla u_k\nabla \varphi dx
\end{align*}
Therefore, since the spaces $ W^{1,\Phi_1}_0(\Omega)$, $W^{1,\Phi_2}_0(\Omega)$ are embedded in $ C(\overline{\Omega})$, we can conclude that
	\begin{align*}
	o_k(1)=	\int_{\Omega}\phi_1(|\nabla u_k|)\nabla u_k\nabla \varphi dx-\int_{\Omega}R_u(x,u_k,v_k)\varphi dx,\;\;\forall k\in\mathbb{N}.
	\end{align*}
\end{proof}

Before proceeding with the results, we need to make the following definitions:
\begin{itemize}
\item We will denote by $D(J_{\Phi_i})\subset W^{1,\Phi_i}_0(\Omega)$, the following set:
\begin{align*}
	D(J_{\Phi_i})=\left\{u\in W^{1,\Phi_i}_0(\Omega): \int_{\Omega}\Phi_i(|\nabla u|)dx<+\infty\right\}
\end{align*}
\item  We will denote by  $dom(\phi(t)t)\subset W^{1,\Phi}_0(\Omega)$, the following set:
\begin{align*}
dom(\phi_i(t)t)=\left\{u\in W^{1,\Phi_i}_0(\Omega): \int_{\Omega}\tilde{\Phi}_i\big(\phi_i(|\nabla u|)|\nabla u|\big)dx<+\infty\right\}
\end{align*}
\end{itemize}

\begin{lemma}\label{2.21}Let $(w_n)$ the sequence obtained in \eqref{4.01}. Then
	$(w_n)\subset D(J_{\Phi_1})\cap dom(\phi_1(t)t)$, furthermore, 
	\begin{align*}
	c_n=\lim_{k\rightarrow\infty} J_n(u_k, v_k)=J_n(w_n, y_n)
	\end{align*}
	and
		\begin{align}\label{2.60}
\begin{split}
\int_{\Omega} \Phi_1(|\nabla \varphi_1|)dx-\int_{\Omega} &\Phi_1(|\nabla w_n|)dx-\int_{\Omega} \phi_2(|\nabla y_n|)\nabla y_n \nabla(\varphi_2 -y_n)dx\\
\geq&\int_{\Omega}R_u(x,w_n,y_n)(\varphi_1 -w_n) dx+ \int_{\Omega}R_v(x,w_n,y_n)(\varphi_2 -y_n) dx,
\end{split}
	\end{align}
	for all $(\varphi_1,\varphi_2)\in W^{1,\Phi_1}_0(\Omega) \times V_n.$
\end{lemma}
\begin{proof}
%As $u_k \rightarrow w_n$ and $v_k \rightarrow y_n$ in $C(\overline{\Omega})$, then
%	\begin{align}\label{2.11}
%		\int_{\Omega}R_u(x,u_k,v_k)\varphi_1 dx\longrightarrow \int_{\Omega}R_u(x,w_n,y_n)\varphi_1 dx,\;\;\forall g_1\in  W^{1,\Phi_1}(\Omega),
%	\end{align}
%		\begin{align}
%	\int_{\Omega}R_v(x,u_k,v_k)\varphi_2 dx\longrightarrow \int_{\Omega}R_v(x,w_n,y_n)\varphi_2 dx,\;\;\forall g_2\in  V_n,
%	\end{align}
%	and 
%	\begin{align}
%	\int_{\Omega}R(x,u_k,v_k) dx\longrightarrow \int_{\Omega}R(x,w_n,y_n)dx.
%	\end{align}
%	Since  $(J_n(u_k,v_k))_{k\in\mathbb{N}}$ is bounded, we shall suppose that for some subsequence, the sequence $\Big(\int_{\Omega} \Phi_1(|\nabla u_k|)dx\Big)$ has limit which will be denoted by  $L$, that is,
%	\begin{align*}
%		\lim_{k\rightarrow\infty}\int_{\Omega}\Phi_1(|\nabla u_k|)dx=L.
%	\end{align*}
%	As $u_k \xrightharpoonup[\quad]{\ast}w_n$ in $W^{1,\Phi_1}_0(\Omega)$, follows from [\citenum{ET}, Theorem $2.1$, Chapter 8] that
%	\begin{align}\label{2.12}
%		\int_{\Omega}\Phi_1(|\nabla w_n|)dx\leq \lim_{k\rightarrow\infty}\int_{\Omega}\Phi_1(|\nabla u_k|)dx=L,
%	\end{align}
%Therefore $w_n\in D(J_{\Phi_1})$. Showing the first part of Lemma.                     
	
Using the fact that $J_n'(u_k,v_k) \rightarrow 0$ as $k\rightarrow\infty$ together with Lemma \ref{2.9}, we can conclude that
	\begin{align}\label{2.10}\begin{split}
\int_{\Omega}\phi_1(|\nabla u_k|)\nabla u_k\nabla \varphi_1dx-	\int_{\Omega}\phi_2(|\nabla v_k|)\nabla v_k\nabla \varphi_2dx=&\int_{\Omega}R_u(x,u_k,v_k)\varphi_1dx\\&+\int_{\Omega}R_v(x,u_k,v_k)\varphi_2dx+o_k(1),
	\end{split}
	\end{align}
	for each $(\varphi_1,\varphi_2)\in W^{1,\Phi_1} _0(\Omega)\times V_n$ and  $k\in\mathbb{N}$. Since $\Phi_1$ is convex, we have
		\begin{align*}
	\int_{\Omega} \Phi_1(|\nabla \eta_1|)dx-\int_{\Omega} &\Phi_1(|\nabla u_k|)dx\geq \int_{\Omega}\phi_1(|\nabla u_k|)\nabla u_k \nabla(\eta_1-u_k) dx,
	\end{align*}
for all $\eta_1\in W^{1,\Phi_1}_0(\Omega).$ Hence, considering $\varphi_1=\eta_1-u_k$ in \eqref{2.10} and using the inequality above, we get
		\begin{align}\label{2.13}
\begin{split}
	\int_{\Omega} \Phi_1(|\nabla \eta_1|)dx-\int_{\Omega} &\Phi_1(|\nabla u_k|)dx-\int_{\Omega} \phi_2(|\nabla v_k|)\nabla v_k \nabla \varphi_2dx\\
\geq&\int_{\Omega}R_u(x,u_k,v_k)(\eta_1 -u_k) dx+ \int_{\Omega}R_v(x,u_k,v_k)\varphi_2 dx+o_k(1),
\end{split}
	\end{align}
		for every $(\eta_1,\varphi_2)\in W^{1,\Phi_1}_0(\Omega) \times V_n$ and $k\in\mathbb{N}$. Since $u_k\xrightharpoonup[\quad]{\ast} w_n$ in $W^{1,\Phi_1}_0(\Omega)$ follows from [\citenum{ET}, Theorem $2.1$, Chapter 8] that
		\begin{align}\label{2.12}
		\int_{\Omega}\Phi_1(|\nabla w_n|)dx\leq \lim_{k\rightarrow\infty}\int_{\Omega}\Phi_1(|\nabla u_k|)dx,
		\end{align}
	
		Remember that $dim V_n=n$, so $v_k \rightarrow y_n$ in $V_n$. Hence,
		\begin{align}
	\begin{split}
	\int_{\Omega} \Phi_1(|\nabla \eta_1|)dx-\int_{\Omega} \Phi_1(|\nabla w_n|&)dx-\int_{\Omega} \phi_2(|\nabla y_n|)\nabla y_n\nabla \varphi_2dx\\
	\geq&\int_{\Omega}R_u(x,w_n,y_n)(\eta_1 -w_n) dx+ \int_{\Omega}R_v(x,w_n,y_n)\varphi_2 dx,
	\end{split}
	\end{align}
	for each $(\eta_1,\varphi_2)\in W^{1,\Phi_1}_0(\Omega) \times V_n$. Justifying the inequality \eqref{2.60}.

Considering $(\eta_1,\varphi_2)=(w_n,0)$ in the inequality \eqref{2.13}, we get
		\begin{align*}
	\begin{split}
	\int_{\Omega} \Phi_1(|\nabla w_n|)dx-\int_{\Omega} \Phi_1(|\nabla u_k|)dx
	\geq\int_{\Omega}R_u(x,u_k,v_k)(w_n -u_k) dx+o_k(1).
	\end{split}
	\end{align*}	
	Thus,
			\begin{align}\label{2.14}
		\int_{\Omega}\Phi_1(|\nabla w_n|)dx\geq \lim_{k\rightarrow\infty}\int_{\Omega}\Phi_1(|\nabla u_k|)dx.
		\end{align}
	Combining \eqref{2.12} and \eqref{2.14},
			\begin{align*}
		 \lim_{k\rightarrow\infty}\int_{\Omega}\Phi_1(|\nabla u_k|)dx=\int_{\Omega}\Phi_1(|\nabla w_n|)dx.
		\end{align*}
	Therefore, we can conclude that
		\begin{align*}
			c_n=\lim_{k\rightarrow\infty} J_n(u_k,v_k)= J_n (w_n,y_n).
		\end{align*}
		
	Finally, we will show that $w_n\in dom(\phi_1(t)t)$. By the inequality \eqref{2.13},
				\begin{align*}
		\int_{\Omega} \Phi_1(|\nabla u_k-\frac{1}{k}\nabla u_k|)dx-\int_{\Omega}\Phi_1(|\nabla u_k|)dx
		\geq-\dfrac{1}{k}\int_{\Omega}R_u(x,u_k,v_k)u_k dx+o_k(1),
		\end{align*}
	i.e,
			\begin{align*}
		\int_{\Omega}\dfrac{\big( \Phi_1(|\nabla u_k-\frac{1}{k}\nabla u_k|)- \Phi_1(|\nabla u_k|)\big)}{-\frac{1}{k}}dx\leq\int_{\Omega}R_u(x,u_k,v_k)u_k dx+o_k(1).
		\end{align*}
As $(u_k)$ and $(v_k)$ are bounded in $W^{1,\Phi_1}_0(\Omega)$ and $W^{1,\Phi_2}_0(\Omega)$, respectively, there will be $M>0$ such that
			\begin{align*}
		\int_{\Omega} \Phi_1(|\nabla u_k-\frac{1}{k}\nabla u_k|)dx-\int_{\Omega}\Phi_1(|\nabla u_k|)dx
		\leq M,\;\;\forall k\in\mathbb{N}. 
		\end{align*}
		Since $\Phi_1$ is in $ C^{1}$ class, there exists $\theta_k(x)\in [0,1]$ such that
		\begin{align*}
			 \dfrac{ \Phi_1(|\nabla u_k-\frac{1}{k}\nabla u_k|)- \Phi_1(|\nabla u_k|)}{-\frac{1}{k}}=\phi_1(|\big(1-\frac{\theta_k}{k}(x)\big)\nabla u_k|)\big(1-\frac{\theta_k(x)}{k}\big)|\nabla u_k|^2.
		\end{align*}
		Recalling that $0<1-\frac{\theta_k(x)}{k}\leq 1$, we know that $1-\frac{\theta_k(x)}{k}\geq \big(1-\frac{\theta_k(x)}{k}\big)^2$ which leads to
		\begin{align*}
			\int_{\Omega}\phi_1(|\big(1-\frac{\theta_k}{k}(x)\big)\nabla u_k|)\big(1-\frac{\theta_k(x)}{k}\big)^2|\nabla u_k|^2dx\leq M,\;\;\forall k\in\mathbb{N}. 
		\end{align*}
	As $\nabla u_k \xrightharpoonup[\quad]{\ast}\nabla w_n$ in $\big(L^{\Phi_1}(\Omega)\big)^{N-1}$, we also have $ \big (1-\frac{\theta_k(x)}{k}\big)\nabla u_k \xrightharpoonup[\quad]{\ast}\nabla w_n$ in $\big(L^{\Phi_1}(\Omega) \big)^{N-1}$ as $k\rightarrow \infty$. Then, by
	using the fact that $\phi_1(t)t^2$ is convex, we can apply [\citenum{ET}, Theorem $2.1$, Chapter 8] to get
			\begin{align*}
		\liminf_{k\rightarrow\infty}\int_{\Omega}\phi_1(|\big(1-\frac{\theta_k}{k}(x)\big)\nabla u_k|)\big(1-\frac{\theta_k(x)}{k}\big)^2|\nabla u_k|^2\geq \int_{\Omega}\phi_1(|\nabla w_n|)|w_n|^2 dx
		\end{align*}
	and so,
		\begin{align*}
			\int_{\Omega}\phi_1(|\nabla w_n|)|w_n|^2 dx\leq M.
		\end{align*}
	Recalling that
		\begin{align*}
			\phi_1(t)t^2=\Phi_1(t)+\tilde{\Phi }_1(\phi_1(t)t),\;\;\forall t\in\mathbb{R}
		\end{align*}
	we have
		\begin{align*}
			\phi_1(|\nabla w_n|)|\nabla w_n|^2=\Phi_1(|\nabla w_n|)+\tilde{\Phi }_1(\phi_1(|\nabla w_n|)|\nabla w_n|)
		\end{align*}
	which leads to
		\begin{align*}
		\int_{\Omega}\phi_1(|\nabla w_n|)|\nabla w_n|^2dx=\int_{\Omega}\Phi_1(|\nabla w_n|)dx+\int_{\Omega}\tilde{\Phi }_1(\phi_1(|\nabla w_n|)|\nabla w_n|^2)dx.
		\end{align*}	
Since $\int_{\Omega}\phi_1(|\nabla w_n|)|\nabla w_n|^2dx$ is finite, we see that $\int_{\Omega}\Phi_1(|\nabla w_n|)dx$ and $\int_{\Omega}\tilde{\Phi} _1(\phi_1(|\nabla w_n|)|\nabla w_n|^2)dx$ are also finite, showing that $w_n\in D(J_{\Phi_{1}})\cap dom (\phi_1(t)t)$. This finishes the proof.
		
\end{proof}

\begin{lemma}\label{2.20} For each $(\varphi_1,\varphi_2)\in W^{1,\Phi_1}_0(\Omega)\times V_n$, the following equality holds
		\begin{align*}
	\int_{\Omega}\phi_1(|\nabla w_n|)\nabla w_n \nabla \varphi_1dx-\int_{\Omega}\phi_2(|\nabla y_n|)\nabla y_n \nabla \varphi_2dx= \int_{\Omega}R_u(x,w_n,y_n)\varphi_1dx+\int_{\Omega}R_u(x,w_n,y_n)\varphi_2dx.
	\end{align*}
\end{lemma}

\begin{proof}
Given $\varepsilon\in(0,1/2)$ and $\varphi_1\in C^{\infty}_0(\Omega)$, we set the function
\begin{align*}
	v_\varepsilon=\dfrac{1}{1-\frac{\varepsilon}{2}}((1-\varepsilon)w_n+\varepsilon \varphi_1).
\end{align*}
Consider $\varphi_2\in V_n$ and apply $(v_\varepsilon,\varepsilon \varphi_2+y_n)$ on the inequality \eqref{2.60}, hence
\begin{align*}
\int_{\Omega} \Phi_1(|\nabla v_\varepsilon|)dx-\int_{\Omega} \Phi_1(|\nabla w_n|)dx&-\varepsilon\int_{\Omega} \phi_2(|\nabla y_n|)\nabla y_n \nabla \varphi_2dx\\
\geq&\int_{\Omega}R_u(x,w_n,y_n)(v_\varepsilon -w_n) dx+ \varepsilon\int_{\Omega}R_v(x,w_n,y_n)\varphi_2 dx,
\end{align*}
and so,
\begin{align*}
\dfrac{\int_{\Omega} \Phi_1(|\nabla v_\varepsilon|)dx-\int_{\Omega} \Phi_1(|\nabla w_n|)dx}{\varepsilon}&-\int_{\Omega} \phi_2(|\nabla y_n|)\nabla y_n \nabla \varphi_2dx\\
\geq\int_{\Omega}&R_u(x,w_n,y_n)\Big(\dfrac{v_\varepsilon -w_n}{\varepsilon}\Big) dx+ \int_{\Omega}R_v(x,w_n,y_n)\varphi_2 dx.
\end{align*}
Note that
\begin{align*}
\dfrac{\varepsilon \varphi_1}{1-\frac{\varepsilon}{2}}=2\left(1-\dfrac{1-\varepsilon}{1-\frac{\varepsilon}{2}}\right )\varphi_1,
\end{align*}
hence, by the convexity of $\Phi_1$,
\begin{align*}
	\Phi_1\Big(\dfrac{1}{1-\frac{\varepsilon}{2}}\big((1-\varepsilon)\nabla w_n+\varepsilon\nabla \varphi_1\big)\Big)\leq \dfrac{1-\varepsilon}{1-\frac{\varepsilon}{2}}\Phi_1(|\nabla w_n|)+\big(1-\dfrac{1-\varepsilon}{1-\frac{\varepsilon}{2}}\big)\Phi_1(2|\nabla \varphi_1|).
\end{align*}
Hence, by Lebesgue dominated convergence theorem, we get
 \begin{align}\label{2.15}
\begin{split}
 \int_{\Omega} \phi_1(|\nabla w_n|)\nabla w_n(\nabla w_n -\frac{\nabla w_n}{2})&dx-\int_{\Omega} \phi_2(|\nabla y_n|)\nabla y_n \nabla\varphi_2dx\\
\geq\int_{\Omega}&R_u(x,w_n,y_n)\Big(\varphi_1-\dfrac{w_n}{2}\Big) dx+ \int_{\Omega}R_v(x,w_n,y_n)\varphi_2 dx.
\end{split}
 \end{align}
Therefore
 \begin{align}
\int_{\Omega} \phi_1(|\nabla w_n|)\nabla w_n\nabla \varphi_1-\int_{\Omega}R_u(x,w_n,y_n)\varphi_1 dx\geq A,\;\;\;\forall \varphi_1\in C^{\infty}_0(\Omega),
\end{align}
where
\begin{align*}
	A=\dfrac{1}{2}\int_{\Omega} \phi_1(|\nabla w_n|)|\nabla w_n|^2dx-\dfrac{1}{2}\int_{\Omega}R_u(x,w_n,y_n)w_n dx.
\end{align*}
As $C^{\infty}_0(\Omega)$ is a vector space, the last inequality gives
 \begin{align}
\int_{\Omega} \phi_1(|\nabla w_n|)\nabla w_n\nabla \varphi_1-\int_{\Omega}R_u(x,w_n,y_n)\varphi_1 dx=0,\;\;\;\forall \varphi_1\in C^{\infty}_0(\Omega).
\end{align}
We know that $W^{1,\Phi_1}_0(\Omega)$ is the weak$^*$ closure of $C^{\infty}_0(\Omega)$ in $W^{1,\Phi_1}( \Omega)$, then using the fact that $\phi_1(|\nabla w_n|)|\nabla w_n|\in L^{\tilde{\Phi }_1}(\Omega)$ we can conclude that
\begin{align}\label{2.17}
\int_{\Omega} \phi_1(|\nabla w_n|)\nabla w_n\nabla \varphi_1 dx-\int_{\Omega}R_u(x,w_n,y_n)\varphi_1 dx=0,\;\;\forall \varphi_1\in W^{1,\Phi_1}_0(\Omega).
\end{align}

Still by \eqref{2.15}, we have
 \begin{align*}
-\int_{\Omega} \phi_2(|\nabla y_n|)\nabla y_n\nabla \varphi_2\geq\int_{\Omega}R_v(x,w_n,y_n)\varphi_2 dx,\;\;\forall \varphi_2\in V_n.
\end{align*}
Since $V_n$ is a vector space, the above inequality gives
 \begin{align}\label{2.16}
\int_{\Omega} \phi_2(|\nabla y_n|)\nabla y_n\nabla \varphi_2=-\int_{\Omega}R_v(x,w_n,y_n)\varphi_2 dx,\;\;\forall \varphi_2\in V_n.
\end{align}
From \eqref{2.17} and \eqref{2.16} 
 \begin{align*}
\int_{\Omega} \phi_2(|\nabla w_n|)\nabla w_n\nabla \varphi_1 dx-\int_{\Omega} \phi_2(|\nabla y_n|)\nabla y_n\nabla \varphi_2=\int_{\Omega}R_u(x,w_n,y_n)\varphi_1 dx+\int_{\Omega}R_v(x,w_n,y_n)\varphi_2 dx,
\end{align*}
for any $(\varphi_1,\varphi_2)\in W^{1,\Phi_2}_0(\Omega)\times V_n.$

\end{proof}

\begin{lemma}
	The sequence $(w_n,y_n)$ is bounded in $X$.
\end{lemma}

\begin{proof}
	Consider the sequence
	\begin{align*}
	g_n(x)=\eta(w_n(x)),\;\;x\in\Omega.
	\end{align*}
where $\eta$ is given in Lemma \ref{4.02}. A direct computation leads to
	\begin{align*}
	\nabla g_n=\left[1-\dfrac{\Phi_1(w_n)}{w_n^2\phi_1(w_n)}\Big(1+\dfrac{w_n\phi_1 '(w_n)}{\phi_1(w_n)}\Big)\right]\nabla w_n.
	\end{align*}
	The last identity together with  $(\phi_{4}')$ implies that
	\begin{align}\label{2.18}
		|\nabla g_n|\leq |\nabla w_n|, \;\;\forall n\in\mathbb{N}
	\end{align}
On the other hand, $(\phi_3')$ also gives
		\begin{align}\label{2.19}
	| g_n(x)|\leq \frac{1}{\ell_1}|\nabla w_n(x)|, \;\;\forall x\in\Omega.
	\end{align}
From \eqref{2.18} and \eqref{2.19}, $g_n\in D(J_{\Phi_1})$ with
	\begin{align*}
	\lVert g_n\lVert_{1}\leq \lVert w_n\lVert_{1},\;\forall n\in \mathbb{N}.
	\end{align*}
By the Lemmas \ref{2.21} and \ref{2.20},
	\begin{align*}
	\begin{split}
	c_n
	=\int_{\Omega}\Phi_1(|\nabla w_n|)dx-&\dfrac{1}{\mu}\int_{\Omega}\phi_1(|\nabla w_n|)|\nabla w_n|^2S(w_n)dx-\int_{\Omega}\Phi_2(|\nabla y_n|)dx+\dfrac{1}{\nu}\int_{\Omega}\phi_2(|\nabla y_n|)|\nabla y_n|^2dx\\
	+\dfrac{1}{\mu}\int_{\Omega}R_u(x,w_n,&v_k)w_nh(w_n)dx+\dfrac{1}{\nu}\int_{\Omega}R_v(x,w_n,y_n)y_ndx-\int_{\Omega}R(x,w_n,y_n)dx,\\
	\end{split}
	\end{align*}
	where $h(t)=\frac{\Phi_1(t)}{t^2\phi_1(t)}$ and $S(t)=1-\frac{\Phi(t)}{t^2\phi_1(t)}\Big[1+\frac{t\phi_1 '(t)}{\phi_1(t)}\Big]$. By $(R_4 ')(i)$ together with $(\phi_3')$,
\begin{align}\label{2.22}
		c_n\geq \int_{\Omega}\Phi_1(|\nabla w_n|)dx-\dfrac{1}{\mu}\int_{\Omega}\phi_1(|\nabla w_n|)|\nabla w_n|^2S(w_n)dx+\left(\dfrac{\ell_2}{\nu}-1\right)\int_{\Omega}\Phi_2(|\nabla y_n|)dx.
		\end{align}
	
On the other hand, the Lemmas \ref{2.21} and \ref{2.20} together with $(R_4 ')(ii)$ imply that
	\begin{align*}
	\begin{split}
	-\beta c_n=&-\beta\int_{\Omega}\Phi_1(|\nabla w_n|)dx+\dfrac{1}{\mu}\int_{\Omega}\phi_1(|\nabla w_n|)|\nabla w_n|^2S(w_n)dx+\beta\int_{\Omega}\Phi_2(|\nabla y_n|)dx\\
	&-\dfrac{1}{\mu}\int_{\Omega}R_u(x,w_n,y_n)w_n h(w_n)dx+\beta\int_{\Omega}R(x,w_n,y_n)dx,\\
	\geq&-\beta\int_{\Omega}\Phi_1(|\nabla w_n|)dx+\dfrac{1}{\mu}\int_{\Omega}\phi_1(|\nabla w_n|)|\nabla w_n|^2S(w_n)dx,
	\end{split}
	\end{align*}
i.e,
\begin{align}\label{2.23}
-\dfrac{1}{\mu}\int_{\Omega}\phi_1(|\nabla w_n|)|\nabla w_n|^2S(w_n)dx\geq \alpha c_n-\alpha\int_{\Omega} \Phi_1(|\nabla w_n|)dx.
\end{align}
From \eqref{2.22} and \eqref{2.23},
	{\small	\begin{align*}
		(1-\beta)c_n\geq (1-\beta)\int_{\Omega}\Phi_1(|\nabla w_n|)dx+\left(\dfrac{\ell_2}{\nu}-1\right)\int_{\Omega}\Phi_2(|\nabla y_n|)dx.
		\end{align*}}
As a consequence of Lemma \ref{lemma21}, we have that $(c_n)$ is bounded. Therefore the sequences $\big( \int_{\Omega}\Phi_1(|\nabla w_n|)dx\big)$ and $\big(\int_{\Omega}\Phi_2(|\nabla y_n|)dx \big )$ are bounded and consequently $(w_n,y_n)$ is bounded at $ W^{1,\Phi_1}_0(\Omega)\times W^{1,\Phi_2}_0(\Omega)$.
	
\end{proof}

Since $(w_n,y_n)$ is bounded, there is no loss of generality in assuming that
\begin{align}\label{2.39}
(w_n,y_n)\xrightharpoonup[\quad]{\ast} (u,v)\;\;\text{ in }W^{1,\Phi_1}_0(\Omega)\times W^{1,\Phi_2}_0(\Omega)\;\text{ as }\; n\rightarrow \infty.
\end{align}
%\begin{equation}\label
%\left\{\;\begin {aligned}
%w_n \xrightharpoonup[\quad]{\ast}u&\;\; \text{weakly\;in } W^{1,\Phi}_0(\Omega)& \\\vartheta_n \xrightharpoonup[\quad]{\ast}v&\;\; \text{weakly\;in } W^{1,\Psi}_0(\Omega)
%\end{aligned}
%\right..
%\end{equation}
By [\citenum{ET}, Theorem $2.1$, Chapter 8], we can conclude that it is worth
\begin{align}\label{2.25}
	\int_{\mathbb{R}^{N}} \Phi_1(|\nabla u|)dx\leq\displaystyle \liminf_{n\rightarrow\infty }\int_{\mathbb{R}^{N}} \Phi_1(|\nabla w_n|)dx
\end{align}
and
\begin{align}\label{2.26}
\int_{\mathbb{R}^{N}} \Phi_2(|\nabla v|)dx\leq\displaystyle \liminf_{n\rightarrow\infty }\int_{\mathbb{R}^{N}} \Phi_2(|\nabla y_n|)dx.
\end{align}

\begin{proposition}
The pair $(u, v)$ is a nontrivial solution of $(S_2)$.
\end{proposition}
\begin{proof}
Fixing  $k,n\in\mathbb{N}$ with $n\geq k$, we have $X_k '\subset X_n '$. Thus, for $(\varphi_1,\varphi_2)\in X_k '$, it follows from Lemma \ref{2.20} that
{\small\begin{align}\label{2.30}
\begin{split}
\int_{\Omega}\phi_1(|\nabla w_n|)\nabla w_n \nabla \varphi_1dx-\int_{\Omega}\phi_2(|\nabla y_n|)\nabla y_n \nabla \varphi_2dx=& \int_{\Omega}R_u(x,w_n,y_n)\varphi_1dx\\&+\int_{\Omega}R_v(x,w_n,y_n)\varphi_2dx,\;\;\forall n\geq k.
\end{split}
\end{align}}
By the above equality together with the convexity of $\Phi_1$, we will obtain
\begin{align}\label{2.28}
\int_{\Omega}\Phi_1(|\nabla \varphi_1|)dx-\int_{\Omega}\Phi_1(|\nabla w_n|)dx\geq \int_{\Omega}R_u(x,w_n,y_n)(\varphi_1-w_n)dx,\;\;\forall \varphi_1\in W^{1,\Phi_1}_0 (\Omega).
\end{align}
From this inequality, we can conclude that
%\begin{align}\label{2.27}
%\int_{\Omega}\Phi(|\nabla w_n|)dx\leq \int_{\Omega}R_u(x,w_n,y_n)w_ndx,\;\;\forall n\geq k.
%\end{align}
%then from \eqref{2.27} and \eqref{2.25}
%\begin{align}
%\int_{\Omega}\Phi(|\nabla u|)dx\leq\liminf_{n\rightarrow\infty }\int_{\Omega}\Phi(|\nabla w_n|)dx\leq\int_{\Omega}R_u(x,u,v)udx,
%\end{align}
%showing that $u\in D(J_\Phi)$. Ainda pela inequality \eqref{2.28},
\begin{align*}
\int_{\Omega} \Phi_1(|\nabla w_n-\frac{1}{n}\nabla w_n|)dx-\int_{\Omega}\Phi_1(|\nabla w_n|)dx
\geq-\dfrac{1}{n}\int_{\Omega}R_u(x,w_n,y_n)w_n dx,
\end{align*}
i.e,
\begin{align*}
\int_{\Omega}\dfrac{\big( \Phi_1(|\nabla w_n-\frac{1}{n}\nabla w_n|)- \Phi_1(|\nabla w_n|)\big)}{-\frac{1}{n}}dx\leq\int_{\Omega}R_u(x,w_n,y_n)w_n dx.
\end{align*}
As $(w_n)$ and $(y_n)$ are bounded in $W^{1,\Phi_1}_0(\Omega)$ and $W^{1,\Phi_2}_0(\Omega)$, respectively, there will be $M>0$ such that
\begin{align*}
\int_{\Omega} \frac{\Phi_1(|\nabla w_n-\frac{1}{n}\nabla w_n|)-\Phi_1(|\nabla w_n|)}{-\frac{1}{n}}dx
\leq M,\;\;\forall n\in\mathbb{N}. 
\end{align*}
Since $\Phi_1$ is in $ C^{1}$ class, there exists $\theta_n(x)\in [0,1]$ such that
\begin{align*}
\dfrac{ \Phi_1(|\nabla w_n-\frac{1}{n}\nabla w_n|)- \Phi_1(|\nabla w_n|)}{-\frac{1}{n}}=\phi_1(|\big(1-\frac{\theta_n(x)}{n}\big)\nabla w_n|)\big(1-\frac{\theta_n(x)}{n}\big)|\nabla w_n|^2.
\end{align*}
Recalling that $0<1-\frac{\theta_n(x)}{n}\leq 1$, we know that $1-\frac{\theta_n(x)}{n}\geq \big(1-\frac{\theta_n(x)}{n}\big)^2$ which leads to
\begin{align*}
\int_{\Omega}\phi_1(|\big(1-\frac{\theta_n(x)}{n}\big)\nabla w_n|)\big(1-\frac{\theta_n(x)}{n}\big)^2|\nabla w_n|^2dx\leq M,\;\;\forall n\in\mathbb{N}. 
\end{align*}
As $\nabla w_n \xrightharpoonup[\quad]{\ast}\nabla u
$ in $\big(L^{\Phi_1}(\Omega)\big)^{N-1}$, we also have $ \big (1-\frac{\theta_n(x)}{n}\big)\nabla w_n \xrightharpoonup[\quad]{\ast}\nabla u$ in $\big(L^{\Phi_1}(\Omega) \big)^{N-1}$ as $n\rightarrow \infty$. Then, by
using the fact that $\phi_1(t)t^2$ is convex, we can apply [\citenum{ET}, Theorem $2.1$, Chapter 8] to get
\begin{align*}
\liminf_{n\rightarrow\infty}\int_{\Omega}\phi_1(|\big(1-\frac{\theta_n(x)}{n}\big)\nabla w_n|)\big(1-\frac{\theta_n(x)}{n}\big)^2|\nabla w_n|^2\geq \int_{\Omega}\phi_1(|\nabla u|)|\nabla u|^2 dx
\end{align*}
and so,
\begin{align*}
\int_{\Omega}\phi_1(|\nabla w_n|)|w_n|^2 dx\leq M.
\end{align*}
Recalling that
\begin{align*}
\phi_1(t)t^2=\Phi_1(t)+\tilde{\Phi }_1(\phi_1(t)t),\;\;\forall t\in\mathbb{R}
\end{align*}
we have

\begin{align*}
\phi_1(|\nabla w_n|)|\nabla w_n|^2=\Phi_1(|\nabla w_n|)+\tilde{\Phi }_1(\phi_1(|\nabla w_n|)|\nabla w_n|)
\end{align*}
which leads to
\begin{align*}
\int_{\Omega}\phi_1(|\nabla w_n|)|\nabla w_n|^2dx=\int_{\Omega}\Phi_1(|\nabla w_n|)dx+\int_{\Omega}\tilde{\Phi }_1(\phi_1(|\nabla w_n|)|\nabla w_n|^2)dx.
\end{align*}
Since $\int_{\Omega}\phi_1(|\nabla u|)|\nabla u|^2dx$ is finite, we see that $\int_{\Omega}\Phi_1(|\nabla u|)dx$ and $\int_{\Omega}\tilde{\Phi }_1(\phi_1(|\nabla u|)|\nabla u|^2)dx$ are also finite, showing that $u\in D(J_{\Phi_1})$ and $u\in dom (\phi_1(t)t)$. Furthermore, it follows from \eqref{2.25} and \eqref{2.28} that
\begin{align}\label{2}
\int_{\Omega}\Phi_1(|\nabla \varphi_1|)dx-\int_{\Omega}\Phi_1(|\nabla u|)dx\geq \int_{\Omega}R_u(x,u,v) (\varphi_1-u)dx,\;\;\forall \varphi_1\in W^{1,\Phi_1}_0 (\Omega).
\end{align}
%
%Além disso, we also have
%\begin{align*}
%\int_{\Omega}\Phi(|\nabla u|)dx-\int_{\Omega}\Phi(|\nabla w_n|)dx\geq \int_{\Omega}R_u(x,w_n,y_n)(u-w_n)dx,\;\;\forall n\geq k.
%\end{align*}
%
%
%
%
%Ainda pela desigualdade \eqref
%{2.28} junto com o limite \eqref{2.25}, podemos concluir que
%\begin{align}\label{221}
%\int_{\Omega}\Phi(|\nabla\varphi_1|)dx-\int_{\Omega}\Phi(|\nabla u|)dx\geq \int_{\Omega}R_u(x,u,v)(\varphi_1-u)dx,\;\;\forall \varphi_1\in W^{1,\Phi}_0 (\Omega).
%\end{align}

%Portanto,
%\begin{align}\label{2.29}
%\int_{\Omega} \Phi(|\nabla u|)dx\geq\displaystyle \limsup_{n\rightarrow\infty }\int_{\Omega} \Phi(|\nabla w_n|)dx.
%\end{align}
%Combining \eqref{2.25} with \eqref{2.29}, we get
%\begin{align}\label{2.41}
%\lim_{n\rightarrow\infty }\int_{\Omega} \Phi(|\nabla w_n|)dx=\int_{\Omega} \Phi(|\nabla u|)dx
%\end{align}
On the other hand, it follows from the equality \eqref{2.30} that
\begin{align}\label{2.32}
\int_{\Omega}\Phi_2(|\nabla \varphi_2|)dx-\int_{\Omega}\Phi_2(|\nabla y_n|)dx\geq- \int_{\Omega}R_v(x,w_n,y_n)(\varphi_2-y_n)dx,\;\;\forall \varphi_2\in V_k.
\end{align}
From this inequality, we can conclude that
%\begin{align}\label{2.27}
%\int_{\Omega}\Phi(|\nabla w_n|)dx\leq \int_{\Omega}R_u(x,w_n,y_n)w_ndx,\;\;\forall n\geq k.
%\end{align}
%then from \eqref{2.27} and \eqref{2.25}
%\begin{align}
%\int_{\Omega}\Phi(|\nabla u|)dx\leq\liminf_{n\rightarrow\infty }\int_{\Omega}\Phi(|\nabla w_n|)dx\leq\int_{\Omega}R_u(x,u,v)udx,
%\end{align}
%showing that $u\in D(J_\Phi)$. Ainda pela inequality \eqref{2.28},
\begin{align*}
\int_{\Omega} \Phi_2(|\nabla y_n-\frac{1}{n}\nabla w_n|)dx-\int_{\Omega}\Phi_2(|\nabla y_n|)dx
\geq\dfrac{1}{n}\int_{\Omega}R_v(x,w_n,y_n)y_n dx,
\end{align*}
i.e,
\begin{align*}
\int_{\Omega}\dfrac{\Phi_2(|\nabla y_n-\frac{1}{n}\nabla y_n|)- \Phi_2(|\nabla y_n|)}{-\frac{1}{n}}dx\leq-\int_{\Omega}R_v(x,w_n,y_n)y_n dx.
\end{align*}
As $(w_n)$ and $(y_n)$ are bounded in $W^{1,\Phi_1}_0(\Omega)$ and $W^{1,\Phi_2}_0(\Omega)$, respectively, there will be $M>0$ such that
\begin{align*}
\int_{\Omega} \frac{\Phi_2(|\nabla y_n-\frac{1}{n}\nabla y_n|)-\Phi_2(|\nabla y_n|)}{-\frac{1}{n}}dx
\leq M,\;\;\forall n\in\mathbb{N}. 
\end{align*}
Since $\Phi_2$ is in $ C^{1}$ class, there exists $\theta_n(x)\in [0,1]$ such that
\begin{align*}
\dfrac{ \Phi_2(|\nabla y_n-\frac{1}{n}\nabla y_n|)- \Phi_2(|\nabla y_n|)}{-\frac{1}{n}}=\phi_2(|\big(1-\frac{\theta_n(x)}{n}\big)\nabla y_n|)\big(1-\frac{\theta_n(x)}{n}\big)|\nabla y_n|^2.
\end{align*}
Recalling that $0<1-\frac{\theta_n(x)}{n}\leq 1$, we know that $1-\frac{\theta_n(x)}{n}\geq \big(1-\frac{\theta_n(x)}{n}\big)^2$ which leads to
\begin{align*}
\int_{\Omega}\phi_2(|\big(1-\frac{\theta_n(x)}{n}\big)\nabla y_n|)\big(1-\frac{\theta_n(x)}{n}\big)^2|\nabla y_n|^2dx\leq M,\;\;\forall n\in\mathbb{N}. 
\end{align*}
As $\nabla y_n \xrightharpoonup[\quad]{\ast}\nabla v
$ in $\big(L^{\Phi_2}(\Omega)\big)^{N-1}$, we also have $ \big (1-\frac{\theta_n(x)}{n}\big)\nabla y_n \xrightharpoonup[\quad]{\ast}\nabla v$ in $\big(L^{\Phi_2}(\Omega) \big)^{N-1}$ as $n\rightarrow \infty$. Then, by
using the fact that $\phi_2(t)t^2$ is convex, we can apply [\citenum{ET}, Theorem $2.1$, Chapter 8] to get
\begin{align*}
\liminf_{n\rightarrow\infty}\int_{\Omega}\phi_2(|\big(1-\frac{\theta_n(x)}{n}\big)\nabla y_n|)\big(1-\frac{\theta_n(x)}{n}\big)^2|\nabla y_n|^2\geq \int_{\Omega}\phi_2(|\nabla v|)|\nabla v|^2 dx
\end{align*}
and so,
\begin{align*}
\int_{\Omega}\phi_2(|\nabla v|)|\nabla v|^2 dx\leq M.
\end{align*}
Recalling that
\begin{align*}
\phi_2(t)t^2=\Phi_2(t)+\tilde{\Phi }_2(\phi_2(t)t),\;\;\forall t\in\mathbb{R}
\end{align*}
we have
which leads to
\begin{align*}
\int_{\Omega}\phi_2(|\nabla v|)|\nabla v|^2dx=\int_{\Omega}\Phi_2(|\nabla v|)dx+\int_{\Omega}\tilde{\Phi }_2(\phi_2(|\nabla v|)|\nabla v|^2)dx.
\end{align*}
Since $\int_{\Omega}\phi_2(|\nabla v|)|\nabla v|^2dx$ is finite, we see that $\int_{\Omega}\Phi_2(|\nabla v|)dx$ and $\int_{\Omega}\tilde{\Phi }_2(\phi_2(|\nabla u|)|\nabla v|^2)dx$ are also finite, showing that $v\in D(J_{\Phi_2})$ and $v\in dom (\phi_2(t)t)$.

% Ademais, segue de \eqref{2.25} and  \eqref{2.28} that
%
%
%
%\newpage
%From this inequality, we can conclude that
%\begin{align}\label{2.31}
%\int_{\Omega}\Psi(|\nabla y_n|)dx\leq- \int_{\Omega}R_v(x,w_n,y_n)y_ndx,\;\;\forall n\geq k.
%\end{align}
%then from \eqref{2.26} and \eqref{2.31},
%\begin{align}
%\int_{\Omega}\Psi(|\nabla v|)dx\leq\liminf_{n\rightarrow\infty }\int_{\Omega}\Psi(|\nabla y_n|)dx\leq-\int_{\Omega}R_v(x,u,v)vdx,
%\end{align}
%showing that $v\in D(J_\Psi)$. Ademais, usando argumentos semelhantes ao que foi apresentado acima, mostra-se que $v\in dom (\psi(t)t)$. 
%%\textcolor{red}{Ainda pela desigualdade \eqref{2.32} junto com o limite \eqref{2.26}, podemos concluir que
%%\begin{align*}
%%\int_{\Omega}\Psi(|\nabla \varphi_2|)dx-\int_{\Omega}\Psi(|\nabla v|)dx\geq -\int_{\Omega}R_v(x,u,v)(\varphi_2-v)dx,\;\;\forall \varphi_2\in V_k.
%%\end{align*}}

Now,  for $\varphi \in W^{1}_0 E^{\Phi_2}(\Omega)$, there exists $\chi_{m} \in V_m$
such that
\begin{align}\label{222}
\lim_{m\rightarrow\infty} \chi_{m}=\varphi\;\;\text{ in } W^{1}_0 E^{\Phi_2}(\Omega).
\end{align}
From \eqref{2.30},
\begin{align*}
-\int_{\Omega}\phi_2(|\nabla y_n|)\nabla y_n \nabla  (\chi_{m}-y_n)dx=\int_{\Omega}R_v(x,w_n,y_n) (\chi_{m}-y_n)dx,\;\;\forall n\geq m.
\end{align*}
The convexity of $\Phi_2$ implies that
\begin{align}\label{61}
\int_{\Omega}\Phi_2(|\nabla \chi_{m}|)dx-\int_{\Omega}\Phi_2(|\nabla y_n|)dx\geq- \int_{\Omega}R_v(x, w_n,y_n)(\chi_{m}-y_n)dx,\;\;\forall n\geq m.
\end{align}
Thus, by the limit \eqref{2.26} we have
\begin{align}
\int_{\Omega}\Phi_2(|\nabla \chi_{m}|)dx-\int_{\Omega}\Phi_2(|\nabla v|)dx\geq- \int_{\Omega}R_v(x, u,v)(\chi_{m}-v)dx.
\end{align}
Now we use \eqref{222} in the above inequality to get
\begin{align}\label{12}
\int_{\Omega}\Phi_2(|\nabla \varphi|)dx-\int_{\Omega}\Phi_2(|\nabla v|)dx\geq- \int_{\Omega}R_v(x,u,v )(\varphi-v)dx.
\end{align}
Repeating the arguments used in Lemma \ref{2.20}, the inequalities \eqref{2} and \eqref{12} imply
\begin{align*}
\int_{\Omega}\phi_1(|\nabla u|)\nabla u\nabla\varphi_1dx= \int_{\Omega}R_u(x,u,v)\varphi_1dx,\;\;\forall \varphi_2\in  W^{1,\Phi_1}_0 (\Omega),
\end{align*}
\begin{align*}
\int_{\Omega}\phi_2(|\nabla v|)\nabla v\nabla \varphi_2dx=- \int_{\Omega}R_v(x,u,v)\varphi_2dx,\;\;\forall \varphi_2\in  W^{1}_0 E^{\Phi_2}(\Omega).
\end{align*}
Finally, the fact that $\phi_2(|\nabla v|)|\nabla v|\in L^{\tilde{\Phi }_2}(\Omega)$ together with the density weak$^*$ of $C ^{\infty}_0( \Omega)$
in $W^{1,\Phi_2}_0(\Omega)$ give
\begin{align*}
\int_{\Omega}\phi_1(|\nabla u|)\nabla u\nabla \varphi_1dx-\int_{\Omega}\phi_2(|\nabla v|)\nabla v\nabla \varphi_2dx= \int_{\Omega}R_u(x,u,v)\varphi_1dx+\int_{\Omega}R_v(x,u,v)\varphi_2dx,
\end{align*}
for every $(\varphi_1,\varphi_2)\in W^{1,\Phi_1}_0 (\Omega)\times W^{1,\Phi_2}_0 (\Omega) $. To conclude, the hypothesis $(R_1 ')$ guarantees that $(u,v)$ is a nontrivial solution for $(S_2)$, and the proof is complete.

\end{proof}

%%%%%%%%%%%%%%%%%%%%%%%%--Apêndice-1--%%%%%%%%%%%%%%%%%%%%%%% 
% 
%
\section*{Appendix}

\subsection*{Basics On Orlicz-Sobolev Spaces}
In this section we recall some properties of Orlicz and Orlicz-Sobolev spaces, which can be found in [\citenum{Adms},\citenum{FN},\citenum{Rao}]. First of all, we recall that a continuous function $\Phi:\mathbb{R}\rightarrow [0,+\infty)$ is a $N$-function if:
\begin{itemize}
	\item[(i)] $\Phi$ is convex;
	\item[(ii)] $\Phi(t)=0\Leftrightarrow t=0$;
	\item[(iii)] $\Phi$ is even;
	\item[(iv)] $\displaystyle\lim_{t\rightarrow 0}\dfrac{\Phi(t)}{t}=0\text{ and }\lim_{t\rightarrow +\infty}\dfrac{\Phi(t)}{t}=+\infty$.
\end{itemize}\vspace*{0.2cm}

We say that a $N$-function $\Phi$ verifies the $\Delta_2$-condition, and we denote by $\Phi\in(\Delta_2)$, if there are constants $K>0,\; t_{0}>0$ such that $$\Phi(2t)\leq K\Phi(t),~~\forall t \geq t_0.$$ In the case of $|\Omega|=+\infty$, we will consider that $\Phi\in(\Delta_2)$ if $t_0=0$. For instance, it can be shown that $\Phi(t)=|t|^p /p$ for $p > 1$ satisfies the $\Delta_{2}$-condition, while $\Phi(t)=(e^{t^2}-1)/2$ does not verify it. 

If $\Omega$ is an open set of $\mathbb{R}^N$, where $N$ can be a natural number such that $N\geq 1$, and $\Phi$ a $N$-function then define the Orlicz space associated with $\Phi$ as $$L^{\Phi}(\Omega)=\left\{u\in L^1_{\text{loc}}(\Omega):~\int_\Omega \Phi\left(\frac{|u|}{\lambda}\right) dx<+\infty~\text{for some}~\lambda>0\right\}.$$ The space $L^{\Phi}(\Omega)$ is a Banach space endowed with the Luxemburg norm given by $$\|u\|_{L^{\Phi}(\Omega)}=\inf\left\{\lambda>0:\int_\Omega \Phi\left(\frac{|u|}{\lambda}\right) dx\leq 1\right\}.$$ In the case that $\Phi$ verifies $\Delta_2$-condition we have $$L^{\Phi}(\Omega)=\left\{u\in L^1_{\text{loc}}(\Omega):~\int_\Omega \Phi(|u|) dx<+\infty\right\}.$$
The complementary function $\tilde{\Phi}$ associated with $\Phi$ is given by the Legendre transformation, that is, $$\tilde{\Phi}(s)=\max_{t\geq 0}\{st-\Phi(t)\},~~\forall\; t\geq 0.$$
The functions $\Phi$ and $\tilde{\Phi}$ are complementary to each other and satisfy the inequality below $$ \tilde{\Phi}(\Phi'(t))\leq \Phi(2t),\;\;\forall\; t>0.$$ Moreover, we also have a Young type inequality given by $$st\leq \Phi(t)+\tilde{\Phi}(s),~~~\forall s,t\geq 0.$$ Using the above inequality, it is possible to establish the following Holder type inequality: $$\left|\int_{\Omega}uvdx\right|\leq 2\|u\|_{L^{\Phi}(\Omega)}\|v\|_{{L^{\tilde{\Phi}}(\Omega)}},~~\text{for all}~~u\in L^{\Phi}(\Omega)~~ \text{and}~~ v\in L^{\tilde{\Phi}}(\Omega).$$

The corresponding Orlicz-Sobolev space is defined by $$W^{1, \Phi}(\Omega)=\left\{ u \in L^ {\Phi}(\Omega): \dfrac{\partial u}{\partial x_{i}} \in L^{\Phi}(\Omega), i=1,...,N \right\} ,$$ with the norm
\begin{align*}
\lVert u\Arrowvert_{1, \Phi} = \Arrowvert \nabla u \Arrowvert_{\Phi} + \Arrowvert u \Arrowvert_{\Phi}.
\end{align*}
The space $W^{1,\Phi}_0 (\Omega)$ is defined as the weak$^*$ clousure of $C^{\infty}_0(\Omega)$ in $W^{1,\Phi} (\Omega)$. Moreover, by the Modular Poincaré's inequality
\begin{align}
	\int_{\Omega} \Phi(|u|/d)dx\leq\int_{\Omega}\Phi(|\nabla u|)dx ,\;\forall u\in W^{1,\Phi}_0 (\Omega),
\end{align}
where $d=2diam(\Omega)$, and it follows that
\begin{align*}
\lVert u\lVert_{\Phi}\leq 2d\lVert \nabla u\lVert_{\Phi},\;\forall u\in W^{1,\Phi}_0 (\Omega).
\end{align*}
The last inequality yields that the functional $\lVert \cdot\lVert:=\lVert \nabla\cdot\lVert_{{\Phi}}$ defines an equivalent norm in $W^{1,\Phi}_0 (\Omega)$. The spaces $L^{\Phi}(\Omega)$, $W^{1,\Phi} (\Omega)$ and $W^{1,\Phi}_0 (\Omega)$ are separable and reflexive, when $\Phi$ and $\tilde{\Phi}$ satisfy $\Delta_{2}$-condition .

If $|\Omega|<\infty$, the space $E^{\Phi}(\Omega)$ denotes the closing of $L^{\infty}(\Omega)$ in $L^{\Phi}( \Omega)$ with respect to the norm $\lVert \cdot\lVert_{\Phi}.$ When $|\Omega|=\infty$, the space $E^{\Phi}(\Omega)$ denotes the closure of $C^{\infty}_{0}(\Omega)$ in $L^{\Phi}(\Omega)$ with respect to norm $\lVert \cdot\lVert_{\Phi}.$ In any of these cases, $L^{\Phi}(\Omega)$ is the dual space of $E^{\tilde{\Phi}}(\Omega)$, while $L^{\tilde{\Phi}}( \Omega)$ is the dual space of $E^{{\Phi}}(\Omega)$.  Moreover, $E^{{\Phi}}(\Omega)$ and $E^{\tilde{\Phi}}(\Omega)$ are separable and all continuous functional $M:E^{{\Phi} }(\Omega)\longrightarrow\mathbb{R}$ are of the form
\begin{align*}
M(v)=\int_{\Omega} v(x)g(x)dx,\;\;\;\;\text{for some function}\;\; g\in L^{\tilde{\Phi}}(\Omega).
\end{align*}
We recall that if $\Phi$ verifies the $\Delta_{2}$-condition, we then have $E^{\Phi}(\Omega)=L^{\Phi}(\Omega)$.

The next result is crucial in the approach explored in Sections 3 and 4, and its proof follows directly from the Banach-Alaoglu-Bourbaki theorem \cite{Brezis}.

\begin{lemma}\label{10.6}
	Assume that $\Phi$ is a $N$-function. If $(u_n)\subset W^{1,\Phi}_0(\Omega)$ is a bounded sequence, then are a subsequence of $(u_n)$, still denoted by itself,
	and  $u\in W^{1,\Phi}_0(\Omega)$ such that
	\begin{align}\label{10.23}
	u_n\xrightharpoonup[\quad]{\ast} u\;\;\;\text{ in }\;L^{\Phi}(\Omega)\;\;\;\;\text{ and }\;\;\;\;	\dfrac{\partial u_n}{\partial x_i}\xrightharpoonup[\quad]{\ast} \dfrac{\partial u}{\partial x_i}\;\;\;\text{ in }\;L^{\Phi}(\Omega)
	\end{align}
and 
	\begin{align*}
	\int_{\Omega}u_nvdx\rightarrow \int_{\Omega}uvdx,\quad
	\int_{\Omega}\dfrac{\partial u_n}{\partial x_i}wdx\rightarrow\int_{\Omega}\dfrac{\partial u}{\partial x_i}wdx,\;\;\forall v, w \in E^{\tilde{\Phi}}(\Omega).
	\end{align*}
\end{lemma}

We denote the limit \eqref{10.23} by $u_n\xrightharpoonup[\quad]{\ast} u$ in $ W^{1,\Phi}_0(\Omega)$.
As an immediate consequence of the last lemma, we have the following corollary.

\section{Acknowledgments} 
The authors are grateful to the Paraíba State Research Foundation (FAPESQ), Brazil, and the Conselho Nacional de Desenvolvimento Científico e Tecnológico (CNPq), Brazil, whose funds partially supported this paper.

\end{document}